\documentclass[12pt,a4paper]{amsart}

 \usepackage[a4paper,left=3cm,right=3cm,top=2.5cm,bottom=2.5cm]{geometry}
\usepackage{amsthm,amsmath}
\usepackage[utf8]{inputenc}
\usepackage{amssymb}
\usepackage[english]{babel}
\usepackage{graphicx}
\usepackage{amsfonts,amssymb}
\usepackage{oldgerm}
\usepackage{mathrsfs}
\usepackage[active]{srcltx}
\usepackage{verbatim}
\usepackage{enumitem, kantlipsum} %label option for enumerate \alph* \Alph* \roman* \Roman* or other... [label=$\bullet$], [label=\alph*]usepackage{enumitem,}
\usepackage{aliascnt}
\usepackage{bbm}
\usepackage{array}
\usepackage{hyperref}
\hypersetup{
  colorlinks = true,
  citecolor = blue,
}
\reversemarginpar
 \usepackage[disable]{todonotes}
\usepackage{xargs}
\usepackage{cellspace}
\usepackage{xr}

\usepackage[Symbolsmallscale]{upgreek}
\usepackage{xcolor}
\definecolor{mydarkblue}{rgb}{0,0.08,0.45}
\usepackage{multicol} 
\usepackage{subcaption} 
\usepackage{wrapfig}

\usepackage{accents}
\usepackage{dsfont}
\usepackage{aliascnt}
\usepackage{cleveref}
\makeatletter
\newtheorem{theorem}{Theorem}
\crefname{theorem}{theorem}{Theorems}
\Crefname{Theorem}{Theorem}{Theorems}

\newaliascnt{lemma}{theorem}
\newtheorem{lemma}[lemma]{Lemma}
\aliascntresetthe{lemma}
\crefname{lemma}{lemma}{lemmas}
\Crefname{Lemma}{Lemma}{Lemmas}

\newaliascnt{corollary}{theorem}
\newtheorem{corollary}[corollary]{Corollary}
\aliascntresetthe{corollary}
\crefname{corollary}{corollary}{corollaries}
\Crefname{Corollary}{Corollary}{Corollaries}

\newaliascnt{proposition}{theorem}
\newtheorem{proposition}[proposition]{Proposition}
\aliascntresetthe{proposition}
\crefname{proposition}{proposition}{propositions}
\Crefname{Proposition}{Proposition}{Propositions}

\newaliascnt{definition}{theorem}

\aliascntresetthe{definition}
\crefname{definition}{definition}{definitions}
\Crefname{Definition}{Definition}{Definitions}

\newaliascnt{remark}{theorem}
\newtheorem{remark}[remark]{Remark}
\aliascntresetthe{remark}
\crefname{remark}{remark}{remarks}
\Crefname{Remark}{Remark}{Remarks}

\crefname{example}{example}{examples}
\Crefname{Example}{Example}{Examples}

\crefname{figure}{figure}{figures}
\Crefname{Figure}{Figure}{Figures}

\newtheorem{assumption}{\textbf{A}\hspace{-3pt}}
\Crefname{assumption}{\textbf{A}\hspace{-3pt}}{\textbf{A}\hspace{-3pt}}
\crefname{assumption}{\textbf{A}}{\textbf{A}}

\Crefname{assumptionG}{\textbf{G}\hspace{-3pt}}{\textbf{G}\hspace{-3pt}}
\crefname{assumptionG}{\textbf{G}}{\textbf{G}}

\Crefname{assumptionQ}{\textbf{Q}\hspace{-3pt}}{\textbf{Q}\hspace{-3pt}}
\crefname{assumptionQ}{\textbf{Q}}{\textbf{Q}}

\newcounter{hypA}
\usepackage{stmaryrd}
\newtheorem{thm}{Theorem}
\newtheorem{prop}[thm]{Proposition}

\usepackage{pgfplots}
\pgfplotsset{width=10cm,compat=1.9}
\graphicspath{{../Figures/}}
\usepackage{xr}
\def\bfn{\mathbf{n}}

\def\msi{\mathsf{I}}
\def\msa{\mathsf{A}}
\def\tmsa{\tilde{\msa}}
\def\msd{\mathsf{D}}
\def\msk{\mathsf{K}}
\def\mss{\mathsf{S}}

\def\msb{\mathsf{B}} 
\def\msc{\mathsf{C}}
\def\mse{\mathsf{E}}

\def\msm{\mathsf{M}}
\def\msu{\mathsf{U}}
\def\msv{\mathsf{V}}

\def\msy{\mathsf{Y}}

\newcommand{\mcb}[1]{\mathcal{B}(#1)}

\def\mcf{\mathcal{F}}

\def\bmcf{\bar{\mathcal{F}}}

\def\mcp{\mathcal{P}}

\def\rset{\mathbb{R}}

\def\zset{\mathbb{Z}}
\def\nset{\mathbb{N}}
\def\nsets{\mathbb{N}^*}

\def\rmb{\mathrm{b}}

\def\rmd{\mathrm{d}}
\def\rmZ{\mathrm{Z}}
\def\mrd{\mathrm{d}}

\def\rme{\mathrm{e}}
\def\rmn{\mathrm{n}}
\def\mrn{\mathrm{n}}
\def\mrc{\mathrm{C}}
\def\mrcc{\mathrm{c}}
\def\rmc{\mathrm{C}}

\def\rmcc{\mathrm{c}}

\newcommand{\po}{\left(}
\newcommand{\pf}{\right)}

\newcommand{\R}{\mathbb R}

\newcommand{\dd}{\mathrm{d}}

\newcommand{\na}{\nabla}
\newcommand{\loiy}{\mu_{\mathrm{v}}}

\def\MeasFspace{\mathbb{M}}

\newcommandx{\functionspace}[2][1=+]{\mathbb{F}_{#1}(#2)}
\newcommandx{\VarDeux}[3][3=]{\operatorname{Var}^{#3}_{#1}\left\{#2 \right\}}

\newcommand{\1}{\mathbbm{1}}

\newcommand{\LeftEqNo}{\let\veqno\@@leqno}

\newcommand{\floor}[1]{\left\lfloor #1 \right\rfloor}

\newcommand{\N}{\ensuremath{\mathbb{N}}}

\newcommand{\PE}{\mathbb{E}}
\newcommand{\PP}{\mathbb{P}}

\newcommand{\abs}[1]{\left\vert #1 \right\vert}

\newcommand{\tvnorm}[1]{\| #1 \|_{\mathrm{TV}}}

\newcommandx{\Vnorm}[2][1=V]{\| #2 \|_{#1}}
\newcommandx{\VnormEq}[2][1=V]{\left\| #2 \right\|_{#1}}
\newcommandx{\norm}[2][1=]{\ifthenelse{\equal{#1}{}}{\left\Vert #2 \right\Vert}{\left\Vert #2 \right\Vert^{#1}}}
\newcommandx{\normLigne}[2][1=]{\ifthenelse{\equal{#1}{}}{\Vert #2 \Vert}{\Vert #2\Vert^{#1}}}

\newcommand{\parenthese}[1]{\left(#1 \right)}
\newcommand{\parentheseLigne}[1]{(#1 )}
\newcommand{\parentheseDeux}[1]{\left[ #1 \right]}
\newcommand{\defEns}[1]{\left\lbrace #1 \right\rbrace }
\newcommand{\defEnsLigne}[1]{\lbrace #1 \rbrace }

\newcommand{\ps}[2]{\left\langle#1,#2 \right\rangle}

\newcommand{\proba}[1]{\mathbb{P}\left( #1 \right)}
\newcommand{\probaCond}[2]{\mathbb{P}\left( \left. #1  \middle\vert #2 \right.\right)}

\newcommandx\probaMarkovTilde[2][2=]
{\ifthenelse{\equal{#2}{}}{{\widetilde{\mathbb{P}}_{#1}}}{\widetilde{\mathbb{P}}_{#1}\left[ #2\right]}}

\newcommand{\expe}[1]{\PE \left[ #1 \right]}

\newcommand{\expeMarkov}[2]{\PE_{#1} \left[ #2 \right]}

\newcommand{\plusinfty}{+\infty}

\def\ie{\textit{i.e.}}
\def\cadlag{càdlàg}
\def\eqsp{\;}
\newcommand{\coint}[1]{\left[#1\right)}
\newcommand{\ocint}[1]{\left(#1\right]}
\newcommand{\ooint}[1]{\left(#1\right)}
\newcommand{\ccint}[1]{\left[#1\right]}
\newcommand{\cointLigne}[1]{[#1)}

\newcommandx{\weight}[2][2=n]{\omega_{#1,#2}^N}

\newcommand{\ball}[2]{\operatorname{B}(#1,#2)}

\def\dist{\operatorname{dist}}

\newcommandx\sequence[3][2=,3=]
{\ifthenelse{\equal{#3}{}}{\ensuremath{\{ #1_{#2}\}}}{\ensuremath{\{ #1_{#2}, \eqsp #2 \in #3 \}}}}
\newcommandx\sequenceD[3][2=,3=]
{\ifthenelse{\equal{#3}{}}{\ensuremath{\{ #1_{#2}\}}}{\ensuremath{( #1)_{ #2 \in #3} }}}

\newcommandx{\sequencen}[2][2=n\in\N]{\ensuremath{\{ #1_n, \eqsp #2 \}}}
\newcommandx\sequenceDouble[4][3=,4=]
{\ifthenelse{\equal{#3}{}}{\ensuremath{\{ (#1_{#3},#2_{#3}) \}}}{\ensuremath{\{  (#1_{#3},#2_{#3}), \eqsp #3 \in #4 \}}}}
\newcommandx{\sequencenDouble}[3][3=n\in\N]{\ensuremath{\{ (#1_{n},#2_{n}), \eqsp #3 \}}}

\newcommand{\wrt}{w.r.t.}

\def\iid{i.i.d.}

\def\eg{e.g.}

\newcommand{\opnorm}[1]{{\left\vert\kern-0.25ex\left\vert\kern-0.25ex\left\vert #1 
    \right\vert\kern-0.25ex\right\vert\kern-0.25ex\right\vert}}

\def\generator{\mathcal{A}}

\def\Id{\operatorname{Id}}

\newcommandx{\CPE}[3][1=]{{\mathbb E}_{#1}\left[#2 \left \vert #3 \right. \right]} %%%% esperance conditionnelle
\newcommandx{\CPVar}[3][1=]{\mathrm{Var}^{#3}_{#1}\left\{ #2 \right\}}
\newcommand{\CPP}[3][]
{\ifthenelse{\equal{#1}{}}{{\mathbb P}\left(\left. #2 \, \right| #3 \right)}{{\mathbb P}_{#1}\left(\left. #2 \, \right | #3 \right)}}

\newcommandx{\osc}[2][1=]{\mathrm{osc}_{#1}(#2)}

\def\Id{\operatorname{Id}}

\def\domain{\mathrm{D}}

\def\martfg{M^{f,g}}
\newcommand\Ddir[1]{D_{#1}}

\def\Refl{\mathrm{R}}

\def\transpose{\operatorname{T}}
\def\v{v}
\def\w{w}

\def\bU{\bar{U}}

\def\lambdab{\bar{\lambda}}

\def\bG{\bar{G}}

\def\S{S}

\def\nE{E}
\def\nF{F}

\def\tmsk{\tilde{\msk}}
\def\tW{\tilde{W}}

\def\yt{\tilde{y}}
\def\ty{\yt}
\def\Mt{\tilde{M}}
\def\tM{\Mt}

\def\tx{\tilde{x}}

\def\tX{\tilde{X}}
\def\tY{\tilde{Y}}
\def\tG{\tilde{G}}
\def\tE{\tilde{E}}
\def\tT{\tilde{T}}
\def\tS{\tilde{S}}

\def\bG{\bar{G}}

\def\bS{\bar{S}}

\def\bH{\bar{H}}

\def\sphere{\mss}

\def\rate{\lambda_{\mathrm{r}}}

\newcommand{\ensemble}[2]{\left\{#1\,:\eqsp #2\right\}}

\def\mrd{\mathrm{D}}
\def\mrc{\mathrm{C}}

\def\lyap{V}
\newcommand\coupling[2]{\Gamma(\mu,\nu)}
\def\supp{\mathrm{supp}}
\def\tpi{\tilde{\pi}}
\newcommand\adh[1]{\overline{#1}}

\renewcommand{\geq}{\geqslant}
\renewcommand{\leq}{\leqslant}
\def\poty{H}

\def\talpha{\tilde{\alpha}}
\def\Leb{\mathrm{Leb}}

\def\iff{ if and only if }

\def\vareps{\varepsilon}
\def\varespilon{\varepsilon}

\def\projd{\operatorname{proj}^{\msd}}
\def\Phibf{\mathbf{\Phi}}

\begin{document}

\title[Geometric ergodicity of the Bouncy Particle Sampler]{Geometric ergodicity of the Bouncy Particle Sampler}
 \author[A. Durmus, A. Guillin, P. Monmarché]{Alain Durmus,  Arnaud Guillin, Pierre Monmarché}

\address{CMLA, ENS Cachan, CNRS, Université Paris-Saclay, 94235 Cachan, France}
\email{alain.durmus@cmla.ens-cachan.fr}

\address{Laboratoire Jacques-Louis Lions and Laboratoire de Chimie Th{\'e}orique, Sorbonne Universit{\'e}}
\email{pierre.monmarche@sorbonne-universite.fr}
\urladdr{https://www.ljll.math.upmc.fr/monmarche/}

\address{Laboratoire de Math\'ematiques Blaise Pascal\\
CNRS - UMR 6620\\
Universit\'e Clermont-Auvergne\\
Avenue des landais,\\
63177 Aubiere cedex, France}
\email{guillin@math.univ-bpclermont.fr}
\urladdr{http://math.univ-bpclermont.fr/$\sim$guillin}

% \subjclass[2010]{Primary }

% \keywords{}

\date{}

\begin{abstract}
  The Bouncy Particle Sampler (BPS) is a Monte Carlo Markov Chain
  algorithm to sample from a target density known up to a
  multiplicative constant. This method is based on a kinetic piecewise
  deterministic Markov process for which the target measure is
  invariant. This paper deals with theoretical properties of
  BPS. First, we establish geometric ergodicity of the associated
  semi-group under weaker conditions than in
  \cite{Doucet2017} both on the target distribution and the velocity
  probability distribution. This result is based on a new coupling of the process
  which gives a quantitative minorization condition and yields more insights on the convergence. In addition, we
  study on a toy model the dependency of the convergence rates on the dimension of
  the state space. Finally, we apply our results to the analysis of simulated annealing algorithms based on BPS. 
  
% This paper deals with theoretical properties of BPS
%  In this paper, we establish $V$-uniform
%   geometric ergodicity of the process.  In particular, our results holds for
%    lighter-tail distribution target with constant refreshment rate.
%   Besides, we provide explicit construction of a coupling of the
%   process, which yields more insights on the convergence.
\end{abstract}

%\end{frontmatter}

\maketitle

\section{Introduction}
Markov chain Monte Carlo methods is a core requirement in many
applications, \eg~in computational statistics
\cite{green:latuszynski:pereyra:robert:2015}, machine learning
\cite{andrieu:defreitas:doucet:jordan:2003}, molecular dynamics
\cite{binder:heermann:reolofs:mallinckrodt:mckay:1993}.   These methods
are used to get approximate samples from a target distribution
denoted $\pi$, with density \wrt~the Lebesgue measure given for all $x
\in \rset^d$ by 
\begin{equation}
  \label{eq:density_pi}
\pi(x) = \exp(-U(x)) \eqsp,
\end{equation}
for a potential $U : \rset^d \to \rset$, known up to an additive constant. 
They rely on the construction of Markov chains which are ergodic with respect to $\pi$, see
\cite{tierney:1994}.

While the first and best-known  MCMC methods
are based on reversible chains, such as many Metropolis-Hastings
type algorithms \cite{metropolis:1953}, there has been since the last
decade an increasing interest in non-reversible discrete-time
processes
\cite{Diaconis2000,BierkensFearnheadRoberts,PetersdeWith,Volte-Face}. %\alain{A
%  discuter, il faut dire pourquoi etc...}. 
Indeed, consider a Markov chains
$(X_k)_{k \in \nset}$ on the state space $\{1,\ldots,n\}$. If $(X_k)_{k \in \nset}$ is reversible, for any $n \in \nset$, the event
$\{X_{n+2} = X_n\}$ has a positive probability,  which explains  why  reversible processes typically used in MCMC show  a diffusive behaviour,
covering a distance $\sqrt{K}$ after $K$ iterations. This makes the
exploration of the space slow and affects the efficiency of the
algorithm.  One of the first attempt to avoid this diffusive behaviour
has been proposed in \cite{Neal2004}, where the author suggests to
modify the transition matrix $\mathbf{M}$ of $(X_k)_{k \in \nset}$, reversible with respect to $\mu$, in such way that the
obtained transition matrix is non-reversible but still leaves $\mu$
invariant. By definition of $\tilde{\mathbf{M}}$, the
probability of backtracking is smaller than for $\mathbf{M}$, \ie~$\tilde{\mathbf{M}}^2_{i,i} \leq \mathbf{M}^2_{i,i}$ for any $i \in \{1,\ldots,n\}$.  In addition, \cite{Neal2004} shows that the asymptotic variance of $\tilde{\mathbf{M}}$ is always smaller than the one of $\mathbf{M}$.

For general state space and in particular in order to sample from $\pi$ defined by \eqref{eq:density_pi}, a now popular idea to construct non-reversible Markov chain is based on lifting, see \cite{Diaconis2000} and the references therein. The idea is to extend the state space $\rset^d$ and consider a Markov chain $(X_k,Y_k)_{k \in \nset}$ on $\rset^d \times \msy$, $\msy \subset \rset^d$,  which admits an invariant
distribution for which the first marginal is the probability measure of interest.
It turns out that, appropriately scaled, some of these lifted chains
converge to continuous-time Markov processes. For instance, the
persistent walk on the discrete torus introduced in
\cite{Diaconis2000} converges to the integrated telegraph on the
continuous torus \cite{Volte-Face}, while the lifted chain defined in
\cite{turitsyn:cherkov:vucelja:2011} for spin models converges to the
Zig-zag process \cite{bierkens:roberts:2017} (see also the event-chain
MC with infinitesimal steps in the physics literature
\cite{michel:kapfer:krauth:2014,PetersdeWith}). In these cases, the
continuous-time limits belong to the class of velocity jump processes
$(X_t,Y_t)_{t \geq 0}$ on $\rset^{d} \times \msy$,
$\msy \subset \rset^d$,  satisfying
$X_t = X_0 + \int_0^t Y_s \rmd s$ for all $t \geq 0$ with
$(Y_t)_{t \geq 0}$ piecewise-constant on random time intervals. The
velocity $(Y_t)_{t \geq 0}$ acts as an instantaneous memory, or
inertia, so that $(X_t)_{t \geq 0}$ tends to continue in the same direction
for some time instead of backtracking. In addition, these processes
may be designed to target a given probability measure defined on
$(\rset^{d} \times \msy,\mcb{\rset^{d} \times \msy})$ of the form
\begin{equation}
  \label{eq:def_tilde_pi}
  \tpi = \pi \otimes \loiy \eqsp,
\end{equation}
where $\loiy$ is a probability measure on $\msy$, and therefore can be used as MCMC samplers.   
%$(\coint{S_n,S_{n+1}})_{n \in \nset}$ for an increasing sequence of
%random variables $(S_n)_{n \geq 0}$ Examples of such velocity jump
%processes are the Zig-Zag process \cite{BierkensFearnheadRoberts}
%where $\mathsf{V}$ is a finite set, the run-\&-tumble process
This kind of dynamics, which are not new \cite{kac:1974,goldstein:1951}, have regained a particular interest in the last decade, in two separate fields:  stochastic algorithms, as we presented, but also biological modelling, where they model the motion of a bacterium \cite{erban:othmer:2004,Calvez,FontbonaGuerinMalrieu2016} and are sometimes called run-\&-tumble processes.

%, and the bouncy particle
%sampler \cite{PetersdeWith,Doucet2015}\alain{va falloir choisir}. All
%these processes are very close (especially in dimension 1), and their
%different names mainly means that the different authors of these works
%were not aware one of the others at that time\alain{c'est un peu agressif !}.

From a numerical point of view, an advantage of these continuous-time
processes is that, under appropriate conditions on the potential $U$,
an exact simulation is possible, following a thinning strategy
\cite{lewis:shedler:1979,Doucet2015,Thieullen2016}. Therefore, no
discretization schemes are needed to approximate the continuous time
trajectory, contrary to Langevin diffusions or Hamiltonian
dynamics. As a consequence, no Metropolis filter is necessary to
preserve the invariance of $\pi$, see
\cite{rossky:doll:friedman:1978,duane:1987,neal:2011,roberts:tweedie:1996} and the reference therein.

This work deals with the velocity jump process introduced in
\cite{PetersdeWith,MonmarcheRTP}. Following \cite{Doucet2015}, we
refer to it as the Bouncy Particle Sampler (BPS). The aim of this
paper is to establish geometric convergence to equilibrium for the
BPS in dimension larger than 1. As detailed below,
we relax the conditions of \cite{Doucet2017}, in particular we show that
any constant refreshment rate is sufficient for thin tail target distributions.
The paper is organized as follows.
\Cref{sec:main-results} presents the BPS process and our main
results, which are proven in
\Cref{sec:proof_main_result}. Finally, \Cref{sec:miscellaneous} is devoted to
a discussion on our result and approach. First, in
\Cref{sec:prec-expl-bound}, we give explicit bound for a toy model,
paying a particular attention to the dependency on the dimension of the
state space in the constants we get. Second, in
\Cref{sec:metast-regime-anne}, we apply our results to study the
annealing algorithm based on the BPS, extending the results
of \cite{MonmarcheRTP}. Some technical proofs are postponed to an Appendix.

Although the work is restricted to the BPS, our arguments
can easily be adapted to other velocity jump processes,
such as randomized variants of the BPS. In particular, the coupling
argument in \Cref{sec:couplage} applies as soon as the process
admits a refreshment mechanism.

\subsection*{Notations}
For all $a,b \in \rset$, we denote $a_+= \max(0,a)$, $a\vee b =
\max(a,b)$, $a \wedge b = \min(a,b)$. $\Id$ stands for the identity
matrix on $\rset^d$. 

For all $x,y \in \rset^d$, the scalar product between $x$ and $y$ is
denoted by $\ps{x}{y}$ and the Euclidean norm of $x$ by $\norm{x}$.
We denote by $\sphere^d = \ensemble{\v \in \rset^d}{\norm{\v} = 1}$,
the $d$-dimensional sphere with radius $1$ and for all $x \in
\rset^d$, $r >0$, by $\ball{x}{r}=\ensemble{\w \in \rset^d}{\norm{w-x}
  \leq r}$ the ball centered in $x$ with radius $r$.  For any
$d$-dimensional matrix $M$, define by $\norm{M} = \sup_{\w \in
  \ball{0}{1}}  \norm{M \w} $ the operator norm associated with
$M$.

% Let $f : \rset^d \to \rset^m$ be a continuously differentiable function. The Jacobian matrix of $f$ is denoted by $\jac(f)$.

Denote by $\mrc(\rset^d)$
 the set of continuous function from $\rset^d$ to $\rset$
and for all $k \in \nsets$, $\mrc^k(\rset^d)$ the set of $k$-times
continuously differentiable function from $\rset^d \to \rset$. Denote for
all $k \in \nset$, $\mrc^{k}_c(\rset^d)$ and $\mrc^{k}_b(\rset^d)$ the set
of functions belonging to $\mrc^k(\rset^d)$ with compact support and
the set of bounded functions belonging to $\mrc^k(\rset^d)$ respectively.
For all function $f : \rset^d \to \rset$, we denote by $\nabla f$ and
$\nabla^2 f$, the gradient and the Hessian of $f$ respectively, if
they exist. For all function $F : \rset^d \to \rset^m$ and compact set
$\msk \subset \rset^d$, denote
$\norm{F}_{\infty} = \sup_{x \in \rset^d} \norm{F(x)}$,
$\norm{F}_{\infty, \msk} = \sup_{x \in \msk} \norm{F(x)}$.  We denote
by $\mcb{\rset^d}$ the Borel $\sigma$-field of and $\mcp(\rset^d)$ the set
of probability measures on $\rset^d$. For $\mu,\nu \in \mcp(\rset^d)$,
$\xi \in \mcp(\rset^d \times \rset^d)$ is called a transference plan between $\mu$
and $\nu$ if for all $\msa \in \mcb{\rset^d}$,
$\xi( \msa \times \rset^d) = \mu(\msa)$ and
$\xi(\rset^d \times \msa) = \nu(\msa)$. The set of transference plan
between $\mu$ and $\nu$ is denoted $\coupling{\mu}{\nu}$.  The random
variables $X$ and $Y$ on $\rset^d$ are a coupling between $\mu$ and $\nu$
if the distribution of $(X,Y)$ belongs to $ \coupling{\mu}{\nu}$.  The
total variation norm between $\mu$ and $\nu$ is defined by
\begin{equation*}
  \tvnorm{\mu-\nu} = 2 \inf_{\xi \in \coupling{\mu}{\nu}} \int_{\rset^d \times \rset^d} \1_{\Delta_\rset^d} (x,y) \, \rmd \xi(x,y) \eqsp, 
                     % & = \sup \ensemble{\abs{\int_{\rset^d} f \rmd \mu - \int_{\rset^d} f\rmd \nu }} {\text{$f :\msm \to \rset$ is Borel measurable and $\sup_{\msm} f < 1$}} \eqsp.
\end{equation*}
where $\Delta_{\rset^d} = \ensemble{(x,y) \in \rset^d \times \rset^d}{x=y}$. For
$\lyap : \rset^d \to \coint{1,\plusinfty }$, define the
$\lyap$-norm between $\mu$ and $\nu$ by
\begin{equation*}
  \Vnorm[\lyap]{\mu-\nu} =\sup \ensemble{\abs{\int_{\rset^d} f \rmd \mu - \int_{\rset^d} f\rmd \nu }} {\text{$f : \rset^d\to \rset$, \, $  \norm{f/\lyap}_{\infty} < 1$}} \eqsp. 
\end{equation*}
When $\lyap(x) = 1$ for all $x \in \rset^d$, the $\lyap$-norm is simply the total variation norm.
For all $\mu \in \mcp(\rset^d)$, define the support of $\mu$ by
\begin{equation*}
  \supp \, \mu = \adh{\ensemble{x \in \rset^d}{ \text{ for all open set } \msu \ni x,\,  \mu(\msu) >0}} \eqsp.
\end{equation*}

In the sequel, we take the convention that $\inf \emptyset = \plusinfty$. 

%%% Local Variables:
%%% mode: latex
%%% TeX-master: "main"
%%% End:

\section{Geometric convergence of the BPS}
\label{sec:main-results}
%\emph{Présentation du BPS permière version et non-explosion et résultat principal d'ergodicité.}

%\bigskip

\subsection{Presentation of the BPS}
\label{sec:presentation-bps}
In all this work, we assume that the potential $U$, given by
\eqref{eq:density_pi}, is continuously differentiable on $\rset^d$.
Let $\msy \subset\rset^d$ be a closed $\mrc^{\infty}$-submanifold $\msy \subset \rset^d$, which is rotation invariant, \ie~for any rotation $O \in \rset^{d \times d}$, $O \msy=\msy$. 
The BPS process $(X_t,Y_t)_{ t\geq 0}$ associated with $U$ evolves on
$(\rset^d \times \msy, \mcb{\rset^d \times \msy})$ and is
defined as follows.

Consider some initial point $(x,y) \in \rset^{d} \times \msy$, and a
family of \iid~random variables $(E_i, F_i, G_i)_{i \in \nset^*}$
on the same probability space $(\Omega,\mcf,\mathbb{P})$, where for
all $i \in \nset^*$, $E_i, F_i$ are exponential random variables
with parameter $1$, $G_i$ is a random variable with a given
distribution $\loiy$ on $(\msy,\mcb{\msy})$, referred to as the
refreshment distribution. In addition, for all $i \in \nset^*$, $E_i$, $F_i$ and
$G_i$ are independent. Let $\rate >0$, referred to as the refreshment rate, $(X_0, Y_0) = (x,y)$ and $S_0=0$. We define by recursion the jump times of the process and
the process itself. Assume that $S_n$ and $(X_t,Y_t)_{t \leq S_n}$ have been defined for $n \geq 0$. Consider 
\begin{eqnarray}
T_{n+1}^{(1)} &=&   E_{n+1}/ \rate \notag\\
\label{eq:def-temps-bounce}
T_{n+1}^{(2)} &=& \inf\ensemble{t \geq 0}{\int_0^t \ps{Y_{\S_n}}{\nabla U (X_{\S_n}+s Y_{\S_n})}_+  \rmd s \geq F_{n+1}} \\
T_{n+1} &=& T_{n+1}^{(1)} \wedge T_{n+1}^{(2)} .\notag
\end{eqnarray}
Set $S_{n+1} = S_n + T_{n+1}$, $(X_t,Y_t) = ( X_{\S_n} + t Y_{\S_n},Y_{\S_n})$,  for all $t \in \cointLigne{S_n,S_{n+1}}$,
$X_{S_{n+1}} = X_{\S_n} + T_{n+1} Y_{\S_n} $ and 
\begin{equation*}
 Y_{\S_{n+1}} = 
  \begin{cases}
     G_{n+1} &  \text{ if $T_{n+1} = T_{n+1}^{(1)}$} \\
 \Refl(X_{\S_{n+1}} , Y_{\S_n}) & \text{ otherwise} \eqsp,
  \end{cases}
\end{equation*}
where $\Refl :
\rset^{2d} \to \rset^d$ is the function given for all $x,y \in
\rset^d$ by
\begin{equation}
  \label{eq:definition_Refl}
  \begin{aligned}
  \Refl(x,y) &=   y - 2\ps{y}{\rmn(\nabla U(x))} \rmn(\nabla U(x)) \eqsp,
  \\
  \, \text{ where for all $z \in\rset^d$} \, \eqsp, \, \rmn(z) & = 
\begin{cases}
z/\norm{z} & \text{ if $z \not =0$} \\
0 & \text{ otherwise } \eqsp.
\end{cases}
\end{aligned}
\end{equation}
Note that for all $(x,y) \in \rset^{2d}$ with $\nabla U(x) \not = 0$,
$\Refl(x,y)$ is the reflection of $y$ orthogonal to $\nabla U(x)$ and
therefore for all $(x,y) \in \rset^{2d}$,
$\norm{\Refl(x,y)} = \norm{y}$.

If $T_{n+1} = T_{n+1}^{(1)}$, we say that, at time $T_{n+1}$, the velocity has been refreshed, and we call $T_{n+1}$ a refreshment time. If $T_{n+1} = T_{n+1}^{(2)}$, we say that, at time $T_{n+1}$, the process has bounced, and we call $T_{n+1}$ a bounce time.

Then, $(X_t,Y_t)$  is defined for all $t < \sup_{n \in \nset} \S_n$ and we
set for all $t \geq \sup_{n \in \nset} \S_n$, $(X_t,Y_t) =
\infty$, where $\infty$ is a cemetery point.

%\begin{lemma}[\Cref{prop:non-explosion_BPS}]
%  \label{lem:definition_true}
%  Almost surely, $\sup_{n \in \nset} S_n =
%  \plusinfty$. Therefore almost surely $(X_t,Y_t)_{t \geq 0}$ is a $(\rset^{d} \times \msy)$-valued \cadlag~process.
%\end{lemma}
%\begin{proof}
%The proof is postponed to \Cref{sec:proof-crefl}.
%\end{proof}

In fact, it is proven in \cite[\Cref{prop:non-explosion_BPS}]{DurmusGuillinMonmarche:toolbox} that almost surely, $\sup_{n \in \nset} S_n =
  \plusinfty$. Therefore, almost surely, $(X_t,Y_t)_{t \geq 0}$ is a $(\rset^{d} \times \msy)$-valued \cadlag~process. By \cite[Theorem
25.5]{davis:1993}, the BPS process $(X_t,Y_t)_{t \geq 0}$ defines
a strong Markov semi-group $(P_t)_{t\geqslant 0}$ given for all
$(x,y)\in \rset^d \times \msy$ and $\msa \in \mcb{\rset^d \times \msy}$ by
\begin{equation*}
  %\label{EqDefP_t}
P_t ((x,y),\msa)  =  \proba{(X_t,Y_t) \in \msa} \eqsp,
\end{equation*}
where $\sequenceD{X_t,Y_t}[t][ \rset_+]$ is the BPS process started
from $(x,y)$. 

% Recall that $\nu$ is called invariant for $(P_t)_{t\geqslant 0}$ if $\nu P_t = \nu$ for all $t\geqslant0$. 
% It is explained in \cite{PetersdeWith,MonmarcheRTP} that the BPS process is defined in order for its equilibrium to be the distribution $\tilde{\pi}$ defined by \eqref{eq:def_tilde_pi}. We will prove this claim under the following assumption:
Consider the following basic assumption.
\begin{assumption}
  \label{assum:hyp_base}
  The potential $U$ is twice continuously differentiable, $\loiy$ is rotation invariant
  and $(x,y) \mapsto \norm{y}\norm{\nabla U(x)}$ is integrable with respect
  to $\tilde{\pi}$ defined by \eqref{eq:def_tilde_pi}.
\end{assumption}
It is shown in \cite[\Cref{prop:invarince-BPS}]{DurmusGuillinMonmarche:toolbox}, and contrary to the popular belief it
is quite technical and difficult, that under \Cref{assum:hyp_base}, the probability measure $\tpi$ defined by \eqref{eq:def_tilde_pi}  is invariant for $(P_t)_{t \geq 0}$, \ie~$\tpi P_t = \tpi$ for all $t \geq 0$.
%\begin{proposition}
%  \label{propo:invariance}
%Under \Cref{assum:hyp_base}, the law $\tpi$, defined by \eqref{eq:def_tilde_pi}, is invariant for $(P_t)_{t \geq 0}$.
%\end{proposition}
%\begin{proof}
%  The proof is postponed to \Cref{sec:proof_invariant}.
%\end{proof}

%%% Local Variables:
%%% mode: latex
%%% TeX-master: "main"
%%% End:

\subsection{Main results}
\label{sec:main-results}

For $\lyap : \rset^{d} \times \msy \to \coint{1,\plusinfty }$,  the semi-group $(P_t)_{t \geq 0}$ with invariant measure $\tpi$ is said to be $\lyap$-uniformly
  geometrically ergodic if there exist $C,\rho>0$ such that for all
  $t\geq 0$ and all $\mu \in \mcp(\rset^d \times \msy)$ with $\mu(\lyap) < \plusinfty$, it holds
\begin{equation}
\label{eq:def_unif_V_geo_ergo}
  \Vnorm[\lyap]{ \mu P_t - \tpi }  \leq C \rme^{-\rho t} \mu(\lyap) \eqsp.
\end{equation}
We state in this section our main results regarding the $\lyap$-uniform geometric ergodicity of the BPS. 
%To derive our results regarding the geometric ergodicity of the
%BPS process, we consider several different conditions on $U$. 

Our basic assumptions to prove geometric ergodicity are the following.
\begin{assumption}
  \label{ass:geo_ergo_1}
  \begin{enumerate}[label=(\roman*)]
  \item   The potential $U$ is positive and satisfies $\int_{\rset^d} \exp\parenthese{  - U(x)/2} \rmd x  < \plusinfty $ and $\lim_{\norm{x} \to \plusinfty} U(x) = \plusinfty$.
  \item   \label{ass:geo_ergo_1_loiy} $\loiy$ admits a density \wrt~the Lebesgue measure on
    $\rset^d$ or there exists $r_0 >0$ such that
    $\loiy(r_0 \sphere^d) >0$.
  \end{enumerate}
  % \begin{equation*}
  %    \eqsp,
  % \end{equation*}
  %   \begin{equation*}
  %     \limsup_{\norm{x} \to \plusinfty} \defEns{ \norm{\nabla U(x)}/\, U(x) + \norm{\nabla^2 U(x)}/\norm{\nabla U(x)}} < \plusinfty \eqsp.
  %   \end{equation*}
  \end{assumption}
Here, we establish practical conditions on the potential $U$, $\loiy$ and $\msy$ implying that $(P_t)_{t \geq 0}$ is 
$\lyap$-uniformly geometrically ergodicity. In fact, these conditions are derived from a more general result. However, since its assumptions and  statement may seem very intricate, for the sake of clarity we have decided to give this result after its corollaries.

Consider the following alternative conditions, which will be used in the case where $\msy$ is bounded. 
 \begin{assumption}
    \label{ass:geo_ergo_1bis}
 The potential $U$ satisfies 
 \[\lim_{\norm{x} \to \plusinfty} \norm{\nabla U(x)} = \infty\eqsp,\quad \sup_{x\in\rset^2} \norm{\nabla^2U(x)} <\infty\eqsp.\]
  \end{assumption}

  \begin{assumption}
    \label{ass:geo_ergo_2}
%One of these two conditions holds.
%\begin{enumerate}[label=(\alph*)]
 % \item  \label{ass:2_b} $\lim_{\norm{x} \to \plusinfty} \norm{\nabla U(x)} = \plusinfty$ and $\sup_{x \in \rset^d} \norm{\nabla^2 U(x)} < \plusinfty$.
  %  \item    \label{ass:2_a}
 There exists $\varsigma \in \ooint{0,1}$ such that
    \begin{equation*}
      \liminf_{\norm{x} \to \plusinfty} \defEns{ \norm{\nabla U(x)}/\, U^{1-\varsigma}(x) } >0 \eqsp, \, \limsup_{\norm{x} \to \plusinfty} \defEns{\norm{\nabla U(x)}/\, U^{1-\varsigma/2}(x)} < \plusinfty \eqsp,
    \end{equation*}
    \begin{equation*}
       \limsup_{\norm{x} \to \plusinfty} \defEns{\norm{\nabla^2 U(x)}/\, U^{1-\varsigma}(x)} < \plusinfty \eqsp.
    \end{equation*}
   % \end{enumerate}
  \end{assumption}

%In the following, we derive other conditions on the potential than the
%one in \Cref{ass:geo_ergo_2} which implies geometric
%convergence of the BPS process.
\begin{assumption}
 \label{ass:geo_ergo_2_b}
 The potential $U$ satisfies $    \lim_{\norm{x}\to \plusinfty} \norm{\nabla^2 U(x)}/\norm{\nabla U(x)} = 0 $ and there exists $\varsigma \in \ooint{0,1}$ such that 
 \begin{equation*}
\liminf_{\norm{x}\to \plusinfty } \norm{\nabla U(x)}/ U^{1-\varsigma}(x) >0   \text{ and }  \, \, \lim_{\norm{x} \to \plusinfty} \norm{\nabla U(x)}/ U^{2(1-\varsigma)}(x) = 0 \eqsp.
 \end{equation*}
\end{assumption}
Note that \Cref{ass:geo_ergo_2_b} is similar to \Cref{ass:geo_ergo_2} but these two conditions are different: none of them implies the other. Indeed, on $\rset^2$, consider $U(x_1,x_2)= (1+|x_1|^2)^{\alpha/2} + (1+|x_2|^2)^{\beta/2}$ for some $\alpha,\beta>1$. Then for all $(x_1,x_2) \in \rset^2$, we have
\begin{align*}
  \nabla U(x) &= [\alpha x_1 (1+x_1^2)^{\alpha/2-1}, \beta x_2 (1+x_2^2)^{\beta/2-1}]^{\transpose}  \\
  \nabla^2 U(x) & = \begin{pmatrix}
F(\alpha,x_1)& 0 \\
    0 & F(\beta,x_2)
  \end{pmatrix} \\
 \text{where } F(\alpha,x_1)  & =     \alpha (1+x_1^2)^{\alpha/2-1} + 2 \alpha x_1^2 (\alpha/2-1)  (1+x_1^2)^{\alpha/2-2}  \eqsp. 
\end{align*}
In that case \Cref{ass:geo_ergo_2}  is satisfied \iff $\ccint{(\alpha \vee \beta)/2,  \alpha\wedge  \beta} \neq \emptyset $, while \Cref{ass:geo_ergo_2_b}  is satisfied \iff $\ccint{2 (\alpha\vee \beta)/(1+\alpha\vee \beta),\alpha \wedge \beta} \neq \emptyset $, chosing in both cases $\varsigma^{-1}>1$ in the corresponding interval. In particular, if both $\alpha,\beta\geqslant 2$, then \Cref{ass:geo_ergo_2_b}  is satisfied, but \Cref{ass:geo_ergo_2} may not (if $\alpha>2\beta$ for instance). On the contrary if, say, $\alpha = 4/3$ and $\beta \in (1,8/7)$, then \Cref{ass:geo_ergo_2}   holds while \Cref{ass:geo_ergo_2_b}  does not.

\begin{theorem}
  \label{theo:V_geo_ergo_loi_borne}
  Assume \Cref{assum:hyp_base}, \Cref{ass:geo_ergo_1}, $\msy$ is bounded  and either  \Cref{ass:geo_ergo_1bis}, \Cref{ass:geo_ergo_2}  or \Cref{ass:geo_ergo_2_b}.
  In the case where \Cref{ass:geo_ergo_1bis} holds, set $\varsigma = 1$. Then,  for any refreshment rate $\rate >0$, there exists $\kappa\in(0,1]$ such that $(P_t)_{t \geq 0}$ is $\lyap$-uniformly geometrically ergodic with $\lyap : \rset^d \times \msy \to \coint{1,\plusinfty}$ given for all $(x,y) \in \rset^d \times \msy$ by $\lyap(x,y) = \exp\parenthese{\kappa U^{\varsigma}(x)} $.
\end{theorem}
\begin{proof}
The proof is postponed to \Cref{sec:autre-demo_1}.
\end{proof}
%The geometric ergodicity of BPS was also shown in  \cite[Theorem 3.1]{Doucet2017} under \Cref{ass:geo_ergo_1bis} and with the condition that $\rate$ is sufficiently large. Note that we do not impose this last condition in \Cref{theo:V_geo_ergo_loi_borne}. 

Note that \Cref{ass:geo_ergo_1bis}, \Cref{ass:geo_ergo_2} and \Cref{ass:geo_ergo_2_b} all require that $\lim_{\norm{x} \to \plusinfty} \norm{\nabla U(x)} = \plusinfty$. We consider now the case where $\liminf_{\norm{x} \to \plusinfty} \norm{\nabla U(x)} < \plusinfty$ possibly.

\begin{assumption}
    \label{ass:geo_ergo_5}
    The potential $U$ satisfies
    \begin{equation*}
      \text{$\liminf_{\norm{x} \to \plusinfty} \norm{\nabla U(x)} >0$ \, and \,  $\lim_{\norm{x} \to \plusinfty} \norm{\nabla^2U(x)} = 0$}\eqsp.
    \end{equation*}
%     \begin{eqnarray*}
%     \underset{\norm{x}\rightarrow\infty}\liminf \norm{\nabla U(x)} & > & 0 \eqsp,\\
%      \underset{\norm{x}\rightarrow\infty}\lim \norm{\nabla^2 U(x)} & = & 0 \eqsp.
% \end{eqnarray*}         
\end{assumption}
%The following result is a generalization of \cite[Theorem 3.1]{Doucet2017} which considers $\msy = \sphere^d$ and $\loiy$ is the uniform distribution on $\sphere^d$.
\begin{theorem}
  \label{theo:V_geo_ergo_lambda_0}
  Assume \Cref{assum:hyp_base}, \Cref{ass:geo_ergo_1}, \Cref{ass:geo_ergo_5} and $\msy$ is bounded. Then, there exists $\lambda_0>0$ such that, if $\rate \in \ocint{0, \lambda_0}$, $(P_t)_{t \geq 0}$ is $\lyap$-uniformly geometrically ergodic with $\lyap : \rset^d \times \msy \to \coint{1,\plusinfty}$ given for all $(x,y) \in \rset^d \times \msy$ by $\lyap(x,y) = \exp(\kappa U(x))$, for $\kappa \in \ocint{0,1}$.
  % \[\forall (x,y)\in \rset^d\times \msy, \eqsp \lyap(x,y)  = e^{\frac12 U(x)} \eqsp .\]
\end{theorem}
\begin{proof}
The proof is postponed to \Cref{sec:proof_lambda_0}. 
\end{proof}
Note that contrary to the setting of \Cref{theo:V_geo_ergo_loi_borne}, the result of \Cref{theo:V_geo_ergo_lambda_0} requires that the refreshment rate $\rate$ is sufficiently small for the BPS to be $\lyap$-uniformly geometrically ergodic.  

We now turn to the case where $\msy$ is unbounded. Indeed, this case is interesting from the numerical experiments conducted in \cite[Section 4.3]{Doucet2015} which shows that the choice of $\msy = \rset^{d}$   and $\loiy$ being the $d$-dimensional Gaussian distribution appears to be better and less sensitive to the choice of the refreshment rate $\rate$ compared to $\msy = \sphere^{d}$ and the uniform distribution on this set.
  
In the case where $\msy$ is unbounded,   \Cref{ass:geo_ergo_2} must be strengthen as follow.
\begin{assumption}
  \label{ass:geo_ergo_4}
 There exists $\varsigma \in \ooint{0,1}$ such that
 \begin{align*}
    \liminf_{\norm{x} \to \plusinfty} \defEns{ \norm{\nabla U(x)}/\, U^{1-\varsigma}(x) }&>0
    \\
    \limsup_{\norm{x} \to \plusinfty} \defEns{\norm{\nabla U(x)}/\, U^{1-\varsigma}(x)}  &<  \plusinfty 
 \\
 \limsup_{\norm{x} \to \plusinfty} \defEns{\norm{\nabla^2 U(x)}/\, U^{1-2\varsigma}(x)} &< \plusinfty \eqsp.
    \end{align*}     
\end{assumption}
   \Cref{ass:geo_ergo_4} (and therefore \Cref{ass:geo_ergo_2})  holds when $U$ is a perturbation of an
  $\alpha$-homogeneous function:
\begin{proposition}
  \label{propo:alpha_homogeneous}
  Let $\alpha \in \ooint{1,\plusinfty}$ and assume that $U = U_1 + U_2$ with  $U_1,U_2 \in \mrc^2(\rset^d)$ satisfying
  \begin{enumerate}[label=$\bullet$,wide, labelwidth=!, labelindent=0pt]
  \item $U_1$ is $\alpha$-homogeneous: for all $t \geqslant 1$ and $x \in \rset^d$ with $\norm{x}\geqslant 1$, $$      U_1(tx) = t^{\alpha}U_1(x)  \text{ and }       \lim_{\norm{x} \to \plusinfty} U_1(x) = \plusinfty \eqsp.$$
  \item
    \begin{equation*}
      \limsup_{\norm{x} \to \plusinfty} \defEns{U_2(x)/\norm[\alpha]{x} + \norm{\nabla U_2(x)}/\norm[\alpha-1]{x} + \norm{\nabla^2 U_2(x)}/\norm[\alpha-2]{x} } = 0 \eqsp. 
    \end{equation*}
    % \begin{equation*}
    %   U_1(x) = t^{\alpha} \norm{x} \eqsp.
    % \end{equation*}
  \end{enumerate}
  %\begin{enumerate}
  %\item If $ \alpha >1$, 
  Then \Cref{ass:geo_ergo_4}  holds with $\varsigma = 1/\alpha$.
  %\item If $\alpha =1$ , then \Cref{ass:geo_ergo_1} and \Cref{ass:geo_ergo_3} hold.
  %\end{enumerate}
\end{proposition}
  The proof is postponed to  Appendix A.

This class of potentials is considered in \cite[Theorem 4.6]{jarner:hansen:2000}, which shows that the Random Walk Metropolis algorithm is geometrically ergodic for target distributions $\pi$ associated to a potential belonging to this class. 

\begin{theorem}
  \label{theo:V_geo_ergo_loi_non_borne}
  Assume \Cref{assum:hyp_base}, \Cref{ass:geo_ergo_1} , \Cref{ass:geo_ergo_4} and $\loiy$ admits a Gaussian moment: there exists $\eta>0$ such that
  $\int_{\msy} \rme^{\eta \norm{y}^2}  \loiy(\rmd y)  <  \plusinfty$.
  Then,  for any refreshment rate $\rate >0$, there exists $\kappa\in(0,1]$ such that $(P_t)_{t \geq 0}$ is $\lyap$-uniformly geometrically ergodic with $\lyap : \rset^d \times \msy \to \coint{1,\plusinfty}$ given for all $(x,y) \in \rset^d \times \msy$ by $\lyap(x,y) = \exp\parenthese{\kappa U^{\varsigma}(x)} + \exp\parentheseLigne{\eta \norm[2]{y}}$.
%  \[\forall (x,y)\in \rset^d\times \msy, \eqsp \lyap(x,y)  = e^{U^{\varsigma}(x)} + e^{\delta \norm[2]{y}} \eqsp .\]
%  geometrically ergodic
%  \Cref{assum:hyp_base}-\Cref{ass:geo_ergo_1} and $\loiy$ is a non-degenerate
%  Gaussian measure on $\rset^d$.  Then there exists
%  $\lyap : \rset^{d} \times \rset^d \to \coint{1,\plusinfty}$ satisfying
%  there exists $\beta >0$ such that 
%  \begin{align*}
%   \liminf_{\norm{x} +\norm{y} \to \plusinfty}  \defEns{V(x,y)  \exp\parenthese{-\beta (U^{\varsigma}(x)+\norm[2]{y})}}  & >0\\
%      \limsup_{\norm{x}+\norm{y}  \to \plusinfty}  \defEns{V(x,y)  \exp\parenthese{- \beta(U^{\varsigma}(x)+\norm[2]{y})} } &< \plusinfty \eqsp. \end{align*}                                                                                                                              \begin{enumerate}
%\item If \Cref{ass:geo_ergo_4} holds, for all jump rate $\rate$, $(P_t)_{t \geq 0}$ is $\lyap$-uniformly
%  geometrically ergodic.
%\item If \Cref{ass:geo_ergo_3}  holds then there exists $\rate_0$ such that for all $\rate \leq \rate_0$, $(P_t)_{t \geq 0}$ is $\lyap$-uniformly
%  geometrically ergodic.
%  \end{enumerate}
\end{theorem}
\begin{proof}
The proof is postponed to \Cref{sec:proof-crefth}.
\end{proof}
%  \alain{dire que c'est les hypo de doucet et comparer, on a pas de condition sur le taux}
%For other kind of potentials $U$, we have the following results.

% \begin{example}[Bayesian logistic regression]
%   In a Bayesian logistic regression setting (see \eg~\cite[Chapter 16]{gelman2013bayesian}), given $n\in\nset^*$ observations $(y_i)_{i=1}^n$ and the associated covariates $\{x_{i,j} \, : \, j \in \{1,\ldots,d\}\}_{i=1}^{n}$ the likelihood function is given for any $\bfbeta \in \rset^d$ by
%   \begin{equation*}
%     \bfbeta \mapsto 
%   \end{equation*}
%   We can apply our result to this particular scenario. If for $t$ large enough, $g(t) = C \abs{t}^{\varpi}$ for $C \geq 0$ and $\varpi >0$, then ...
%   If for $t$ large enough, $g(t) = C \abs{t}$ for $C \geq 0$, then ... 
% \end{example}

We now  compare  our results to the ones established by
  \cite{Doucet2017}.  First, their results deal only with the case
  where $\msy = \sphere^d$ and $\loiy$ is the uniform distribution on
  $\sphere^d$, while our work can be applied to much broader cases.
  We discuss in the following our main contributions compared to
  \cite{Doucet2017} in the case where $\msy$ is bounded. 
  The basic assumptions of  \cite{Doucet2017} are the following: (i) $\nabla^2 U$ is locally Lipschitz; (ii)
  $\int_{\rset^d}\norm{\nabla U(x)} \rmd \pi(x) < \plusinfty$; (iii)
  $\liminf_{\norm{x} \to \plusinfty} \{\rme^{U(x)/2} /$
  $ \norm{\nabla U(x)}^{1/2} \}>0$;
  \[
 \text{(iv)} \qquad   \inf_{(x,v) \in \rset^d \times \sphere^d} \frac{\rme^{U(x)/2}}{\{\ps{\nabla U(x)}{v}_+\Lambda_{\mathrm{ref}}\}^{1/2}} > 0  \eqsp, 
\]
where $\Lambda_{\mathrm{ref}} : \rset^d\to \rset_+$ is a function chosen in the results. These conditions are similar to \Cref{assum:hyp_base} and \Cref{ass:geo_ergo_1} in our work. We now give the results obtained by  \cite{Doucet2017} in detail in order to highlight the differences with the present work.  Apart from the CLT which is a consequence of the others, there are three main results in \cite{Doucet2017} for the geometric ergodicity of the BPS. The first one, concerning regular tail distributions (\cite[Theorem 3.1]{Doucet2017}), establishes that the BPS process as defined at the beginning of \Cref{sec:presentation-bps}  is $V$-geometrically ergodic if $\Lambda_{\mathrm{ref}} = \rate$ and one of the following conditions holds:
\begin{enumerate}[label=(\Alph*)]
\item \label{item_condition_doucet_A} $\liminf_{\norm{x} \to \plusinfty} \norm{\nabla U(x)} = \plusinfty$, $\limsup_{\norm{x}\to \plusinfty} \norm{\nabla^2 U(x)} < \plusinfty$ and $\rate >C_1$ for some constant $C_1>0$.
\item \label{item_condition_doucet_B}  $\liminf_{\norm{x} \to \plusinfty} \norm{\nabla U(x)} >0$, $\lim_{\norm{x}\to \plusinfty} \norm{\nabla^2 U(x)} = 0$\footnote{In the statement of the Theorem, the authors claim that $\limsup_{\norm{x}\to \plusinfty} \norm{\nabla^2 U(x)} < \plusinfty$ but a careful reading of the proof shows that $\lim_{\norm{x}\to \plusinfty} \norm{\nabla^2 U(x)} = 0$ is necessary.} and $\rate < C_2$ for some constant $C_2>0$.
\end{enumerate}
Note that \Cref{theo:V_geo_ergo_loi_borne} applied with \Cref{ass:geo_ergo_1bis} generalizes \cite[Theorem 3.1]{Doucet2017}-\ref{item_condition_doucet_A} since no condition on $\rate$ is required, which is nice in practice. In addition, \Cref{theo:V_geo_ergo_loi_borne} can be applied with other conditions than \Cref{ass:geo_ergo_1bis} \ie~\Cref{ass:geo_ergo_2} and \Cref{ass:geo_ergo_2_b}, which yields new results.   
Also, \Cref{theo:V_geo_ergo_lambda_0} is similar to  \cite[Theorem 3.1]{Doucet2017}-\ref{item_condition_doucet_B}, except that, as stated before, it  holds with more general choices for $\msy$.

The second results of \cite{Doucet2017} studies, in the case of
  thin tail distributions, the BPS process where $\rate$ is replaced
  by $\Lambda_{\mathrm{ref}} : \rset^d\to \rset_+$ defined for any
  $x \in \rset^d$ by
  $\rate + \norm{\nabla U(x)}/ \max(1,\norm{x}^{\epsilon})$ for some
  $\epsilon >0$.  Then, under the conditions that
  \begin{equation*}
    \lim_{\norm{x} \to \plusinfty} \norm{\nabla U(x)} / \norm{x}  = \plusinfty \eqsp, \quad \lim_{\norm{x}\to \plusinfty} \{ \norm{\nabla^2 U(x)}\norm{x}^{\epsilon} / \norm{\nabla U(x)} \} = 0 \eqsp,
  \end{equation*}
\cite[Theorem 3.2]{Doucet2017} shows that   the BPS with refreshment rate $\Lambda_{\mathrm{ref}}$ is
  $V$-geometrically ergodic.  The use of a non-constant, unbounded refreshment rate is motivated in \cite{Doucet2017} by the fact that \cite[Theorem 3.1]{Doucet2017} (the result with constant rate)
  does not apply to potentials equivalent at infinity  to $\norm{x}^\alpha$, $\alpha>2$. For instance, the case 
  of the  Bayesian logistic regression presented in \cite[Example 2]{Doucet2017} for which
  \begin{eqnarray}\label{Eq-U-regression-log}
  U(x) & = & \sum_{i=1}^d g(x_k) + \sum_{i=1}^{n_l} \po -b_i\ps{c_i}{x} + \log \po 1 + \rme^{\ps{c_i}{x}}\pf\pf\,,
  \end{eqnarray}
  with $y_i\in\{0,1\}$ and $c_i\in\mathbb R^d$ for all $i\in\{1,\ldots,n_l\}$, $n_l \in \nsets$ is the number of data points,  and $g(u) = (1+u^2/\sigma^2)^{\beta/2}$ for some parameters $\sigma>0$ and $\beta>2$, is covered by \cite[Theorem 3.2]{Doucet2017} but not \cite[Theorem 3.1]{Doucet2017}. Following the results of \cite{Doucet2017}, one would use a non-constant, unbounded refreshment rate in that practical case. However, first, from a computational point of view, 
  this kind of refreshment rate function may be
  problematic when there is no simple thinning method to sample the refreshment times exactly. Even when a thinning method is available, the cost of each jump is increased since $\nabla U$ has to be computed when a refreshment is proposed.    Moreover, at least for $d =1$ (see \cite{bierkens:duncan:2017}), increasing the refreshment rate - hence the amount of randomness in the system and its diffusive behaviour - increases the asymptotic variance. For these reasons, it was an important question to understand whether the use of a non-constant, unbounded refreshment rate in \cite{Doucet2017} was a practical necessity  or a technical restriction in the theoretical study. Although the assumptions of \Cref{theo:V_geo_ergo_loi_borne} are slightly more restrictive than the conditions of  \cite[Theorem 3.2]{Doucet2017}, our results shows that a constant refreshment (with any positive value) is in fact sufficient for a large class of thin tail distributions, including the logistic regression case \eqref{Eq-U-regression-log} or more generally the cases where $U$ behaves at infinity like  $\norm{x}^\alpha$ for any $\alpha>1$ (from \Cref{theo:V_geo_ergo_loi_borne} with \Cref{ass:geo_ergo_2} thanks to \Cref{propo:alpha_homogeneous}).

Finally, \cite[Theorem 3.3]{Doucet2017} deals with thick tail distributions. It consists in applying
  smooth bijective parametrizations of the space proposed by
  \cite{johnson2012} to get geometric ergodicity of Metropolis-Hastings algorithms
for thick tail distributions by transforming the target into a thin tail one.  It is in fact a general trick that could also be applied in combination of our results. 

\bigskip

As noticed before, \Cref{theo:V_geo_ergo_loi_borne}, \Cref{theo:V_geo_ergo_lambda_0} and \Cref{theo:V_geo_ergo_loi_non_borne} ensue from a more general results, which holds under the following assumption.
\begin{assumption}
  \label{ass:geo_erg_general}
  There exist some positive functions $H \in \mrc(\rset_+)$, $\psi \in\mrc^2(\rset)$, $\ell\in\mrc^1(\rset^d)$, and some constants $R,r,\delta>0$, $c_i>0$ for $i=1,\ldots,4 $  satisfying the following conditions.
  \begin{enumerate}[label=(\roman*)]
  \item   \textbf{Conditions on $U$.} The function $\bU$, defined by  $\bU = \psi \circ U$, satisfies
    \begin{align}
  \label{ass:geo_erg_general_eq_1}      
    &  \lim_{\norm{x} \to \plusinfty} \bU(x)   =  \plusinfty \eqsp, \qquad \int_{\rset^d} \exp\parenthese{ \bU(x) - U(x)} \rmd x  < \plusinfty \\
    \label{ass:geo_erg_general_eq_2}
    & \underset{x\in\R^d}\sup \defEns{ \exp\parenthese{- \bU(x) /4}\po \norm{\na \bU(x)} + \norm{\na^2 \bU(x)}\pf} < \plusinfty \eqsp,  
    \end{align}
    and for all $x \in \rset^d$ with $\norm{x}>R$,
    \begin{equation}
     \label{ass:geo_erg_general_eq_3}
      \norm{\na \bU(x)} \ell(x) \geqslant c_1 \eqsp,  \ell(x) \leqslant  c_2\eqsp,   \norm{\na U(x)} \ell(x) / \norm{\na \bU(x)} \geqslant c_3\eqsp.
   \end{equation}
  \item       \label{ass:geo_erg_general_eq_4}
    \textbf{Conditions on $\loiy$.}
   \begin{equation*}
  \int_{\msy} \rme^{\poty(\norm{y})}  \loiy(\rmd y)  <  \infty\eqsp,   \qquad   \underset{y\in \msy}\sup \defEns{\rme^{-\poty(\norm{y})/2}\norm{y}^2}  <  \infty\eqsp,  
\end{equation*}
\begin{equation*}
    \int_{\msy} \1_{\coint{r,\plusinfty}}(y_1) \loiy(\rmd y) \geqslant \frac{\delta}{2} \eqsp.
\end{equation*}
  \item   \textbf{Conditions on $U$ and $\loiy$.} 
    For $x \in \rset^{d}$, define
    \begin{equation}
      \label{eq:def_a_x}
      \msa_x = \ensemble{ y \in \msy}{\poty(\norm{y}) \leq 3 \bU(x) } \eqsp.
    \end{equation}
     Assume that
   \begin{equation}
    \label{ass:geo_erg_general_eq_5}
    \lim_{\norm{x} \to \plusinfty}  \parentheseDeux{  \norm{\na \ell(x)} \defEns{1 \vee \sup_{y\in \msa_x} \norm{y}} }  =  0  \eqsp,
    \end{equation}     
    and for all $x \in \rset^d$ with $\norm{x}>R$,
   \begin{eqnarray}
\label{ass:geo_erg_general_eq_6}
 \norm{\na^2 \bU(x)} \ell(x) \defEns{ \underset{y\in \msa_x}\sup \norm{y}^2} & \leqslant &    c_4 \eqsp.
    \end{eqnarray}
    % \item   \textbf{Conditions linking $U$, $\loiy$ and $\rate$.}  
  %\label{item:ass:geo_erg_general_1}  $\lim_{t \to \plusinfty}\poty(t) = \plusinfty$, $\sup_{ x\in \rset^d} \ell(x) \leq 1$,
   
%\item \label{item:ass:geo_erg_general_2}

%\item \label{item:ass:geo_erg_general_3}

  % Finally, there exists $R \geq 0$ such that for all $x \in \rset^d$, $\norm{x} \geq R$,
%   \begin{align*}
%           %\label{ass:geo_erg_general_eq_6}
%     c_6  &=   \norm{\na \bU(x)} \ell(x)  >  0\\
%          % \label{ass:geo_erg_general_eq_7}
% \norm{\na U(x)}\ell(x) / \norm{\na \bU(x)} & >  c_7  =  \fracm{2^{20} (1+c_3 + 2 c_4+\lambdab)^5}{\po \rate \int_{\rset^d} \{1\wedge  (y_1c_6)^2\}   \loiy(\rmd y)\pf^4} \eqsp.
%   \end{align*}
  \end{enumerate}
\end{assumption}

%Under \Cref{ass:geo_erg_general},
%consider $V :\rset^{2d} \to \rset$ defined for all
%$(x,y)\in\R^{2d}$ by
%\begin{equation}
%  \label{eq:def_lyap_gene}
%V(x,y) =  \exp( \kappa \bU(x)) \po \zeta+  \varphi\co \ps{y}{ \na \bU(x)}\ell(x) -\eta\cf\pf + \exp(\poty(\norm{y})) \eqsp,
%\end{equation} 
%where $\varphi \in \mrc^1_{b}(\rset)$ is defined for all $s \in \rset$ by
%\begin{equation}
%  \label{eq:def_varphi}
%  \varphi(s) \ = \ \begin{cases}
%-1 & \text{for }s\leq -2\\
%-1 +(s+2)^2/2 & \text{for } s\in [-2,-1]\\
%- s^2/2 & \text{for } s\in [-1,0]\\
%0 & \text{for } s\geq 0 \eqsp.
%                   \end{cases}
%                 \end{equation}
%                 and
%                 \begin{equation}
%                   \label{eq:def_J_kappa}
%                    J = 2^{-1} \int_{\rset^d} \defEns{(\ty_1 c_6)_-^2\wedge 1} \rmd \loiy (\ty) \eqsp, \eqsp \kappa = (1/2)\wedge \rate \eqsp,
%                 \end{equation}
%\begin{equation}
%  \label{eq:def_eta_kappa}
%\eta = \defEns{-\rate J( 4 c_4 + \rate)^{-1}} \wedge(1/2)  \eqsp, \, \zeta =  (\kappa \eta)^{-1}\defEns{ \rate J/2 +c_4+\rate} + 2 \eqsp.
% \end{equation}
%Under \Cref{ass:geo_erg_general}, define 
%\begin{equation} 
%  \label{eq:def_J}
%J =   2^{-1} \int_{\msy} \defEns{ (\gamma y_1 c_1)_-^2\wedge 1} \loiy ( \rmd y) \eqsp.
%\end{equation}  
\begin{theorem}
  \label{theo:geo_ergo_gene}
  Assume \Cref{assum:hyp_base}-\Cref{ass:geo_ergo_1}-\Cref{ass:geo_erg_general}. Assume in addition that  the following inequalities hold
\begin{multline}
  \label{eq:condition_geo_4}
 [ 16 \rate c_2/(rc_1)] \vee [ 64 c_4 c_2 /(rc_1)^2]  \\\leq      \parentheseDeux{(1/3) \wedge\{ \rate \delta r c_1 /(16 c_4) \} } \parentheseDeux{\{c_3/(4c_2)\} \wedge \{\rate \delta  c_3 / (100 r  c_1)\}^{1/2}} \eqsp.
\end{multline}
% and
% \begin{multline}
%   \label{eq:condition_geo_5}
%  \\ \leq \parentheseDeux{(1/3) \wedge\{ \rate \delta r c_1 /(16 c_4) \} } \parentheseDeux{\{c_3/(4c_2)\} \wedge \{\rate \delta  c_3 / (10 r  c_1)\}^{1/2}} \eqsp.
% \end{multline}
% and the jump rate $\rate$ satisfies for all $x \in \rset^d$, $\norm{x} \geq R$,
%  \begin{equation}
%    \label{eq:assum_theo_gene_1}
%         \frac{\norm{\nabla U(x)} \ell(x)}{     \norm{\nabla \bU(x)}}  \geqslant  \frac{\po \kappa + 4 c_2 + 2\rate\pf^5}{\po \rate J\pf^4} \eqsp.
%  \end{equation}
 Then there exists $\kappa \in \ocint{0,1}$ given below by \eqref{eq:def_kappa}, 
  such that $(P_t)_{t \geq 0}$ is
  $\lyap$-uniformly geometrically ergodic with $\lyap$ given for all $(x,y) \in \rset^d \times \msy$ by $   \lyap(x,y) = \exp\parenthese{\kappa \bU(x)} + \exp\parentheseLigne{\poty(\norm{y})}$.
% If, in particular, \eqref{ass:geo_erg_general_eq_3} holds with $c_2$ (respectively  $c_3$) arbitrarily small (respectively large), then $\kappa=1$.
\end{theorem}
\begin{proof}
 The proof is postponed to \Cref{sec:proof_main_theo}.
\end{proof}

\begin{remark}
\label{rem:additio:ass}
  Note that, under \Cref{ass:geo_erg_general}, \eqref{eq:condition_geo_4} is implied by either one of the two following additional assumptions:
\begin{enumerate}[label=(\alph*)]
\item \label{item:condition_add_1} $\lim_{\norm{x} \to \plusinfty} \norm{\nabla \bU(x)} = \plusinfty$;
\item  \label{item:condition_add_2} $ \lim_{\norm{x} \to \plusinfty} \ell(x) = 0$;
\item \label{item:condition_add_3} $\lim_{\norm{x} \to \plusinfty} \norm{\nabla U(x)} \ell(x) /\norm{\nabla \bU(x)} = \plusinfty$. 
\end{enumerate}
Indeed, if \ref{item:condition_add_1} holds, then $c_1$ can be chosen as large as necessary while $c_2,c_4,c_3$ can be held fixed so that \eqref{eq:condition_geo_4} is satisfied.
If \ref{item:condition_add_2} holds, then $c_2$ can be chosen as small as necessary while $c_1,c_3,c_4$ can be held fixed.
Finally if \ref{item:condition_add_3} holds, then $c_3$ can be chosen as large as necessary while $c_1,c_2,c_4$ can be held fixed.
\end{remark}

 Note that if $(P_t)_{t \geq 0}$ is
$V$-uniformly geometrically ergodic then by \cite[Theorem 4.4]{glynn:meyn:1996}, a
functional Central Limit Theorem (FCLT) holds.    Let $g : \rset^d \times \msy \rightarrow \rset$ satisfying for all $(x,y) \in \rset^d \times \msy$, $\abs{g}^2 \leqslant C V$ for some $C>0$. Let $(X_t,Y_t)_{t \geq 0}$ be a BPS process with initial distribution $\mu_0 \in \mcp(\rset^d \times \msy)$, satisfying $\mu_0(V) < \plusinfty$. For $t\geqslant 0$ and $n\in \nset_*$, define
\[G^{n}_t = \frac{1}{\sqrt{n}} \int_0^{nt} \po g(X_s,Y_s) - \tpi (g) \pf \dd s.\]
Then, there exists $\sigma_g \geqslant 0$ such that the sequence of processes $\{(G_t^n)_{t \geq 0}, n \in \nset\}$ converges as $n\rightarrow\infty$ toward $ (\sigma_g B_t)_{t \geq 0}$ in the Skorokhod space, where $(B_t)_{t \geq 0}$ is a standard Brownian motion. It is also possible to consider moderate deviation \cite{Guillin2001,DoucGuillinMoulines2008} or large deviation principle \cite{Wu2000,KontoyiannisMeyn2005}

\section{Proofs of the main results} %\label{sec:proof-main}
\label{sec:proof_main_result}

For the proof \Cref{theo:geo_ergo_gene}, we follow the Meyn and Tweedie approach, based upon two ingredients: a Foster-Lyapunov drift and a local Doeblin condition on compact sets. % Here, a Lyapunov function is a function $V: \rset^d \times \msy \rightarrow [1,\infty)$ with compact level sets, such that, for some $a,b>0$,
% \[\generator V \ \leqslant \ -a V + b\eqsp,\]
% where $\generator$ is the infinitesimal generator of the semi-group $(P_t)_{t\geqslant0}$.
This section is organized as follows. Before showing the Foster-Lyapunov drift in \Cref{sec:Lyapunov}, we introduce the generator of the BPS in \Cref{sec:Generator}. Then in \Cref{sec:couplage}, we show that under appropriate conditions, the BPS satisfies a local Doeblin condition on compact sets.   Contrary to the previous works \cite{MonmarcheRTP,Doucet2017,bierkens:roberts:zitt:2017}, this result is obtained in the case where $\loiy$ has a density with respect to the Lebesgue measure  by a direct coupling. With these two elements in hand, \Cref{theo:geo_ergo_gene} is  proven in \ref{sec:proof_main_theo}. The proofs of \Cref{theo:V_geo_ergo_loi_borne}, \Cref{theo:V_geo_ergo_lambda_0}  and \Cref{theo:V_geo_ergo_loi_non_borne}  are given in \Cref{sec:autre-demo_1}, \Cref{sec:proof_lambda_0} and \Cref{sec:proof-crefth}.

\subsection{Generator of the BPS}
\label{sec:Generator}

 The BPS process  belongs to the class of Piecewise
Determistic Markov Processes (PDMP). Indeed, consider the ordinary
differential equation on $\rset^{2d} $
\begin{equation}
\label{eq:definition_ode}
  \frac{\rmd }{\rmd t} 
\begin{pmatrix}
x_t\\ 
y_t
\end{pmatrix}
=
\begin{pmatrix}
y_t\\ 
0
\end{pmatrix} \eqsp,
\end{equation}
and define  for all $t \geq 0$,
the map $\phi_t : \rset^{2d}  \to \rset^{2d} $
given for all $(x,y) \in \rset^{2d} $ by
\begin{equation}
  \label{eq:definition_flow}
 \phi_t(x,y)  =  (x+ty,y) \eqsp.
\end{equation} 
The family $\sequenceD{\phi_t}[t][\rset_+]$ is referred to as the flow
of diffeomorphisms associated with \eqref{eq:definition_ode} \ie~for
all $(x,y) \in \rset^{2d}$, $t \mapsto \phi_t(x,y)$ is solution of
\eqref{eq:definition_ode} started at $(x,y)$ and for all $t \geq 0$,
$(x,y) \mapsto \phi_t(x,y)$ is a $\mrc^{\infty}$-diffeomorphism.
% Besides by uniqueness of solutions of
% \eqref{eq:definition_ode}, we have for all $t,s \in \rset_+$,
% $\phi_{t+s}(x,y) = \phi_t \circ \phi_{s}(x,y)$. 
 In addition to the deterministic flow $\sequenceD{\phi_t}[t][\rset_+]$, the BPS, as a PDMP, is characterized by a
function $\lambda : \rset^d \times \msy \to \rset_+$, referred to as the jump
rate, and a Markov kernel $Q$ on $\rset^d \times \msy \times \mcb{\rset^d \times \msy}$,
 defined for all $(x,y) \in \rset^d \times \msy$ and $\msa \in \mcb{\rset^d \times \msy }$ by
\begin{align*}
%\label{eq:definition_lambda}
  \lambda(x,y) & =   \ps{y}{ \nabla U(x)}_+ + \lambdab \eqsp,  \\
  \nonumber
Q((x,y), \msa) & = \parentheseDeux{\updelta_x \otimes \defEns{\frac{\ps{y}{ \nabla U(x)}_+}{\lambda(x,y)} \updelta_{\Refl(x,y)} + \frac{\lambdab} {\lambda(x,y)} \loiy }}(\msa) \eqsp,
\end{align*}
where $\updelta_{x}$ is the Dirac measure at $x \in \rset^d$. With
these definitions in mind, we can define a PDMP (in the sense of
\cite{davis:1993}) $(\tX_t,\tY_t)_{t \geq 0}$ which has the same
distribution as $(X_t,Y_t)_{t \geq 0}$ on the space
$\mrd(\rset_+,\rset^d)$ of càdlàg functions
$\omega : \rset_+ \to \rset^d $, endowed with the Skorokhod
topology, see \cite[Chapter 6]{jacod:shiryaev:2003}.

Consider some initial condition $(x,y) \in
\rset^{2d}$, a family of \iid~random variables $(\tE_i,\tG_i,\tW_i)_{i \geq
  1}$ on the probability space $(\Omega,\mcf,\mathbb{P})$ introduced in \Cref{sec:presentation-bps}, where for all $i \geq 1$, $\tE_i$ is an exponential random
variable with parameter $1$, $\tG_i$ is a  random
variable with distribution $\loiy$, $\tW_i$ is a uniform
random variable and $\tE_i$, $\tG_i$ and $\tW_i$ are independent. Set $(\tX_0, \tY_0) = (x,y)$ and $\tS_0=0$. We define by
recursion the jump times of the process and the process itself. For all $n
\geq 0$, let
\begin{equation*}
  \tT_{n+1} = \inf\ensemble{t \geq 0}{\int_0^t \lambda\defEns{\phi_s(\tX_{\tS_n},\tY_{\tS_n})} \rmd s \geq \tE_{n+1}} .
\end{equation*}
Set $\tS_{n+1} = \tS_n + \tT_{n+1}$, $(\tX_t,\tY_t) = \phi_t(\tX_{\tS_n},\tY_{\tS_n})$ for all $t \in \cointLigne{\tS_n,\tS_{n+1}}$,  
$\tX_{\tS_{n+1}} = \tX_{\tS_n} + \tT_{n+1} \tY_{\tS_n} $ and
\begin{equation*}
  \tY_{\tS_{n+1}} = 
  \begin{cases}
     \tG_{n+1} &  \text{ if $\tW_{n+1} \leq \lambdab/\lambda(\tX_{\tS_{n+1}} ,\tY_{\tS_n})$} \\
     \Refl(\tX_{\tS_{n+1}}  , \tY_{\tS_n}) & \text{ otherwise} \eqsp,
   \end{cases}
 \end{equation*}
 where $\Refl$ is defined by \eqref{eq:definition_Refl}.
Thus, $(\tX_t,\tY_t)$ is defined for all $t < \sup_{n \in \nset} \tS_n$ and we
set for all $t \geq \sup_{n \in \nset} \tS_n$, $(\tX_t,\tY_t) =
\infty$, where $\infty$ is a cemetery point. Note that for all $n \in \nset^*$, $(\tX_{\tS_n}, \tY_{\tS_n})$
is distributed according to $Q((\tX_{\tS_n},
\tY_{\tS_{n-1}}),\cdot)$.

From \cite[\Cref{lem:superposition2}]{DurmusGuillinMonmarche:toolbox}, $(\tX_t,\tY_t)_{t \geq 0}$ and $(X_t,Y_t)_{t \geq 0}$ have the same distribution (in particular, almost surely $\sup_{n \in \nset} \tS_n = \infty$ and $(\tX_t,\tY_t)_{t \geq 0}$ is a   $(\rset^{d} \times \msy)$-valued \cadlag~process).

\bigskip

Consider the canonical process associated with the BPS process
$(X_t,Y_t)_{t \geq 0}$, still denoted by $(X_t,Y_t)_{t \geq 0}$ on the
Skorokhod space
$(\mrd(\rset_+,\rset^d \times \msy),\mcf,(\mcf_t)_{t \geq 0},$ $(\mathbb{P}_{x,y})_{(x,y) \in \rset^d \times \msy})$, where $\mcf$ is the Borel
$\sigma$-field associated with the Skorokhod topology,
$(\mcf_t)_{t \geq 0}$ is the completed natural filtration, and for all
$(x,y) \in \rset^d \times \msy$, $\PP_{x,y}$ is the distribution of the
BPS process starting from $(x,y) \in \rset^d \times \msy$.  For all $t \geqslant 0$
and Borel measurable functions
$f,g : \rset^d \times \msy \to \rset$ such that, for all $(x,y) \in \rset^d \times \msy$,
$s \mapsto g((X_s,Y_s))$ is integrable $\mathbb{P}_{(x,y)}$-almost
surely, denote
\begin{equation}
\label{eq:def_martin_local_pdmp_generator}
  \martfg_t = f(X_t,Y_t)-f(X_0,Y_0)-\int_0^t g(X_s,Y_s) \rmd s \eqsp.
\end{equation}
The (extended) generator and its domain $(\generator, \domain(\generator))$
associated with the semi-group $(P_t)_{t \geq 0}$ are defined as follows: $f \in
\domain(\generator)$ if there exists a Borel measurable function $g : \rset^d \times \msy \to \rset$
such that $(\martfg_t)_{t\geqslant0} $ is a local martingale under $\mathbb{P}_{(x,y)}$ for
all $(x,y) \in \rset^d \times \msy$ and, for such a function, $\generator f = g$.
Despite its very formal definition, $(\generator,
\domain(\generator))$ associated with $(P_t)_{t \geq 0}$ can be easily
described. Indeed, \cite[Theorem 26.14]{davis:1993}  shows that $\domain(\generator) = \mse_1 \cap \mse_2$ where
\begin{multline*}
%  \label{eq:definition_esp_fun_abs_cont}
  \mse_1 =\left\lbrace f \in \MeasFspace(\rset^d \times \msy) \, : \, t \mapsto f(\phi_t(x,y)) \right.  \\
\left.    \text{ is absolutely continuous on $\rset_+$ for all $(x,y) \in \R^{2d}$} \right\rbrace \eqsp,
\end{multline*}
and $\mse_2$ is the set of Borel measurable functions $f : \rset^d \times \msy \to \rset$ such that there exists an increasing sequence of $(\mcf_t)_{t \geq 0}$-stopping time $(\sigma_n)_{n\geqslant 0}$, such that for all $(x,y) \in \R^{2d}$, $\lim_{n \to \plusinfty } \sigma_n = \plusinfty\, \, \mathbb{P}_{(x,y)}$-almost surely, and for all $n \in \nset^*$,
\begin{equation}
  \label{eq:condition_generator}
  \expeMarkov{(x,y)}{\sum_{k = 1}^{\plusinfty} \1_{\{\S_k \leq \sigma_n\}} \abs{f(X_{\S_k},Y_{\S_k}) - f(X_{\S_k-},Y_{S_k-})} }< \plusinfty \eqsp.
\end{equation}
Taking for all $n \in \nset^*$, $\sigma_n = S_n \wedge n \wedge \upsilon_n$, where $\upsilon_n = \inf \{ t \geq 0\, : \, \norm{X_t} \geq n \}$, \eqref{eq:condition_generator} is satisfied for any  function $f\in \mrc(\rset^d \times \msy)$ such that for all $x\in\rset^d$, $\int_{\msy} |f(x,w)| \rmd \loiy(w) < \infty$. 

% \mathbb{E}_2 &=\left \lbrace f \in \MeasFspace(\rset^{2d}) \, : \, \text{there exists an increasing sequence of $(\mcf_t)_{t \geq 0}$-stopping time $(\S_n)$}, \right. \\
% &  \left. \lim_{n \to \plusinfty } \S_n = \plusinfty \mathbb{P}_x-\as, \, \expeMarkov{(x,y)}{\sum_{k \in \nset} \1_{T_k \geq \sigma_n} \defEns{f(X_{T_k},Y_{T_k}) - f(X_{T_k},Y_{T_k})} }< \plusinfty \text{for all $x \in \rset^d$ and $n \in \nset$}  \right \rbrace  \\
% \end{equation}

Then, for all $f \in  \domain(\generator)$ and $x,y \in \rset^d \times \msy$,
\begin{multline}
  \generator f(x,y) = \Ddir{y} f(x,y) + (\ps{y}{\nabla U(x)})_+ \defEns{f(x,\Refl(x,y)) - f(x,y)}\\ + \rate \defEns{\int_{\msy} f(x,w) \rmd \loiy(w)  - f(x,y)} \eqsp,
\end{multline}
where
\begin{equation*}
  \Ddir{y} f(x,y) =
  \begin{cases}
    \lim_{t \to 0} \frac{f(\varphi_t(x,y)) - f(x,y)}t \eqsp, &\text{ if this limit exists} \\
    0 & \text{ otherwise} \eqsp.
  \end{cases}  
\end{equation*}
In particular, if $x \mapsto f(x,y)$ is $\mrc^1$ for all $y\in\msy$, then
\begin{multline}
  \label{eq:extended_generator_expression}
  \generator f(x,y) = \ps{y}{\nabla f(x,y)} + (\ps{y}{\nabla U(x)})_+ \defEns{f(x,\Refl(x,y)) - f(x,y)}\\ + \rate \defEns{\int_{\msy} f(x,w) \rmd \loiy(w)  - f(x,y)} \eqsp.
\end{multline}

\subsection{Foster-Lyapunov drift condition}
\label{sec:Lyapunov}

For $a,b,c \in\rset_+$, $a\leq b \leq c$, $c-b \leq b-a \leq a$ and
$\varepsilon \in \ocint{0,1}$ consider a non-decreasing continuously
differentiable function $\varphi : \rset_+ \to \coint{1,\plusinfty}$
satisfying
\begin{equation}
\label{eq:def_varphi_1}
\begin{aligned}
  \varphi(s)&  =    1  & \text{ if $s \in \ocint{-\infty,-2}$} \\
    1+a(s+2) -\varepsilon  \leq \varphi(s) &\leq 1+a(s+2) +\varepsilon & \text{ if $s \in \ooint{-2,-1}$}\\
    \varphi(s) &=  1+b+s(b-a) &  \text{ if $s \in \ccint{-1,0}$} \\
    1+b+s(c-b) -\varepsilon \leq \varphi(s)  &\leq 1+b+s(c-b) +\varepsilon &  \text{ if $s \in \ooint{0,1}$} \\
  \varphi(s)&  =    1+c  &  \text{ if $s \in \ccint{1,\plusinfty}$} \eqsp,    
  \end{aligned}
\end{equation}
and
\begin{equation}
\label{eq:def_varphi_2}
  \sup_{s \in \ccint{-2,-1}} \varphi'(s) \leq a +\varepsilon  \eqsp, \sup_{s\in \ccint{0,1}} \varphi'(s) \leq c-b + \varepsilon \eqsp. 
\end{equation}
% An explicit example of a non-decreasing continuously
% differentiable function $\varphi$ is given in \Cref{aa}.
In addition for $\kappa \in \ocint{0,1}$, under \Cref{ass:geo_erg_general}, define the Lyapunov function $V : \rset^d \times \msy \to \coint{1,\plusinfty}$ by
% \begin{equation}
%   \label{eq:def_lyap_gene}
%   \begin{aligned}
%     V(x,y) &= V_1(x,y) + V_2(y) \eqsp, \\
%     V_1(x,y) & = \exp( \kappa \bU(x))  \varphi\defEns{ (2 \ell(x)/(r c_1))\ps{y}{ \na \bU(x)}}   \eqsp, \qquad V_2(y) =  \exp(\poty(\norm{y})) \eqsp.
%   \end{aligned}
% \end{equation}
\begin{equation}
  \label{eq:def_lyap_gene}
    V(x,y)  = \exp( \kappa \bU(x))  \varphi\defEns{ (2 \ell(x)/(r c_1))\ps{y}{ \na \bU(x)}}  + \exp(\poty(\norm{y})) \eqsp.
\end{equation}
This section is devoted to the proof of a Foster-Lyapunov drift condition for the generator $\generator$ given by \eqref{eq:extended_generator_expression} and the function $V$ defined in \eqref{eq:def_lyap_gene}.

\begin{lemma}
  \label{lem:lyapunov}
Assume  \Cref{assum:hyp_base}-\Cref{ass:geo_ergo_1}-\Cref{ass:geo_erg_general}  and \eqref{eq:condition_geo_4} hold.
There exist $a,b,c \in\rset_+$, $a \leq b \leq c$, $c-b\leq b-a \leq a$,
$\varepsilon \in \ocint{0,1}$ and $\kappa \in \ocint{0,1}$ such that $\generator$ given by \eqref{eq:extended_generator_expression} satisfies a Foster-Lyapunov drift condition with the Lyapunov function $V$, \ie~there exist $A_1,A_2>0$ such that, for all $(x,y)\in\R^{d}\times\msy$,
\begin{equation}
  \label{eq:drift_condition}
\generator V(x,y) \leq   A_1 \left(A_2 - V(x,y) \right) \eqsp. 
\end{equation}
\end{lemma}

% \textcolor{red}{
 Inequality \eqref{eq:drift_condition} means that, away from a given compact set, in average, $V$ tends to decay along a trajectory of the BPS. Before proceeding into the details, let us give a brief explanation on the roles of the different parts of $V$ in this decay. %Call 
 %\[\mathcal A \ =\ \{(x,y)\in\R^{2d}, \ |x|\leqslant 2R_3\text{ or } 1 + |y|^2 \geqslant 5 \sigma^2 f(x)\}.\] 
 When $x$ has a large norm and $y\notin \msa_x$, the leading term of both
 $\lyap$ and $\generator\lyap$  is $\exp(\poty(\norm{y})) $, which appears in $\generator\lyap$, thanks to the refreshment operator,
 with the negative factor $-\rate$. In other words, when the scalar velocity is large, then it will typically decrease at the next refreshment time, so that $V$ will decrease. The main difficulty appears as
 $y\in \msa_x$.  
  The reason why $V$ should decrease in average depends on $\theta(x,y)=\ps{y}{ \na \bU(x)}$: when this is large  enough, the process is likely to bounce, which causes $\varphi(\theta)$ to change to $\varphi(-\theta)$, which is smaller, so that $V$ decreases. When $\theta$ is negative enough, the deterministic transport leads $\exp(\kappa \bU)$, hence $V$, to decrease. Finally, when $|\theta|$ is small, $\varphi(\theta)$ is close to 1, hence is larger than its mean with respect to $\loiy$, so that it can be expected to decrease at the next refreshment time. 
 
  Remark that, because of the operator  $f \mapsto \int_{\msy} f(\cdot,w) \rmd \loiy(w)$, the construction of $V$ at a point $(x,y)$ influences the value of $\generator V$ at all points  $\{(x,v)$, $v\in\msy \}$. Similarly, the term   $f( x,\Refl(x,y))$ is non-local. This yields contradictory constraints: for instance, when $\theta$ is large, while the bounce mechanism typically makes $\varphi(\theta)$ decrease, the deterministic transport leads  $\exp(\kappa \bU)$ to increase. Thus, in order for $V$ to decrease in average, we need $\kappa$ to be small enough. On the contrary, when $\theta$ is negative enough, $\exp(\kappa \bU)$ tends to decrease, but then $\varphi(\theta)$ is below its mean with respect to $\loiy$, so that it is expected to increase at the next refreshment time. Then we would like $\kappa$ to be large enough. The condition \eqref{eq:condition_geo_4} on the $c_i$'s and on $\rate$ ensures that the different constraints are compatible. % The same mechanism explains the restriction on the refreshment rate in \cite{Doucet2017}.}

\begin{proof}
  For ease of notation, we denote in the following  for any
  $(x,y) \in \rset^d \times \msy$
  $\theta(x,y) = \ps{\nabla \bU(x)}{y}$. From \eqref{eq:extended_generator_expression} and the facts that 
 $\nabla \bU(x) = \psi'(U(x)) \nabla U(x)$ and $\norm{\Refl(x,y) }= \norm{y}$, for any $(x,y) \in \rset^d \times \msy$,
\begin{multline}
  \label{eq:bound_lyap_gene_ini}
  \generator V(x,y)  = e^{\kappa \bU(x)} J(x,y) + \rate \defEns{\int_{\msy} e^{\poty(\norm{w})} \loiy( \rmd w) - e^{\poty(\norm{y})}} \eqsp,
\end{multline}
where

\begin{align}
\label{eq:bound_lyap_gene_ini_0}
&  J(x,y)  = \kappa \theta(x,y) \varphi\defEns{2\ell(x)\theta(x,y)/(rc_1)} \\
  \nonumber
  & + (2 /(rc_1)) \varphi'\defEns{2\ell(x)\theta(x,y)/(rc_1)}\parentheseDeux{ \ell(x) \ps{y}{ \nabla^2 \bU(x) y} +   \theta(x,y) \ps{\nabla \ell(x)}{y}}  \\
  \nonumber
         &   + \frac{\norm{\nabla U(x)}}{ \norm{\nabla \bU(x)}} \{ \theta(x,y)\}_+ \parentheseDeux{\varphi\defEns{-2\ell(x)\theta(x,y)/(rc_1)} - \varphi\defEns{2\ell(x)\theta(x,y)/(rc_1)}} \\
  \nonumber
         &   + \rate\defEns{\int_{\msv} \varphi\defEns{(2\ell(x)/(rc_1)) \ps{\nabla \bU(x)}{w}} \rmd \loiy(w)  - \varphi\defEns{2\ell(x)\theta(x,y)/(rc_1)}}\eqsp.
\end{align}

The first step of the proof is to show that there exist
$A_{1,1}, A_{1,2} >0$ such that
\begin{equation}
  \label{eq:proof_drift_bound_0_1}
  \text{
    $\generator V(x,y) \leq -A_{1,1} V(x,y) + A_{1,2}$ for any
    $(x,y)\in \rset^d\times \msy$, $y \not \in \msa_x$} \eqsp,  
\end{equation}
where
$\msa_x \subset \msy$ is defined by \eqref{eq:def_a_x}.  In a
second step, we show that there exist
$A_{2,1} ,A_{2,2} >0$ such that
\begin{equation}
  \label{eq:proof_drift_bound_0_2}
  \text{
    $\generator V(x,y) \leq -A_{2,1} V(x,y) + A_{2,2}$ for any
    $(x,y)\in \rset^d\times \msy$, $y  \in \msa_x$} \eqsp. 
\end{equation}
Note that if \eqref{eq:proof_drift_bound_0_1} and \eqref{eq:proof_drift_bound_0_2} hold, then the proof is concluded.

\textit{Proof of \eqref{eq:proof_drift_bound_0_1}.} Let $(x,y)\in \rset^d\times \msy$, $y \not \in \msa_x$. From \eqref{eq:bound_lyap_gene_ini_0} and the facts that $\varphi$ is bounded by $1+c$, that $\varphi(-s) - \varphi(s) \leq 0$ for any $s \in \rset_+$ since $\varphi$ is non-decreasing,  and that $\sup_{s \in \rset} \varphi'(s) \leq (a+\varepsilon) \vee b \vee ((c-b) + \varepsilon) \leq 1+c$ since $\varepsilon \leq 1$,  we have
\begin{multline}
  \label{eq:bound_J_1_0}
  J(x,y) \leq (1+c) \Big[ \kappa \norm{\nabla \bU(x)}\norm{y}  \\
\left.    +(2/(rc_1))\defEns{\norm{y} \norm{\nabla \ell(x)} + \ell(x) \norm{y}^2\norm{\nabla^2 U(x)}} + \rate \right] \eqsp.
\end{multline}
By  \eqref{ass:geo_erg_general_eq_3} and \eqref{ass:geo_erg_general_eq_5} and since $\ell \in \mrc^1(\rset^d)$, $ \norm{\nabla \ell}_{\infty} + \norm{\ell}_{\infty} < \infty$.  Therefore plugging 
\eqref{eq:bound_J_1_0} in \eqref{eq:bound_lyap_gene_ini_0} and using  \eqref{ass:geo_erg_general_eq_2} and \Cref{ass:geo_erg_general}-\ref{ass:geo_erg_general_eq_4}, we get
\begin{align}
  \nonumber
  \generator &V(x,y)  \leq  C_1 (1\vee \norm{y}^2)   \exp(5 \bU(x)/4) + C_2 - \rate \exp(H(\norm{y})) \eqsp,\\
  \nonumber
  C_1 &=  (1+c) \left\{(\kappa \normLigne{\nabla \bU \rme^{-\bU/4}}_{\infty}) \vee (2 \norm{\nabla \ell}_{\infty}/(rc_1)) \right.  \\
  \nonumber
  & \qquad \qquad \qquad \left. \vee \rate \vee (2 \normLigne{\nabla^2 \bU \rme^{\bU/4}}_{\infty} \norm{\ell}_{\infty} / (r c_1)) \right \} < \plusinfty \eqsp,
\end{align}
\vspace{-0.6cm}
\begin{equation}
\label{eq:def_c2_proof_drift}
  C_2 = \rate \int_{\msy} \exp(H(\norm{y})) \rmd \loiy(w) < \plusinfty \eqsp.
\end{equation}
Using now \Cref{ass:geo_erg_general}-\ref{ass:geo_erg_general_eq_4} and the continuity of $H$, we get that $C_3 = C_1\sup_{y \in \msy} (1\vee \norm{y}^2) \rme^{-H(\norm{y})/2}$ is finite. Since $y \not \in \msa_x$, $3 \bU(x) \leq H(\norm{y})$ and we obtain
\begin{align*}
  \generator V(x,y) &\leq C_3 \exp(11 H(\norm{y})/12) + C_2 - \rate \exp(H(\norm{y}))\eqsp,  \\
  & \leq -(\rate/2)  \exp(H(\norm{y})) + C_4 \eqsp, \, C_4 = C_2 + \sup_{s \in \rset_+} \{ C_3 \rme^{11 s /12} - \rate \rme^{s} \} \eqsp.
\end{align*}
The proof of  \eqref{eq:proof_drift_bound_0_1} follows upon noting that  $\kappa\leqslant 1$ and that $\varphi$ is bounded by $1+c$, so that  $V(x,y) \leq (2+c) \exp(H(\norm{y}))$ if $y \not \in \msa_x$.

\textit{Proof of \eqref{eq:proof_drift_bound_0_2}.}  We show in \Cref{lem:bound_J} below that there exist $a,b,c \in \rset_+$, $a\leq b \leq c$, $\varepsilon \in \ocint{0,1}$, $\kappa \in \ooint{0,1}$, $R_1 \in \rset_+$ and $\eta \in \rset_+^*$ such that for all $(x,y) \in \rset^d \times \msy$, $y \in \msa_x$ and $\norm{x} \geq R_1$, $  J(x,y) < -\eta$.
% \begin{equation}
%   \label{eq:bound_J_proof_2_0}
%   J(x,y) < -\eta \eqsp. 
% \end{equation}
Note that if this result holds, then for all $(x,y) \in \rset^d \times \msy$, $y \in \msa_x$ and $\norm{x} \geq R_1$, by \eqref{eq:bound_lyap_gene_ini}, 
\begin{align}
  \nonumber
  \generator V(x,y) & \leq -\eta \exp(\kappa \bU(x)) + C_2 - \rate \exp(H(\norm{y})) \\
    \label{eq:proof_drift_bound_2_1}
  & \leq - \{(\eta/(1+c)) \wedge \rate\} V(x,y) +C_2  \eqsp,
\end{align}
where $C_2$ is given by \eqref{eq:def_c2_proof_drift} and we have used
for the last inequality that $\varphi$ is bounded by $1+c$. This result concludes the proof of \eqref{eq:proof_drift_bound_0_2} for $\norm{x} \geq R_1$.
It remains to consider the case $\norm{x} \leq R_1$.

Since $\psi$ and $U$ are continuous, so is $\bU$, so that there exists $M_1$
such that for all  $x \in \ball{0}{R_1}$ and $y\in \msa_x$, $H(\norm{y}) \leq M_1$. Since
$\sup_{w \in \msy} \norm{w}^2 \rme^{-H(\norm{w})} < \plusinfty$ by
\Cref{ass:geo_erg_general}-\ref{ass:geo_erg_general_eq_4}, it follows
that there exists $M_2$ such that for all
$x\in \ball{0}{R_1}$, $\msa_x \subset \ball{0}{M_2}$. Then, using that $\bU \in \rmc^2(\rset^d)$,
$\ell \in \rmc^1(\rset^d)$, $H \in \rmc(\rset_+)$ and
$\varphi \in \rmc^1(\rset)$ we get that there exists $C_5,C_6$ such
that for all $x \in \ball{0}{R_1}$ and $y \in \msa_x$, $\generator V(x,y) \leq C_5$ and
$ V(x,y) \leq C_6$. Combining this result and
\eqref{eq:proof_drift_bound_2_1} concludes the proof of
\eqref{eq:proof_drift_bound_0_2}.
\end{proof}
Let us now precise the parameters we chose in the definition of $V$. Set 
\begin{align}
  \label{eq:def_a}
  & a = 1 \wedge \parenthese{\parentheseDeux{(1/3) \wedge\{ \rate \delta r c_1 /(16 c_4) \} } \parentheseDeux{\{c_3/(4c_2)\} \wedge \{\rate \delta  c_3 / ( 100 r c_1)\}^{1/2}}}^{-1}
\end{align}
\begin{equation}
\label{eq_def_b}
  b-a =   a \parentheseDeux{(1/3) \wedge\{ \rate \delta r c_1 /(16 c_4) \} }
\end{equation}
\begin{align}
\label{eq:def_kappa}
&  \kappa = (b-a)\parentheseDeux{\{c_3/(4c_2)\} \wedge \{\rate \delta  c_3 / ( 100 r c_1)\}^{1/2}}\\
  \nonumber
&= a \parentheseDeux{(1/3) \wedge\{ \rate \delta r c_1 /(16 c_4) \} } \parentheseDeux{\{c_3/(4c_2)\} \wedge \{\rate \delta  c_3 / ( 100 r c_1)\}^{1/2}}
  %\leq 1
\end{align}
\begin{equation}
\label{eq:def_c}
      c-b = [\delta \rate a /(4(4c_4/(rc_2) + 2 \rate))  ] \wedge (b-a) \wedge [(b-a)c_3/(4\kappa c_2)] \wedge (\delta b /4) % \parenthese{1 \wedge \{c_3 /(4 \kappa c_2)\} \wedge \{r c_3 c_1 /(32 c_2)\} \wedge  \{(\rate    \delta c_3 / (2 rc_1))^{1/2}/4 \}} 
\end{equation}
\begin{equation}
  \label{eq:def_vareps}
  \varepsilon = (1/2)\wedge (c-b) \wedge (\kappa rc_1/4 ) \wedge (\rate c_2) \eqsp. 
\end{equation}
Note that $\kappa \leq 1$ and 
\begin{equation}
  \label{eq:bound_b_a_a}
0 \leq  c-b \leq  b-a  \leq a  \leq 1  \eqsp.
\end{equation}

\begin{lemma}
  \label{lem:bound_J}
Assume  \Cref{assum:hyp_base}-\Cref{ass:geo_ergo_1}-\Cref{ass:geo_erg_general} and \eqref{eq:condition_geo_4} hold.
Then for $a,b,c, \kappa, \vareps\in \ocint{0,1}$, given in \eqref{eq:def_a}-\eqref{eq_def_b}-\eqref{eq:def_c}-\eqref{eq:def_kappa}-\eqref{eq:def_vareps} respectively, there exist $\tilde{R},\eta>0$ such that for all $x \in \rset^d $ with $\norm{x} \geq \tilde{R}$ and all $y \in \msa_x$, $ J(x,y) < -\eta$, where $J$ and $\varphi$ are defined by \eqref{eq:bound_lyap_gene_ini_0} and \eqref{eq:def_varphi_1} respectively.
% begin{equation}
%   \label{eq:bound_J_proof_2_0}
%   \eqsp. 
% \end{equation}  
% \
\end{lemma}

% \begin{equation*}
% \generator V(x,y) \leq C_5 \eqsp, \, \, V(x,y) \leq C_6  
% \end{equation*}

\begin{proof}
In the proof, we first give a bound on $J$ for any $(x,y) \in \rset^d$, $y \in \msa_x$. Second, denoting again $\theta(x,y) = \ps{\nabla \bU(x)}{y}$ for $(x,y)\in \rset^d\times \msy$,  we distinguish five cases depending on the
value of  $ 2\ell(x) \theta(x,y)/(rc_1)$ which determines the
contribution of $\varphi$  and $\varphi'$ in $J$.

By \eqref{ass:geo_erg_general_eq_5}, there exists $R_1 \in \rset_+$
such that for any $(x,y) \in \rset^d$, $y \in \msa_x$, $\norm{x} \geq R_1$,
\begin{equation}
  \label{eq:bound_nabla_ell_bound_J}
  \norm{\nabla \ell(x)} \norm{y} \leq \vareps \eqsp.
\end{equation}

From \eqref{ass:geo_erg_general_eq_3}, $\norm{\nabla \bU(x)} \ell(x) \geq c_1$ for all $x \in \rset^d$ with $\norm{x} \geq R$. Using \Cref{ass:geo_erg_general}-\ref{ass:geo_erg_general_eq_4} and  the facts that $\loiy$ is rotation invariant and that $\varphi$ is non-decreasing, bounded by $1+c$ and equal to $1$ on $(-\infty,2]$, we then have for any $x \in \rset^d$ with $\norm{x} \geq R$
  \begin{align*}
    &    \int_{\msv} \varphi\defEns{\frac{2\ell(x)}{rc_1} \ps{\nabla \bU(x)}{w}} \rmd \loiy(w)  =     \int_{\msv} \varphi\defEns{\frac{2\ell(x)\abs{\nabla \bU(x)} w_1}{ rc_1}} \rmd \loiy(w) \\
                                                                                      % & = \int_{\msv} \1_{\ocint{-\infty,r}}(v_1) \varphi\defEns{2\ell(x)\abs{\nabla \bU(x)} v_1 /(rc_1)} \rmd \loiy(w) + \int_{\msv} \1_{\ooint{r,\plusinfty}}(v_1) \varphi\defEns{2\ell(x)\abs{\nabla \bU(x)} v_1 /(rc_1)} \rmd \loiy(w) \\
%&                                      \qquad                                                 \leq  \int_{\msv} \1_{\ocint{-\infty,-r}}(w_1) \varphi\defEns{2\ell(x)\abs{\nabla \bU(x)} w_1 /(rc_1)} \rmd \loiy(w) + (1+c)\int_{\msv} \1_{\ooint{-r,\plusinfty}}(w_1) \rmd \loiy(w) \\
    &  \leq                                                                                \int_{\msv} \1_{\ocint{-\infty,-r}}(w_1) \rmd \loiy(w) + (1+c)\int_{\msv} \1_{\ooint{-r,\plusinfty}}(w_1) \rmd \loiy(w)\\
&     \qquad \qquad  \qquad \qquad  \qquad \qquad \leq  1+(1-\delta/2)c \eqsp. 
% &                                     \qquad                                          \leq (1+c) + \int_{\msv} \1_{\ocint{-\infty,-r}}(w_1)   \parentheseDeux{\varphi\defEns{2\ell(x)\abs{\nabla \bU(x)} w_1 /(rc_1)}-(1+c)} \rmd \loiy(w) \\ \    
  \end{align*}
Therefore, combining this result, \eqref{eq:bound_nabla_ell_bound_J}, \eqref{ass:geo_erg_general_eq_6} and the fact that $\varphi$ is non-decreasing so that $\varphi'(s) \geq 0$ for any $s \in \rset$,  we get, for any $x \in \rset^d$ with $\norm{x} \geq R_2 = R \vee R_1$ and all $y \in \msa_x$, 
\begin{align}
\nonumber
  &   J(x,y) \leq \kappa \theta(x,y) \varphi\defEns{2\ell(x)\theta(x,y)/(rc_1)} \\
  \nonumber
  & \qquad \qquad + (2 /(rc_1)) \varphi'\defEns{2\ell(x)\theta(x,y)/(rc_1)}\parentheseDeux{ c_4 +   \abs{\theta}(x,y) \vareps}  \\
  \nonumber
         &   + \frac{\norm{\nabla U(x)}}{ \norm{\nabla \bU(x)}} \{ \theta(x,y)\}_+ \parentheseDeux{\varphi\defEns{-2\ell(x)\theta(x,y)/(rc_1)} - \varphi\defEns{2\ell(x)\theta(x,y)/(rc_1)}} 
\end{align}
\vspace{-0.65cm}
\begin{align}
\label{eq:proof_bound_J_first_bound}
         &   \qquad \qquad + \rate\defEns{1+(1-\delta/2)c  - \varphi\defEns{2\ell(x)\theta(x,y)/(rc_1)}}\eqsp.
\end{align}

Let $(x,y) \in \rset^d \times \msy$, $y \in \msy$, $\norm{x} \geq R_2$. We consider now five cases. 

\textit{Case $1$ : $2 \ell(x)\theta(x,y)/(r c_1) \in \ocint{-\infty,-2}$.} Since for $s \in \ocint{-2,-\infty}$, $\varphi(s) = 1$, \eqref{eq:proof_bound_J_first_bound} reads
\begin{equation}
  \label{eq:case_1_0}
  J(x,y)  \leq \kappa \theta(x,y) + (1-\delta/2)\rate c \eqsp. 
\end{equation}
Using the facts that   $2 \ell(x)\theta(x,y)/(r c_1) \in \ocint{-\infty,-2}$, that    $\ell(z) \leq c_2$  for all $z \in \rset^d$ by \eqref{ass:geo_erg_general_eq_3},  that $(b-a)\vee (c-b) \leq a $ by \eqref{eq:bound_b_a_a}, that $a \leq r c_1 \kappa/ (6 \rate c_2)$ by \eqref{eq:def_kappa} and that  \eqref{eq:condition_geo_4} holds, we get 
\begin{equation*}
 rc_1 \kappa/(2\ell(x)) \geq  rc_1 \kappa/(2c_2) \geq 3 \rate a   \geq (1-\delta/2)\rate c   \eqsp.
\end{equation*}
By this result and \eqref{eq:case_1_0}, we obtain
\begin{equation}
  \label{eq:case_1_final}
  J(x,y) \leq -rc_1 \kappa/(2c_2)\eqsp.
\end{equation}

\textit{Case $2$ : $2 \ell(x)\theta(x,y)/(r c_1) \in \ooint{-2,-1}$.} By \eqref{eq:def_varphi_1}-\eqref{eq:def_varphi_2},
$  1+2a+sa -\varepsilon \leq \varphi(s)  \leq 1+2a+sa +\varepsilon$ and $\varphi'(s) \leq a+\varespilon \text{ for $s \in \ooint{-2,-1}$}$, so that \eqref{eq:proof_bound_J_first_bound} reads
\begin{align*}
  &  J(x,y)  \leq \kappa\theta(x,y)\{1+2a+2a\ell(x) \theta(x,y)/(rc_1) -\varepsilon\} \\
  & \qquad  + (2(a+\vareps)/(rc_1))\{c_4-\vareps\theta(x,y)\} \\
  & \qquad  + \rate \{(1-\delta/2)c -2a -2 a \ell(x) \theta(x,y)/(rc_1)+ \varepsilon\} \\
  & \qquad  \leq B_0 + B_1 \theta(x,y) + 2 \ell(x) B_2\theta(x,y)^2/(rc_1)\leq B_0 + (B_1-2\,B_2) \theta(x,y) \eqsp,
\end{align*}
where we have used that $2\ell(x) \theta(x,y) /(rc_1) \in \ooint{-2,-1}$ and that $\ell(x) \leq c_2$ by \eqref{ass:geo_erg_general_eq_3}, and defined
\begin{align*}
  B_0 &= 2(a+\vareps)c_4/(rc_1)+ \rate\{(1-\delta/2)c-2a+\vareps\}\\
  B_1 &= \kappa(1+2a-\varepsilon)-2 \rate a c_2/(rc_1)-2\vareps(a+\vareps)/(rc_1)  \\
        B_2 &= \kappa a  \eqsp. 
\end{align*}
First, \eqref{eq:def_vareps} and \eqref{eq:bound_b_a_a} ensures that $\varepsilon \leq  (1/2) \wedge a \wedge(\rate c_2)$, and therefore 
\begin{equation*}
%  \label{eq:3}
 B_1 - 2\,B _2 \geq \kappa/2 - 4 \rate a c_2 /(rc_1) \geq \kappa/4\eqsp,
\end{equation*}
where we have used that $a \leq rc_1 \kappa /(16 \rate c_2)$ for the last inequality, which is a consequence of \eqref{eq:def_kappa} and \eqref{eq:condition_geo_4}. 
In particular, $B_1 \geq 2 B_2$ and using again that $2 \ell(x)\theta(x,y)/(r c_1) \in \ooint{-2,-1}$ and $\ell(x) \leq c_2$ from \eqref{ass:geo_erg_general_eq_3}, then
\begin{equation}
  \label{eq:ineq_J_case_2_0}
 J(x,y) \leq B_0+(rc_1/(2c_2))(2B_2-B_1) \leq B_0 - rc_1 \kappa/(8c_2)  \eqsp.
\end{equation}
Since $\vareps \leq a \wedge(c-b) $ by \eqref{eq:def_vareps},  $c-b \leq b-a$ by \eqref{eq:def_c} and $b-a \leq a/3$ by  \eqref{eq_def_b}, we have $  B_0 \leq 4ac_4/(rc_1)$. Hence, \eqref{eq:ineq_J_case_2_0} reads
\begin{equation}
    \label{eq:ineq_J_case_2_final}
  J(x,y) \leq  4ac_4/(rc_1)  - rc_1 \kappa /(8c_2)    \leq - rc_1 \kappa /(16c_2)    \eqsp,
\end{equation}
where we have used \eqref{eq:def_kappa} and  \eqref{eq:condition_geo_4}  for the last inequality.

\textit{Case $3$ : $2 \ell(x)\theta(x,y)/(r c_1) \in \ccint{-1,0}$.}
Using the expression of $\varphi$ on $\ccint{-1,0}$ given by \eqref{eq:def_varphi_1},  \eqref{eq:proof_bound_J_first_bound} reads
\begin{align}
\nonumber
  & J(x,y) \leq \kappa \theta(x,y) \{1+b+(b-a)2\ell(x)\theta(x,y)/(rc_1)\} \\
\nonumber
  & \qquad \qquad \qquad \qquad + (2(b-a)/(rc_1))\{c_4-\theta(x,y) \vareps\} \\
  \nonumber
  &\qquad \qquad \qquad \qquad  + \rate \{ (1-\delta/2)c - b -2\ell(x) \theta(x,y)(b-a)/(rc_1)\} \\
    \label{eq:ineq_J_case_3_0}  
  &  \leq B_0 + B_1 \theta(x,y) + B_2 2\ell(x) /(rc_1) \theta(x,y)^2  \leq B_0 + (B_1-B_2) \theta(x,y)  \eqsp,
\end{align}
where we have used that $2\ell(x) \theta(x,y) /(rc_1) \in \ccint{-1,0}$ and $\ell(x) \leq c_2$ by \eqref{ass:geo_erg_general_eq_3}, and defined
\begin{align*}
  B_0 &= 2(b-a)c_4/(rc_1) +\rate\{(1-\delta/2)c-b\} \\
  B_1 & = \kappa(1+b)-2(\vareps+\rate c_2)(b-a)/(rc_1) \\
  B_2 &= \kappa (b-a) \eqsp.
\end{align*}
First, since $c-b \leq \delta b /4 \leq \delta c/4$ and $a \leq c$ by \eqref{eq:def_c}  and \eqref{eq:bound_b_a_a}, we have
\begin{multline}
    \label{eq:ineq_B_0_case_3}
  B_0 \leq 2(b-a)c_4/(rc_1) -\rate \delta c/4 \leq  2(b-a)c_4/(rc_1) -\rate \delta a/4   \\\leq -a \rate \delta / 8 \eqsp,
\end{multline}
where we have used that $b-a \leq \rate \delta a r c_1/(16c_4)$ by \eqref{eq_def_b} for the last inequality.
Second, using $\vareps \leq \rate c_2$ by \eqref{eq:def_vareps}, $(b-a) \leq a/3 \leq 1/3$ by \eqref{eq_def_b}-\eqref{eq:def_a}, we have
\begin{align}
\nonumber
  B_2-B_1&  \leq \kappa (b-a) + 4 \rate c_2 (b-a) /(rc_1) - \kappa(1+b)  \\
  \label{eq:ineq_B_2_B_1_case_3}
         & \leq 4 \rate c_2a/(rc_1) -  \kappa \leq 0 \eqsp,
\end{align}
where we used the definition of $\kappa$ \eqref{eq:def_kappa} and the condition \eqref{eq:condition_geo_4} for the last inequality. 
Combining \eqref{eq:ineq_B_0_case_3} and \eqref{eq:ineq_B_2_B_1_case_3} in \eqref{eq:ineq_J_case_3_0}, we get
\begin{equation}
 \label{eq:ineq_J_case_3_final}
  J(x,y) \leq  -a \rate \delta / 8 
\end{equation}

\textit{Case $4$ : $2 \ell(x)\theta(x,y)/(r c_1) \in \ooint{0,1}$.}
First, note that since $\varphi(s) = 1+b+s(b-a)$ for $s \in \ccint{-1,0}$, and $\varphi$ is non-decreasing, we have for any $s \in \ccint{0,1}$,
\begin{equation*}
  \varphi(-s) - \varphi(s) \leq    \varphi(-s) - \varphi(0) \leq -(b-a)s\eqsp.
\end{equation*}
From this result and the fact by  \eqref{eq:def_varphi_1}-\eqref{eq:def_varphi_2} that 
$  1+b+s(c-b) -\varepsilon \leq \varphi(s)  \leq 1+b+s(c-b) +\varepsilon$ and $\varphi'(s) \leq c-b+\varespilon \text{ for $s \in \ooint{0,1}$}$
we get that \eqref{eq:proof_bound_J_first_bound} reads
\begin{align*}
  &J(x,y) \leq \kappa \theta(x,y) \defEns{1+ b +2\ell(x)\theta(x,y)(c-b+\varepsilon)/(rc_1)+ \varepsilon} \\
  &\qquad \qquad  + (2(c-b+\varepsilon)/(rc_1))\defEns{c_4 + \theta(x,y) \varepsilon}\\
  & \qquad \qquad -(\norm{\nabla U(x)}/\norm{\nabla \bU(x)}) 2\ell(x)(b-a)\theta(x,y)^2/(rc_1)\\
  & \qquad \qquad   + \rate \{1+(1-\delta/2)c -1-b -2\ell(x)\theta(x,y)(c-b-\varepsilon)/(rc_1)+ \varepsilon\}\\
  &\qquad \qquad  \leq B_0 + B_1 \theta(x,y) + 2 \ell(x) B_2 \theta(x,y)^2/(rc_1) \eqsp,
\end{align*}
where we have used that $(\norm{\nabla U(x)}/\norm{\nabla \bU(x)}) \ell(x) \geq c_3$ by \eqref{ass:geo_erg_general_eq_3}, $\theta(x,y) \geq 0$ and defined 
\begin{align*}
  B_0 &= 2c_4(c-b+\varepsilon)/(rc_1)+\rate\{(1-\delta/2)c-b+\varepsilon\}\\
  B_1 &= \kappa(1+b+\vareps) +2\vareps (c-b+\varepsilon)/(rc_1) \\
  B_2 &= \{\kappa (c-b+\varepsilon) -  c_3(b-a)/\ell(x)\} \eqsp.
\end{align*}
Since $\varepsilon \leq c-b$ by \eqref{eq:def_vareps}, $\ell(x) \leq c_2$ by \eqref{ass:geo_erg_general_eq_3} and $ 2 \kappa c_2 (c-b) \leq   c_3(b-a)/2$ by \eqref{eq:def_c}, we get
\begin{equation}
  \label{eq:def_tilde_b_2}
  B_2 \leq -\tilde{B}_2 =-c_3(b-a) /(2\ell(x)) \eqsp,
\end{equation}
and therefore
\begin{equation*}
J(x,y) \leq B_0 + B_1 \theta(x,y) - 2\ell(x) \tilde{B}_2 \theta(x,y)^2/(rc_1) \eqsp.
\end{equation*}
Then, using that  $s \mapsto C_1s - C_2
s^2$ is bounded  by $C_1^2/(2C_2)$ on $\rset$, we obtain
\begin{equation*}
J(x,y) \leq B_0 +\theta(x,y)    rc_1B_1^2/(4 \ell(x) \tilde{B}_2)
\end{equation*}
Therefore, since $\theta(x,y) \in \ooint{0,1}$, to show that
\begin{equation}
  \label{eq:bound_J_case_4_final}
  J(x,y) \leq -\rate \delta c
/16 \eqsp,
\end{equation}
it is sufficient to prove  that
\begin{align}
  \label{eq:bound_J_case_4_1}
  B_0 &\leq -\rate \delta c /4 \\
      \label{eq:bound_J_case_4_3}
 rc_1B_1^2/(4 \ell(x) \tilde{B}_2) & \leq \rate \delta c /8  \eqsp. 
\end{align}
First \eqref{eq:bound_J_case_4_1} holds since using that $\varespilon \leq (c-b)$ by \eqref{eq:def_vareps} and that $a \leq c$, we have
\begin{multline*}
B_0 -\delta/4 =  2c_4(c-b+\varepsilon)/(rc_1)+\rate\{(1-\delta/4)c-b+\varepsilon\} \\\leq (4c_4/(rc_2)+2 \rate)(c-b)  -\delta a \rate /4 \leq 0 \eqsp,
 \end{multline*}
using $   (c-b)  \leq \delta a \rate /(4 (4c_4/(rc_2)+2\rate) )$ by  \eqref{eq:def_c} for the last inequality. 
 % We now show that $B_1 - \tilde{B}_2 \leq 0$ which implies that,
 % \eqref{eq:bound_J_case_4_2} holds. First, note by
 % \eqref{eq:def_kappa} and \eqref{eq:def_a}, we have 
 % \begin{equation*}
 %   2c_2 \kappa - c_3/2 \leq 24\rate c_2^2 a /(rc_1) - c_3/2 \leq 0 \eqsp.
 % \end{equation*}
 % Therefore using this result, $\ell(x) \leq c_2$ by \eqref{ass:geo_erg_general_eq_3},  $\varepsilon \leq a \wedge(c-b)$ by \eqref{eq:def_vareps}, $c-b \leq r c_3 c_1 (b-a)/(32 c_2)$ by \eqref{eq:def_c},  we have 
 % \begin{multline*}
 %   B_1 -\tilde{B}_2 = \kappa(1+b+\vareps) +2\vareps (c-b+\varepsilon)/(rc_1)  -c_3(b-a)/(2\ell(x)) \\ \leq \kappa(1+2a) + 4(c-b)/(rc_1) - c_3(b-a)/(4 c_2)
 % \leq \kappa(1+2a) - c_3(b-a)/(8c_2) \leq 0 \eqsp,
 % \end{multline*}
 % where we used \eqref{eq_def_b} for the last inequality. 
 It remains to establish \eqref{eq:bound_J_case_4_3} which is equivalent by definition of $B_1$ and $\tilde{B}_2$ \eqref{eq:def_tilde_b_2} to
 \begin{equation}
   \label{eq:bound_J_case_4_4}
 \kappa(1+b+\vareps) +2\vareps (c-b+\varepsilon)/(rc_1) \leq  \{\rate c   \delta c_3 (b-a)/ (4 rc_1)\}^{1/2} \eqsp. 
 \end{equation}
Since  $\vareps \leq 1 \wedge (\kappa r c_1/4 )$ by \eqref{eq:def_vareps}, $c-b \leq 1$ and  $b \leq 2$ by \eqref{eq:bound_b_a_a} and \eqref{eq:def_a},   we get
  \begin{equation*}
%   \label{eq:bound_J_case_4_4}
 \kappa(1+b+\vareps)  +2\vareps (c-b+\varepsilon)/(rc_1) \leq 5 \kappa  \eqsp. 
\end{equation*}
This result, the inequality  $b-a \leq c $ and the definition of $\kappa$ \eqref{eq:def_kappa}  implies that \eqref{eq:bound_J_case_4_4} holds.

\textit{Case $5$ : $2 \ell(x)\theta(x,y)/(r c_1) \geq 1$.} Since by \eqref{eq:def_varphi_1}, $\varphi(s) = 1+c$, $\varphi'(s) = 0$ and $\varphi(-s)-\varphi(s) \leq a-c$ for $s \geq 1$, \eqref{eq:proof_bound_J_first_bound} reads
\begin{align}
  \nonumber
&  J(x,y) \leq \kappa \theta(x,y) (1+c)  -\{\norm{\nabla U(x)} / \norm{\nabla \bU(x)}\} \theta(x,y) (c-a) - \rate \delta c/2 \\
\nonumber
        & \leq  \kappa \theta(x,y) (1+c)  -\{\norm{\nabla U(x)} \ell(x)/ (c_2 \norm{\nabla \bU(x)})\} \theta(x,y) (c-a) - \rate \delta c/2 \\
  \nonumber
  &\qquad \leq  \{\kappa (1+c) - c_3(c-a) /c_2\} \theta(x,y) - \rate \delta c/2 \eqsp,
\end{align}
where we have used by \eqref{ass:geo_erg_general_eq_3} that  $\ell(x) \leq c_2$ and 
$\norm{\nabla U(x)} \ell(x) \norm{\nabla \bU(x)}^{-1} \geq c_3$. From $c \leq 3$ by \eqref{eq:bound_b_a_a} we obtain
\begin{align}
  \nonumber
  J(x,y)& \leq  \{\kappa (1+c) - c_3(c-a) /c_2\} \theta(x,y) - \rate \delta c/2 \\
  \label{eq:proof_bound_J_case_1_final1}
        &\leq \{4 \kappa - c_3(b-a) /c_2\} \theta(x,y) - \rate \delta c/2  \leq -\rate \delta c/2 \eqsp,
\end{align}
where we have used the definition of $\kappa$ given by  \eqref{eq:def_kappa} and $\theta(x,y) \geq 0$ for the last inequality.

The proof follows from combining  \eqref{eq:case_1_final}-\eqref{eq:ineq_J_case_2_final}-\eqref{eq:ineq_J_case_3_final}-\eqref{eq:bound_J_case_4_final}-\eqref{eq:proof_bound_J_case_1_final1}.

\end{proof}

\begin{corollary}
\label{Corollary-lyap}
Under \Cref{ass:geo_erg_general}, for all $(x,y)\in \rset\times\msy$ and $t\geqslant 0$,
\begin{equation*}%\label{EqEspV}
 P_t V(x,y)  \leqslant  V(x,y) \rme^{-A_1 t} +  A_2( 1- \rme^{-A_1 t}). 
\end{equation*}
where $V$ is given by \eqref{eq:def_lyap_gene} and $A_1,A_2$ are given by \Cref{lem:lyapunov}.
\end{corollary}
\begin{proof}
  By \cite[Section 31.5]{davis:1993}, since $V \in \domain(\generator)$,  the process $(M_t)_{t \geq 0}$, defined for any $t \in \rset_+$ by
  \begin{equation*}
M_t = \rme^{A_1 t} V(X_t,Y_t) - V(x,y) - \int_0^t\defEns{A_1 \rme^{A_1 s}  V(X_s,Y_s) + \rme^{A_1 s} \generator V(X_s,Y_s)}  \rmd s \eqsp,
\end{equation*}
is a local martingale. Therefore  $(M_{t\wedge \tau_n})_{t \geq 0}$ is a martingale where for all $n \in \nsets$, $\tau_n = \inf\{t \geq 0 \, : \, \norm{X_t}+\norm{Y_t} \geq n\}$ and 
  \begin{multline*}
    \expe{ \rme^{A_1 (t\wedge \tau_n)} V(X_{t\wedge \tau_n},Y_{t\wedge \tau_n})} -  V(x,y)\\
     =  \expe{\int_0^{t\wedge\tau_n}\rme^{A_1 s}\{A_1V(X_s,Y_s)+ \generator V(X_s,Y_s)\} \rmd s} \\
    \leq \expe{\int_0^{t\wedge\tau_n}\rme^{A_1 s} A_1 A_2 \rmd s}\, \leq \, A_2\left(\rme^{A_1 t}-1\right)\eqsp.
  \end{multline*}
  Letting $n$ go to infinity concludes the proof since it yields
  \begin{equation*}
   \rme^{A_1 t}\expe{  V(X_{t},Y_{t})} \, \leq\,  V(x,y) + A_2\left(\rme^{A_1 t}-1\right)\eqsp.
  \end{equation*}
%Differentiating with respect to $t$, we have by \Cref{lem:lyapunov} for all $n \in \nsets$ and $t \geq 0$
%\begin{equation*}
%\partial_t     \expe{ \rme^{A_1 (t\wedge \tau_n)} V(X_{t\wedge \tau_n},Y_{t\wedge \tau_n})}  \leq   A_2( \rme^{A_1 t}-1) \eqsp.  
%\end{equation*}
%Integrating with respect to $t$ and taking $n \to \plusinfty$ concludes the proof.
\end{proof}

\subsection{Mirror Coupling}
\label{sec:couplage}

To obtain geometric ergodicity, the classical Meyn and Tweedie
approach is, once a Lyapunov drift condition holds, to show a Doeblin condition for
some $\msc \subset \rset^d \times \msy$, i.e. that the following holds : there exist
$t >0$, $\varepsilon >0$ and $\nu\in \mcp(\rset^d \times \msy)$, such that
\begin{equation*}
 P_t((x,y), \msa) \geq \varepsilon \nu(\msa) \eqsp \text{ for all
$\msa \in \mcb{\rset^d \times \msy}\,, (x,y)\in\msc $} \eqsp. 
\end{equation*}
A set  $\msc $ that satisfies this is called a small set. 
 %This section is devoted to the proof of the following result:
\begin{lemma} 
\label{lem:condition-doeblin}
Assume \Cref{assum:hyp_base} and \Cref{ass:geo_ergo_1}-\ref{ass:geo_ergo_1_loiy}. Then, any compact set $\msk \subset  \rset^d \times \msy$  is a small set.
\end{lemma}
Previous works \cite{MonmarcheRTP,Doucet2017} establish
\Cref{lem:condition-doeblin} in the case where $\msy=\sphere^d$. The
proof relies on the fact that after two
refreshment events the distribution of $X_t$
has some density \wrt~the Lebesgue density on a ball with a radius
proportional to $t$. 
Nevertheless, the latter
strategy yields a non-explicit rate of convergence. In particular the
dependence of the obtained rate in the dimension of the space is
either intractable or very rough.

For this reason, we will present a different argument, based on an explicit coupling of two BPS processes. However, this will only work under the assumption that $\loiy$ is not singular with respect to the Lebesgue measure on $\R^d$, which rules out, for example, the case of the uniform measure on $\sphere^d$. A general proof of \Cref{lem:condition-doeblin}, with no additional assumption on $\loiy$, may be obtained by a straightforward adaptation of \cite[Lemma 5.2]{MonmarcheRTP} or \cite[Lemma 2]{Doucet2017}. We will only treat the non-singular case, with a particular emphasis on the case where $\loiy$ is a $d$-dimensional non-degenerate Gaussian distribution with zero-mean and covariance matrix $\Sigma$.

 The aim of the rest of this section is to establish the following coupling condition:
%It is commonly known that this condition is equivalent to:
for any compact set $\msc \subset \rset^d \times \msy$, there exist
$t >0$, $\varepsilon >0$ such that for all $(x,y),(\tx,\ty) \in \msc$,
\begin{equation*}
 \tvnorm{P_t((x,y), \cdot) - P_t((\tx,\ty), \cdot)}  \leq 2(1-\varepsilon) \eqsp.
\end{equation*}
This is clearly implied by \Cref{lem:condition-doeblin}. However,
 in order to get good explicit rates of convergence, it may be
more efficient to establish directly a coupling condition, which can then be
directly used to obtain quantitative estimates (see for instance \Cref{thm:ThmCVexpo} in Appendix B and the exemple in \Cref{sec:prec-expl-bound}) . 

% , namely when $\mu_y=\gamma_{\sigma}$ where $\gamma_{\sigma}$ is the Gaussian distribution with variance $\sigma^2$ on $\R^d$, with density
% \[\gamma_{\sigma}(y) \ = \ \frac{1}{\po 2\pi \sigma^2\pf^{d/2}} e^{-\frac1{\sigma^2}\norm[2]{y}}\eqsp.\]
%We call $\gamma_1$ the standard Gaussian law.

Before stating our main result, we need the following lemma
concerning the reflexion coupling (see \cite{lindvall:rogers:1986},
\cite{eberle:2015} and references therein) between two $d$ standard
Gaussian random variables with different means.

\begin{lemma}\label{LemMiroir}
  Let $x^{(1)},x^{(2)} \in \rset^d$, $\Sigma_{\mathrm{R}}$ be a positive definite
  matrix and $(W^{(1)}_t)_{t \geq 0}$ be a standard one dimensional
  Brownian motion. Define $T_{\mrcc} = \inf \defEnsLigne{t \geq 0 \, : \, W^{(1)}_t \geq  \normLigne{\Sigma_{\mathrm{R}}^{-1/2}(x^{(2)}-x^{(1)})}/2}$, the stochastic process $(W^{(2)}_t)_{t \geq 0}$ by
    \begin{equation*}
      W^{(2)}_t =
      \begin{cases}
        -W^{(1)}_t & \text{ if $t \leq T_{\mrcc}$} \\
      -\normLigne{\Sigma_{\mathrm{R}}^{-1}(x^{(2)}-x^{(1)})} + W^{(1)}_t & \text{otherwise} \eqsp,
      \end{cases}
    \end{equation*}
and the $d$-dimensional random variables 
\begin{align*}
  G^{(1)} &= W^{(1)}_1 \mrn\defEns{\Sigma_{\mathrm{R}}^{-1/2}(x^{(2)}-x^{(1)})}  + G_{\mathrm{P}} \eqsp, \\
  G^{(2)} &=  W^{(2)}_1 \mrn\defEns{\Sigma_{\mathrm{R}}^{-1/2}(x^{(2)}-x^{(1)})} +G_{\mathrm{P}} \eqsp,\\
          G_{\mathrm{P}} &= \parenthese{\Id -\mrn\defEns{\Sigma_{\mathrm{R}}^{-1/2}(x^{(2)}-x^{(1)})} \mrn\defEns{\Sigma_{\mathrm{R}}^{-1/2}(x^{(2)}-x^{(1)})}^{\transpose}} G\eqsp,
\end{align*}
where $G$ is a standard $d$-dimensional Gaussian random variable independent of $(W^{(1)}_t)_{t \geq 0}$ and $\rmn$ is given by~\eqref{eq:definition_Refl}. Then $G^{(1)}$ and $G^{(2)}$ are $d$-dimensional standard Gaussian random variables and for all $M \geq 0$,
\begin{multline*}
  \hspace{-.2cm}
\proba{x^{(1)}+ \Sigma_{\mathrm{R}}^{1/2} G^{(1)} = x^{(2)}+\Sigma_{\mathrm{R}}^{1/2} G^{(2)} \, , \, \norm{G^{(1)}-\Sigma_{\mathrm{R}}^{-1/2}(x^{(2)}-x^{(1)})/2} \leq M } \\= \talpha(\normLigne{\Sigma_{\mathrm{R}}^{-1/2}(x^{(2)}-x^{(1)})}, M) \eqsp,
\end{multline*}
where for all $r \geq 0$,
\begin{multline}
\label{eq:def_talpha}
\talpha(r, M) = 
\frac{r}{2 (2\uppi)^{(d+1)/2}}\int_0^{1} \left\{s^{-3/2} \exp\parenthese{-r^2/(8s)} \right. \\
\left.    \int_{\rset^d} \1_{\ccint{0,M}}\parenthese{\parenthese{(1-s)w_1^2+\cdots+w_d^2}^{1/2}} \rme^{-\norm[2]{x}/2} \rmd w \right\}\rmd s   \eqsp.
\end{multline}
%  For $u>0$, let
% %For $x\in\R^d$ and $\sigma>0$, let $\gamma(x,\sigma)$ be the $d$-dimensional Gaussian distribution with mean $x$ and variance $\sigma^2 I_d$, that is to say the law of $x+\sigma G$ where $G$ is a standard $d$-dimensional Gaussian r.v. Then, for all $x,w\in\R^d$ and $\sigma>0$,
% \begin{eqnarray*}
% %\| \gamma(x,\sigma) - \gamma(w,\sigma)\|_1 & \leqslant & 
% p(u) & =&  \int_0^{\frac{4}{u^2}} \frac{t^{-\frac32}e^{-\frac{1}{2 t}}}{\sqrt{2\pi}}   \dd t.
% \end{eqnarray*}
% For $x,w\in\R^d$ and $\eta>0$, there exist non-independent  
% %$d$-dimensional standard Gaussian
% r.v. $G$ and $G'$, each with law $\gamma_1$ and such that, for any $N>0$, %and $A$ is a Bernoulli r.v. with parameter  $p(x_1,x_2,\eta)$, and such that
% \begin{eqnarray*}
% \mathbb P\po  x + \eta G = w + \eta G' \text{ and }  \left| \eta G - \frac{w-x}{2} \right| < N \pf & \geqslant & p\po \frac{|x-w|}{\eta}\pf \mathbb P \po \eta|G|<N\pf.
% \end{eqnarray*}
% On the event $\left\{ x + \eta G = w + \eta G'\right\}$, we say the Gaussian coupling is a success.
\end{lemma}
\begin{proof}
 By
  the Markov property of the Brownian motion $(W_t^{(1)})_{t \geq 0}$,
  since $T_c$ is a $(\mcf^{W}_t)_{t \geq 0}$-stopping time, where
  $\mcf^{W}_t = \sigma(W^{(1)}_s , s \leq t)$, $W^{(2)}_t$ is a
  Brownian motion. Therefore, $G^{(1)}$ and $G^{(2)}$ are $d$-dimensional standard Gaussian random variables.

  Using again the Markov property of
  $(W_t^{(1)})_{t \geq 0}$, given $T_{\mrcc} <1$,
  $W_1^{(1)}-W^{(1)}_{T_{\mrcc}}$ is independent of
  $\mcf_{T_{\mrcc}}^{W}$. Therefore, since $\{x^{(1)}+  \Sigma_{\mathrm{R}}^{1/2} G^{(1)} = x^{(2)}+ \Sigma_{\mathrm{R}}^{1/2} G^{(2)}\} = \{T_{\mrcc} \leq 1\}$ and $G$ is independent of $(W^{(1)}_t)_{t \geq 0}$, we get for all $M \geq 0$,
  \begin{align*}
&      \proba{x^{(1)}+ \Sigma_{\mathrm{R}}^{1/2} G^{(1)} = x^{(2)}+\Sigma_{\mathrm{R}}^{1/2} G^{(2)} \, , \, \norm{G^{(1)}-\Sigma_{\mathrm{R}}^{-1/2}(x^{(2)}-x^{(1)})/2} \leq M } \\
 & =   \expe{\1_{\ccint{0,1}}(T_{\mrcc}) \,  \mathbb{P}\parenthese{\left. \parenthese{(W^{(1)}_1 -W^{(1)}_{T_{\mrcc}})^2 + \norm[2]{\bG}}^{1/2 } \leq M \middle \vert \mcf_{T_{\mrcc}}^{W}  \right. }} \\
 & = (2\uppi)^{-d/2}\expe{\1_{\ccint{0,1}}(T_{\mrcc}) \int_{\rset^d} \1_{\ccint{0,M}}\defEns{\parenthese{(1-T_{\mrcc})w_1^2+\cdots+w_d^2}^{1/2}} \rme^{-\norm[2]{x}/2} \rmd w } \eqsp. 
  \end{align*}
The proof then follows from the explicit expression of the density of $T_{\mrcc}$ \wrt~the Lebesgue measure (see e.g. \cite[ p. 107]{revuz:yor:1994}).
% The Gaussian distribution $\gamma_1$ being isotropic, we can assume without loss of generality that $x=0$ and $w  = |w |(1,0,\dots,0)$. Let $(W_t)_{t\geqslant0}$ be a one-dimensional standard Wiener process and $s = \inf\{t\geqslant0, 2 \eta W_t =|w |\}$. Let
% \[ X_t \ = \ \left\{\begin{array}{ll}
%  -   W_t & \text{for }t\leqslant s\\
% - \frac1\eta|w | +  W_t& \text{for }t\geqslant s,
% \end{array}\right.\]
% which is a standard Wiener process, and let $G''$ be a $(d-1)$-dimensional standard Gaussian r.v., independent from $W$. Then $G=(W_1,G'')$ and  $G'=( X_1, G'')$) are both $d$-dimensional standard Gaussian r.v., and $x+\eta G = w+\eta G'$ whenever $s<1$. The  hitting time distribution for a Wiener process is known (see e.g. \cite[(3.7) Proposition]{revuz:yor:1994}):
% \begin{eqnarray*}
% \mathbb P\po s<1\pf & = & \int_0^1 \frac{|x|}{2\eta \sqrt{2\pi}} t^{-\frac32} e^{-\frac{|x|^2}{8\eta^2 t}} \dd t\ = \ p\po \frac{|x-w|}{\eta}\pf.
% \end{eqnarray*}
% Now, conditionally on the event $\{s<1\}$, $\eta\po W_1 - W_s\pf$, which is equal to $\eta  W_1 - |w|/{2} $, is a Gaussian r.v. independent from what happened before time $s$,  with a variance less than $1$, so that $\eta G - w/2 $ is a Gaussian r.v. with independent coordinates, each of them being of variance at most 1. As a conclusion,
% \begin{eqnarray*}
% \mathbb P\po |\eta G - w/2 |< N \ |\ s<1\pf & \geqslant &  \mathbb P\po \eta|G|< N\pf.
% \end{eqnarray*}
\end{proof}

\begin{lemma}\label{LemCouplageCompact}
Assume \Cref{assum:hyp_base}, $\msy = \rset^d$ and $\loiy$ is the Gaussian measure with zero-mean and covariance matrix $\Sigma$. Then, for all $t>0$ and all compact set $\msk \subset \{(z,w) \in \rset^d \times \msy \, : \, \norm{z} + \norm{w} \leq R_{\msk} \}$ of $\rset^{d} \times \msy$, $R_\msk \geq 0$, for all  $(x,y),(\tx,\ty) \in \msk$ and for all $M \geq 0$, 
\begin{multline*}
(1/2)  \tvnorm{P_t((x,y), \cdot) -  P_t((\tx,\ty),\cdot)} \\
  \leq  1 -\expe{\1_{\ccint{0,\rate t}}(E_1+E_2) \talpha\parenthese{2(\rate+E_1)R_{\msk}\norm{\Sigma^{-1/2}} E_2,M} g(E_2/\rate) } \eqsp,
\end{multline*}
where $\talpha$ is given by \eqref{eq:def_talpha}, for all $r \geq 0$,
\begin{align}
\nonumber
  g(r) &= \proba{r \tM \sup_{z \in \ball{0}{(1+E_1/\rate)R_{\msk}+(r/\rate)\tilde{M}}}\norm{\nabla U(z)} \geq E_3} \eqsp,\\
\label{eq:def_tilde_M}
\tilde{M} &=M + \normLigne{\Sigma^{1/2}} (1+E_1/\rate)R_{\msk} \eqsp,
\end{align}
and $E_1,E_2,E_3$ are three independent exponential random variables with parameter $1$.
\end{lemma}
\begin{proof}
  Let $\msk$ be a compact set of $\rset^{2d}$. Let
  $(x,y),(\tx,\ty) \in \msk$, $(x,y) \not = (\tx,\ty)$. We construct a
  non Markovian coupling $(X_t,Y_t,\tX_t,\tY_t)$ between the two distributions
  $P_t((x,y), \cdot)$ and $P_t((\tx,\ty), \cdot)$ for all $t > 0$, and
  lower bound the quantity $\mathbb{P}((X_t,Y_t)=(\tX_t,\tY_t))$,
  which will conclude the proof using the characterization of the total
  variation distance by coupling.
  
  Before proceeding to its precise definition, let us give a brief and informal description of this coupling (see  \Cref{coupl1}, \Cref{coupl2} and \Cref{coupl3}). We couple both processes to have the same two first refreshment times $H_1$ and $H_2$. At time $H_1$, the Gaussian velocities are chosen according to \Cref{LemMiroir} so that, in the absence of bounces in the meanwhile, with positive probability, the processes will reach the same position at time $H_2$. At time $H_2$, both velocities are refreshed with the same Gaussian variable. Hence, with positive probability, at time $H_2$, the processes have the same position and same velocity, in which case we can keep them equal for all times $t\geqslant H_2$.
  
\begin{figure}[!h]
\begin{tikzpicture}
\begin{axis}[
    xlabel={Time $t$},
    ylabel={Position},
    xmin=0, xmax=100,
    ymin=15, ymax=85,
    xtick={0},
    ytick={0,20,40,60,80,100,120},
    legend pos=north east,
    % ymajorgrids=true,
    extra x ticks={30,60},
    extra x tick
    labels={$H_1$, $H_2$},
    extra x tick style={grid=major}
]
 
\addplot[
    color=blue,
    mark=square,
    ]
    coordinates {
    (0,70)(15,75)(18,78)(30,61)(60,50)(74,20)(89,34)(101,59)
    };
    % \legend{}
\addplot[
    color=red,
    mark=square,
    ]
    coordinates {
    (0,40)(20,45)(25,35)(30,30)(60,50)(74,20)(89,34)(101,59)
    };
    \legend{$(X_t)_{t \geq 0}$,$(\tX_t)_{t \geq 0}$};
  \end{axis}
\end{tikzpicture}
  \caption{Before the first refreshment at time $H_1$, both processes may bounce freely. At time $H_1$, the Gaussian velocities are coupled so that, at time $H_2$ (which is the next refreshment time), provided this Gaussian coupling of the velocities succeeds, and provided they have not bounced in the meanwhile, both processes reach the same position. At time $H_2$, both processes take the same velocity: they have merged, the coupling is a success.}\label{coupl1}
\end{figure}
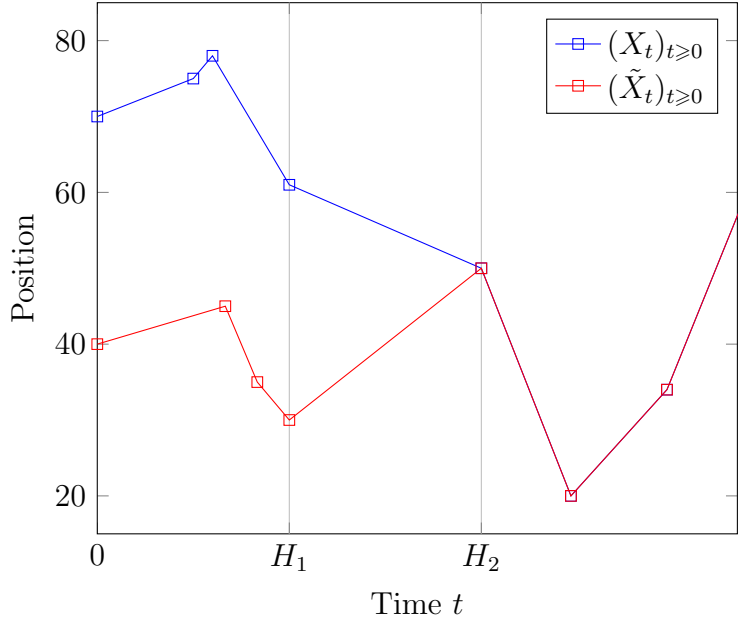

\begin{figure}[!h]
\begin{tikzpicture}
\begin{axis}[
    xlabel={Time $t$},
    ylabel={Position},
    xmin=0, xmax=100,
    ymin=15, ymax=85,
    xtick={0},
    ytick={0,20,40,60,80,100,120},
    legend pos=north east,
    % ymajorgrids=true,
    extra x ticks={30,60},
    extra x tick
    labels={$H_1$, $H_2$},
    extra x tick style={grid=major}
]
 
\addplot[
    color=blue,
    mark=square,
    ]
    coordinates {
    (0,70)(15,75)(18,78)(30,61)(48,54)(60,67)(74,30)(89,24)(101,39)
    };
    % \legend{}
    
\addplot[
    color=cyan,
    dashed
    ]
    coordinates {
(48,54)(60,50)
  };
  
\addplot[
    color=red,
    mark=square,
    ]
    coordinates {
    (0,40)(20,45)(25,35)(30,30)(60,50)(70,20)(90,59)
    };
    \legend{$(X_t)_{t \geq 0}$,,$(\tX_t)_{t \geq 0}$};
  \end{axis}
\end{tikzpicture}
\caption{If one (at least) of the processes bounces between times $H_1$ and $H_2$, then the coupling fails. There may be other bounces after the first one.}\label{coupl2}
\end{figure}
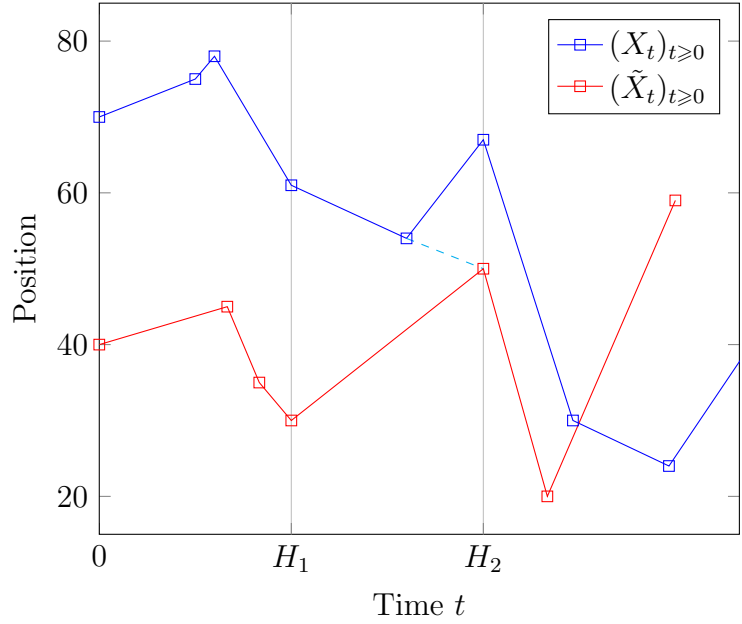

\begin{figure}[!h]
\begin{tikzpicture}
\begin{axis}[
    xlabel={Time $t$},
    ylabel={Position},
    xmin=0, xmax=100,
    ymin=15, ymax=85,
    xtick={0},
    ytick={0,20,40,60,80,100,120},
    legend pos=north east,
    % ymajorgrids=true,
    extra x ticks={30,60},
    extra x tick
    labels={$H_1$, $H_2$},
    extra x tick style={grid=major}
]
 
\addplot[
    color=blue,
    mark=square,
    ]
    coordinates {
    (0,70)(15,75)(18,78)(30,61)(60,78)(74,30)(89,24)(101,39)
    };
    % \legend{}
\addplot[
    color=red,
    mark=square,
    ]
    coordinates {
    (0,40)(20,45)(25,35)(30,30)(60,40)(70,25)(90,59)
    };
    \legend{$(X_t)_{t \geq 0}$,$(\tX_t)_{t \geq 0}$};
  \end{axis}
\end{tikzpicture}
\caption{Even if none of the process bounces between time $H_1$ and $H_2$, the coupling may also fail if the Gaussian coupling of the velocities at time $H_1$ fails.}\label{coupl3}
\end{figure}
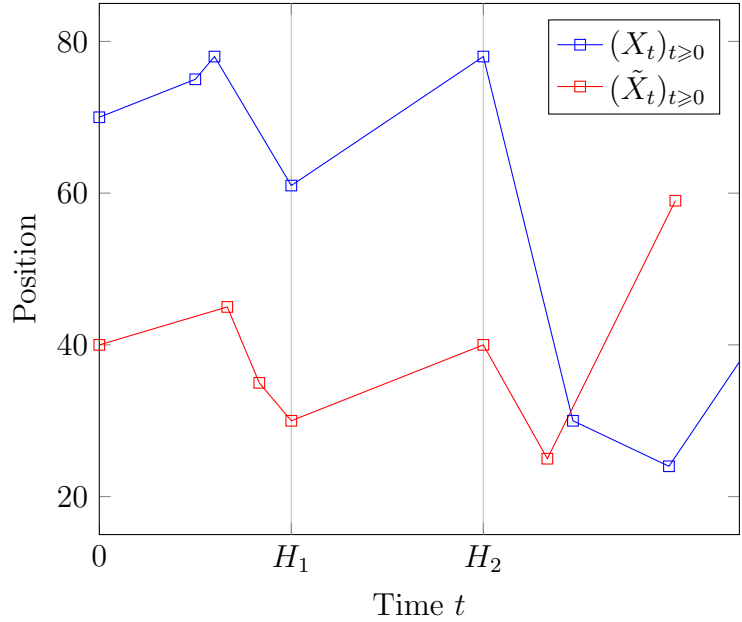
  
  More precisely, the coupling we consider is defined as follows. Let
  $(\nE_i, \nF_i,$ $\bG_{i})_{i \in \nset^*}$ be \iid~random
  variables, where for all $i \in \nset^*$, $\nE_i, \nF_i$
  are independent exponential random variables with parameter $1$ and $\bG_i$ has
  distribution $\loiy$ and is independent from $\nE_i, \nF_i$. In addition, let $G$ be a standard $d$-dimensional Gaussian random variable and $(W_t)_{t \geq 0}$ be a $d$-dimensional standard Brownian motion such that $G$, $(W_t)_{t \geq 0}$ and   $(\nE_i, \nF_i,\bG_{i})_{i \in \nset^*}$ are independent. 

  Set
  $ (X_0, Y_0) = (x,y)$, $ (\tX_0, \tY_0) = (\tx,\ty)$,
  $\S_0 =0$, $H_0=0$, $N_0=0$, $H_1=\nE_1/\rate$ and
  $N_1=1$. The process and its jump times are defined by recursion. Assume that $S_n, N_{n+1}, H_{n+1}$ and $(X_t,Y_t,\tX_t,\tY_t)_{t \in \ccint{0,S_n}}$ have been defined for some $n\in\N$. We  distinguish two cases.
  \begin{enumerate}[label=(\Alph*),wide, labelwidth=!, labelindent=0pt]
  \item If $N_{n+1} = 1$. Define
    \begin{align*}
  T_{n+1}^{(1)} &= \inf\ensemble{t \geq 0}{\int_0^t \defEns{\ps{Y_{\S_n}}{\nabla U (X_{\S_n}+sY_{\S_n})}_+} \rmd s \geq \nF_{n+1}} \eqsp,\\
      \tT_{n+1}^{(1)} &= \inf\Big\{t \geq 0 \, :  \int_0^t \defEns{\ps{\tY_{\S_n}}{\nabla U (\tX_{\S_n}+s\tY_{\bS_n})}_+} \rmd s \geq  \nF_{n+1}  \Big \}\eqsp,\\
T_{n+1} &= H_{n+1} \wedge T_{n+1}^{(1)} \wedge     \tT_{n+1}^{(1)} \eqsp.
\end{align*}
Set $S_{n+1} = S_n + T_{n+1}$, for all
$t \in \coint{S_n,S_{n+1}}$,
$(X_t,Y_t) = \phi_t(X_{S_n}, Y_{S_n})$,
$X_{S_{n+1}} = X_{S_n} + T_{n+1} Y_{S_n} $,
$(\tX_t,\tY_t) = \phi_t(\tX_{S_n}, \tY_{S_n})$,
$\tX_{S_{n+1}} = \tX_{S_n} + T_{n+1} \tY_{S_n} $. If $T_{n+1} = \bH_{n+1} $, consider the two
  random variables $G^{(1)},G^{(2)}$ defined by
  \Cref{LemMiroir}, associated with $(W_t)_{t \geq 0}$ and $G$, and for $x^{(1)} = X_{S_{n+1}}$,
  $x^{(2)} = \tX_{S_{n+1}}$, $\Sigma_{\mathrm{R}} = \nE_2 \Sigma / \rate$, and $M \geq 0$.
\begin{equation*}
  \text{ Still if $T_{n+1} = H_{n+1} $, \, set } 
  \begin{cases}
&   Y_{S_{n+1}} = \Sigma^{1/2}G^{(1)}  \eqsp, \,    \tY_{S_{n+1}} = \Sigma^{1/2} G^{(2)} \\
&    N_{n+2} = 2 \eqsp, \,
    H_{n+2}= E_{N_{n+2}}/\rate \eqsp .
  \end{cases}
\end{equation*}

Otherwise set $N_{n+2} = N_{n+1}
$, $ H_{n+2}= H_{n+1} - T_{n+1}$ and 
\begin{align*}
  &\text{ if $T_{n+1} =    T_{n+1}^{(1)} =    \tT_{n+1}^{(1)}$, \,} 
   Y_{\S_{n+1}} = \Refl(X_{\bS_n} + T_{n+1}
    Y_{S_n} , Y_{S_n}) \eqsp, \,  \\
  & \phantom{\text{ if $T_{n+1} =    T_{n+1}^{(1)} =    \tT_{n+1}^{(1)}$, \,} } \tY_{\S_{n+1}} = \Refl(\tX_{\S_n} + T_{n+1}
    \tY_{S_n} , \tY_{S_n})
    \eqsp,\\
 & \text{ if $T_{n+1} =    T_{n+1}^{(1)} <   \tT_{n+1}^{(1)}$, \,} 
   Y_{\S_{n+1}} = \Refl(X_{\S_n}  + T_{n+1}
Y_{\S_n}  , Y_{\S_n}) \eqsp, \,    \tY_{\S_{n+1}} =  \tY_{\S_{n}}
    \eqsp,\\
&  \text{ if $T_{n+1} =    \tT_{n+1}^{(1)} <   T_{n+1}^{(1)}$, \,} 
   \tY_{\S_{n+1}} = \Refl(\tX_{\S_n} + T_{n+1}
   \tY_{\S_{n}},    \tY_{\S_{n}}) \eqsp, \,    Y_{\S_{n+1}} =    Y_{\S_{n}} \eqsp,
\end{align*}
where $\Refl$ is defined by \eqref{eq:definition_Refl}.

\item If $N_{n+1} \geq 2$.
  Define
      \begin{align*}
  T_{n+1}^{(1)} &= \inf\ensemble{t \geq 0}{\int_0^t \defEns{\ps{Y_{\S_n}}{\nabla U (X_{\S_n}+sY_{\S_n})}_+} \rmd s \geq \nF_{n+1}} \eqsp,\\
      \tT_{n+1}^{(1)} &= \inf\Big\{t \geq 0 \, :  \int_0^t \defEns{\ps{\tY_{\S_n}}{\nabla U (\tX_{\S_n}+s\tY_{\bS_n})}_+} \rmd s \geq  \nF_{n+1}  \Big \}\eqsp,\\
T_{n+1} &= H_{n+1} \wedge T_{n+1}^{(1)} \wedge     \tT_{n+1}^{(1)} \eqsp.
\end{align*}
Set $S_{n+1} = S_n + T_{n+1}$, for all
$t \in \coint{S_n,S_{n+1}}$,
$(X_t,Y_t) = \phi_t(X_{S_n}, Y_{S_n})$,
$X_{S_{n+1}} = X_{S_n} + T_{n+1} Y_{S_n} $,
$(\tX_t,\tY_t) = \phi_t(\tX_{S_n}, \tY_{S_n})$,
$\tX_{S_{n+1}} = \tX_{S_n} + T_{n+1} \tY_{S_n} $
 and
\begin{equation*}
  \text{ if $\tT_{n+1} = H_{n+1} $, \,} 
  \begin{cases}
&   Y_{S_{n+1}} = \bG_{n+1}  \eqsp, \,    \tY_{S_{n+1}} = \bG_{n+1} \\
&    N_{n+2} = N_{n+1}+1 \eqsp, \,
    H_{n+2}= E_{N_{n+2}}/\rate \eqsp,
  \end{cases}
\end{equation*}

Otherwise set $N_{n+2} = N_{n+1}
$, $ H_{n+2}= H_{n+1} - T_{n+1}$ and
\begin{align*}
  &\text{ if $T_{n+1} =    T_{n+1}^{(1)} =    \tT_{n+1}^{(1)}$, \,} 
   Y_{\S_{n+1}} = \Refl(X_{\bS_n} + T_{n+1}
    Y_{S_n} , Y_{S_n}) \eqsp, \,  \\
  & \phantom{\text{ if $T_{n+1} =    T_{n+1}^{(1)} =    \tT_{n+1}^{(1)}$, \,} } \tY_{\S_{n+1}} = \Refl(\tX_{\S_n} + T_{n+1}
    \tY_{S_n} , \tY_{S_n})
    \eqsp,\\
 & \text{ if $T_{n+1} =    T_{n+1}^{(1)} <   \tT_{n+1}^{(1)}$, \,} 
   Y_{\S_{n+1}} = \Refl(X_{\S_n}  + T_{n+1}
Y_{\S_n}  , Y_{\S_n}) \eqsp, \,    \tY_{\S_{n+1}} =  \tY_{\S_{n}}
    \eqsp,\\
&  \text{ if $T_{n+1} =    \tT_{n+1}^{(1)} <   T_{n+1}^{(1)}$, \,} 
   \tY_{\S_{n+1}} = \Refl(\tX_{\S_n} + T_{n+1}
   \tY_{\S_{n}},    \tY_{\S_{n}}) \eqsp, \,    Y_{\S_{n+1}} =    Y_{\S_{n}} \eqsp,
\end{align*}
\end{enumerate}

For $t \geq \sup_{n \in \nsets} S_n$, set
$(X_t,Y_t) = (\tX_t,\tY_t)= \infty$.   %Denote by
%$(\bmcf_n)_{n \geq 1}$, the filtration associated with
%$(E_i,F_i,\bG_{i})_{i \in
%  \nset^*}$. %  and $(\bmcf_t)_{t \geq 0}$ the filtration associated with
%% $(X_t,Y_t)_{t \geq 0}$, $(\tX_t,\tY_t)_{t \geq 0}$.
%Remark that, since the conditional distribution of
%$(G^{(1)},G^{(2)})$ given
%$(E_i,F_i,$ $ \bG_{i})_{i \in
%  \nset^*}$ depends on $E_2$, $(X_t,Y_t,\tX_t,\tY_t)_{t \geq 0}$ is not Markovian.
%However, in the construction of the two processes $(X_t,Y_t)_{t \geq 0}$,
%$(\tX_t,\tY_t)_{t \geq 0}$, note that $(W_t)_{t \geq 0}$ and $G$ are independent from
%$(E_i,F_i,\bG_{i})_{i \in
%  \nset^*}$. By \Cref{LemMiroir}, we have that conditionally to
%$(E_i,F_i,\bG_{i})_{i \in
%  \nset^*}$ is a standard $d$-dimensional Gaussian random
%variable. Therefore from their definitions and
%\cite[\Cref{prop:superposition}]{DurmusGuillinMonmarche:toolbox} (see also  \cite[\Cref{lem:verif-synchrone}]{DurmusGuillinMonmarche:toolbox} and its proof for details), marginally, $(X_t,Y_t)_{t \geq 0}$ and
%$(\tX_t,\tY_t)_{t \geq 0}$ are two BPS processes starting from $(x,y)$
%and $(\tx,\ty)$. 
Remark that, since the conditional distribution of
$(G^{(1)},G^{(2)})$ given
$(E_i,F_i,$ $ \bG_{i})_{i \in
  \nset^*}$ depends on $E_2$, $(X_t,Y_t,\tX_t,\tY_t)_{t \geq 0}$ is not Markovian.   
However, according to \Cref{LemMiroir}, conditionally to
$(E_i, (F_{i,j})_{j\in\nset^*},G_{i})_{i \in \nset^*}$, $G^{(1)}$
and $G^{(2)}$ are both $d$-dimensional standard Gaussian random variables. 
As a consequence, from \cite[\Cref{prop:superposition}]{DurmusGuillinMonmarche:toolbox}, 
 marginally, $(X_t,Y_t)_{t \geq 0}$ and
$(\tX_t,\tY_t)_{t \geq 0}$ are two BPS processes starting from $(x,y)$
and $(\tx,\ty)$.  

Further, from the construction of the two processes, for all
$n \in \nset$ if
$(X_{S_n},Y_{S_n})=(\tX_{S_n},\tY_{S_n})$,
then $(X_t,Y_t)=(\tX_t,\tY_t)$ for all $t >S_n$. Besides,
consider $\tau = \inf\{ n \in \nset \, : \, N_{n+2} = 2\}$. Then by definition, if $T_{\tau+2} = H_{\tau+2}$ and $X_{S_{\tau+1}} + E_2 G^{(1)}/\rate = \tX_{S_{\tau+1}} + E_2 G^{(2)}/\rate$, then $(X_{S_{\tau+2}},Y_{S_{\tau+2}}) = (\tX_{S_{\tau+2}},\tY_{S_{\tau+2}})$. Finally, by definition of $\tau$, $T_{\tau+1} = H_{\tau+1}$ implies $S_{\tau+1} = E_1/\rate$ and if in addition  $T_{\tau+2} = H_{\tau+2}$, we have that $S_{\tau+2} = S$ with $S = (E_1+E_2)/\rate$. Based on these three observations, we get
for all $t >0$,
\begin{align}
\nonumber
&  \proba{(X_t,Y_t)=(\tX_t,\tY_t)}  \\
\nonumber
&\geq \proba{t \geq S_{\tau+2}, T_{\tau+2} = H_{\tau+2}, X_{S_{\tau+1}} + \frac{E_2 \Sigma^{1/2} G^{(1)}}{\rate} = \tX_{S_{\tau+1}} + \frac{E_2 \Sigma^{1/2} G^{(2)}}{\rate}} \\
\label{eq:minration_proba_1}
&  \geq  \proba{\msa \cap \{t \geq  S\}\cap\{ X_{E_1/\rate} + \frac{E_2 \Sigma^{1/2}G^{(1)}}{\rate} = \tX_{E_1/\rate} + \frac{E_2 \Sigma^{1/2} G^{(2)}}{\rate}\}} \eqsp.
% &  \geq  \proba{\msa \cap \{t \geq  (E_1/\rate+E_2)/\rate\}\cap\{ X_{S_{\tau+1}} + E_2 G^{(1)}/\rate = \tX_{S_{\tau+1}} + E_2 G^{(2)}/\rate\}} \eqsp,
\end{align}
where $\msa= \msa_1 \cap \msa_2$,
\begin{align*}
\msa_1 &= \defEns{\int_0^{E_2/\rate} \defEns{\ps{Y_{E_1/\rate}}{\nabla U (X_{E_1/\rate}+sY_{E_1/\rate})}_+} \rmd s \geq  F_{\tau+2}} \eqsp,\\
\msa_2&= \defEns{\int_0^{E_2/\rate} \defEns{\ps{\tY_{E_1/\rate}}{\nabla U (\tX_{E_1/\rate}+s\tY_{E_1/\rate})}_+} \rmd s \geq  F_{\tau+2} } \eqsp.
\end{align*}

Since for all $n \in \{1,\ldots,\tau\}$, $T_{n+1} = T^{(1)}_{n+1} \wedge \tT{(1)}_{n+1}$, $\norm{Y_{S_n}} = \norm{y}$, $\normLigne{\tY_{S_n}} = \norm{\ty}$, so for all $s \in  \ccint{0,E_1/\rate}$, 
\begin{equation}
  \label{eq:majoration_X_s_coupling}
\begin{aligned}
  \normLigne{X_s} &\leq \norm{x} + (E_1/\rate) \norm{y} \leq (1+E_1/\rate)R_{\msk} \eqsp, \\
  \normLigne{\tX_s} & \leq (1+E_1/\rate)R_{\msk}   \eqsp.
\end{aligned}
\end{equation}
For $i=1,2$, by the definition \eqref{eq:def_tilde_M} of $\tilde{M}$, we obtain that 
\begin{equation*}
\msb = \bigcap_{i=1}^2 \{\normLigne{G^{(i)}-(\Sigma^{1/2}/2)(X_{E_1/\rate}-\tX_{E_1/\rate})} \leq M \} \subset \bigcap_{i=1}^2 \{\normLigne{G^{(i)}} \leq \tilde{M}\}  \eqsp.
\end{equation*}
Using that by definition, $S_{\tau+1} = E_1/\rate$, $N_{\tau+1} =1$, so $Y_{S_{\tau+1}} =   \Sigma^{1/2}G^{(1)}$ and  $Y_{S_{\tau+1}} =   \Sigma^{1/2}G^{(1)}$, we get that $\msa_1\cap\msa_2\cap\msb \subset \tmsa$ where
\begin{equation*}
\tmsa =  \defEns{(E_2/\rate)\tM \sup_{z \in \ball{0}{(1+E_1/\rate)R_{\msk}+(E_2/\rate)M}} \norm{\nabla U(z)}  \geq  F_{\tau+2}} \eqsp.
% & \subset \tilde{\msa}_2 = 
% \defEns{(E_2/\rate)\tilde{M} \sup_{z \in \ball{0}{(1+E_1/\rate)R_{\msk}+(E_2/\rate)\tilde{M}}} \norm{\nabla U(z)}  \geq  \bE^{2}_{\tau+2}} \eqsp.
\end{equation*}

Then, we get by \eqref{eq:minration_proba_1}
\begin{align}
\nonumber
&  \proba{(X_t,Y_t)=(\tX_t,\tY_t)}  \\
\nonumber
&  \geq  \proba{\tilde{\msa}  \cap\{t \geq  S\}\cap\Big\{ X_{E_1/\rate} + \frac{E_2 \Sigma^{1/2}G^{(1)}}{\rate} = \tX_{E_1/\rate} + \frac{E_2 \Sigma^{1/2} G^{(2)}}{\rate}\Big\}} \eqsp.
% &  \geq  \proba{\msa \cap \{t \geq  (E_1/\rate+E_2)/\rate\}\cap\{ X_{S_{\tau+1}} + E_2 G^{(1)}/\rate = \tX_{S_{\tau+1}} + E_2 G^{(2)}/\rate\}} \eqsp,
\end{align}
Denoting by
$(\bmcf_n)_{n \geq 1}$ the filtration associated with
$(E_i,F_i,\bG_{i})_{i \in
  \nset^*}$, conditioning on $\bmcf_{\tau+1}$ and $E_2$ and  using that
$F_{\tau+2}$ is independent from $G^{(1)}$, $G^{(2)}$ $E_2$ and
$\bmcf_{\tau+1}$, the definition of $G^{(1)},G^{(2)}$ conditionally to $E_2$ and $\bmcf_{\tau+1}$, \Cref{LemMiroir}  and since $S=(E_1+E_2)/\rate$ by definition, we have 
\begin{align*}
&  \proba{(X_t,Y_t)=(\tX_t,\tY_t)}  \\
& \geq  \expe{\1_{\ccint{0,t}}\defEns{\frac{E_1+E_2}{\rate}} \talpha\parenthese{\frac{\norm{\Sigma^{-1/2}(X_{E_1/\rate}-\tX_{E_1/\rate})}\rate}{E_2},M} g(\frac{E_2}{\rate}) }
\end{align*}
Combining this result with  \eqref{eq:majoration_X_s_coupling} concludes the proof.
\end{proof}

Consider the more general case where $\loiy$ is rotation invariant and not singular with respect to the Lebesgue measure on $\rset^d$.
% It means there exist $r,\delta,c>0$ such that, denoting $\nu_{r,\delta}$ the uniform law on $\{y\in \rset^d,\ r < \norm{y} < r+ \delta\}$, then $\loiy \geqslant c \nu_{r,\delta}$.
The previous proof may be adapted to this case but the result is less explicit. % However, the estimates we can obtain this way are less explicit than in the Gaussian case, so that the gain with respect to a controllability argument in the spirit of \cite[Lemma 5.2]{MonmarcheRTP} or \cite[Lemma 2]{Doucet2017} is minor.

\begin{lemma}
  \label{lem:non_gaussian_coupling}
  Assume for all $\msa \in \mcb{\rset^d}$,
  \begin{equation}
    \label{eq:hyp_lem:non_gaussian_coupling}
   \loiy(\msa) \geqslant c \nu_{r,\delta}(\msa) \eqsp, 
  \end{equation}
 for some $r,\delta,c>0$, where $\nu_{r,\delta}$ the uniform law on $\{y\in \rset^d,\ r < \norm{y} < r+ \delta\}$. Let $\msk \subset \rset^d$, be a compact set.  Then there exists two random variables $G^{(1)},G^{(2)}$ with distribution $\loiy$, $t_0 \geq 0$, $\varepsilon >0$ such that for $s \geq t_0$, there exists $M \geq 0$ satisfying for all  $x,\tx \in \msk$,
  \begin{equation*}
    \proba{x + s G^{(1)} = \tx + s G^{(2)}, \norm{G^{(1)}-(x-\tx)/2} \leqslant M} \geq \varepsilon \eqsp.
  \end{equation*}
\end{lemma}
\begin{proof}
  Let $x,\tilde x \in\msk \subset \ball{0}{R_\msk}$, $R_\msk \geq 0$. If $s>\norm{x-\tilde x}/(2(r+\delta))$ and $M \geq R_\msk + s(r+\delta)$,  then
$\msi(x,\tilde x, s) = \{w \in \rset^d, \: \, \norm{w} \leq M \} \cap \{w\in\rset^d\, : \,  sr < \norm{w-x} < s(r + \delta) \} \cap \{w\in\rset^d\, : \,  sr < \norm{w-\tilde x} < s( r + \delta) \} \neq \emptyset$.
Writing $\bar \nu_{x,s}$ the law of $x + s G$ where $G$ has law $\loiy$, then for all $\msa \in \mcb{\rset^d}$, by \eqref{eq:hyp_lem:non_gaussian_coupling}, there exists $\tilde{c} >0$ such that
\begin{equation}
  \label{eq:lem:non_gaussian_coupling_1}
  \bar \nu_{x,s}(\msa) \wedge \bar \nu_{\tilde x,s}(\msa) \ \geqslant \tilde{c } \, \Leb\parenthese{\msa \cap  \msi(x,\tilde x,s)} \eqsp.
\end{equation}
Besides, (see \eg~\cite{pitman:1976} or \cite{rosenthal:1995}), we can construct a pair $(G_1,G_2)$ of random variables with both $G_1$ and $G_2$ distributed according to $\loiy$, and such that $\mathbb P \po x + s G = \tilde x + s G\pf = \bar \nu_{x,s}(\msa) \wedge \bar \nu_{\tilde x,s}(\msa)$. Combining this result with \eqref{eq:lem:non_gaussian_coupling_1},  the fact the function in the right hand side of \eqref{eq:lem:non_gaussian_coupling_1} is positive and depends continuously of $x$ and $\tilde x$, hence is lower bounded on $\msk$, concludes.
\end{proof}

\begin{lemma} 
Assume \Cref{assum:hyp_base} and \eqref{eq:hyp_lem:non_gaussian_coupling} for some $r,\delta,c>0$, where $\nu_{r,\delta}$ the uniform law on $\{y\in \rset^d,\ r < \norm{y} < r+ \delta\}$. Then, for all compact set $\msk$ of $\R^{d}\times \msy$, there exists $t_0,\alpha>0$ such that for all $(x,y), (\tx,\ty) \in\msk$ and all $t\geqslant t_0$, %denoting by $\delta_{z}$ the Dirac mass at $z$, 
\begin{equation*}%\label{EqCouplageCompact}
 \tvnorm{ P_t((x,y),\cdot) -  P_t((\tx,\ty),\cdot)} \leqslant  2(1 - \alpha) \eqsp.
\end{equation*}
\end{lemma}
\begin{proof}
  The proof is exactly similar to the proof of \Cref{LemCouplageCompact}. %, hence we will only focus on the specificities.
   Indeed it suffices to consider a coupling of two BPS $(X_t,Y_t)_{ t \geq 0} $ and $(\tX_t,\tY_t)_{t \geq 0}$ defined similarly to the processes defined in the proof of  \Cref{LemCouplageCompact} but $G^{(1)},G^{(2)}$ are chosen according to \Cref{lem:non_gaussian_coupling} in place of \Cref{LemMiroir}.
%   , but we need to define  which are refreshed at the same times $T_1$ and $T_2$. We want to construct two non-independent r.v. $G_1$ and $G_2$ such that, conditionally to $(X_{T_1},\widetilde X_{T_1},T_2-T_1)$, then
% \begin{eqnarray}\label{eq:couplage-non-gaussien}
%  X_{T_1} +   (T_2-T_1) G_1 & =& \widetilde X_{T_1} + (T_2-T_1) G_2
% \end{eqnarray}
% with positive probability. Then we set $Y_{T_1} = G_1$ and $\widetilde Y_{T_1} = G_2$ so that, provided the coupling of $G_1$ and $G_2$ is a success (by which we mean that \eqref{eq:couplage-non-gaussien} holds) and provided the two processes do not jump between times $T_1$ and $T_2$, then at time $T_2$ they have merged and the coupling is a success. The only difference from \Cref{LemCouplageCompact} is that we don't want $G_1$ and $G_2$ to be Gaussian r.v., but to have the law $\loiy$.
\end{proof}

Finally, let us detail \Cref{lem:condition-doeblin}, in prevision of the low-temperature study of \Cref{sec:metast-regime-anne}.
\begin{lemma} 
\label{lem:condition-doeblin-precis}
Assume \Cref{assum:hyp_base}. Then, for all compact set $\msk \subset \R^{d}\times \msy$, there exist $t_0,\varepsilon,C,R>0$, which depend on $\msk$, $\loiy$ and $\rate$ but not on $U$, such that for all  $(x,y),(\tx,\ty)\in\msk$ and all $t\geqslant t_0$,
\begin{equation*}
  \tvnorm{ P_t((x,y),\cdot) -  P_t((\tx,\ty),\cdot)}  \leq  2 \parentheseDeux{1 - \varepsilon \exp \parenthese{ - C \norm{\nabla U}_{\infty, \ball{0}{R}} }}\eqsp.
\end{equation*}
% %denoting by $\delta_{z}$ the Dirac mass at $z$, 
% \begin{eqnarray*}%\label{EqCouplageCompact}
% \frac12 \tvnorm{ P_t((x,y),\cdot) -  P_t((\tx,\ty),\cdot)} & \leqslant & 1 - \alpha \exp \po - C \norm{\nabla U}_{\infty, \ball{0}{R}}\pf.
% \end{eqnarray*}
\end{lemma}
\begin{proof}
In the case where $\loiy$ is a Gaussian distribution, the proof follows from the statement of \Cref{LemCouplageCompact}. In the general case, we only give a sketch of proof, since this is a direct adaptation of \cite[Theorem 5.1]{MonmarcheRTP}. First, in the spirit of the proof of  \Cref{LemCouplageCompact} or of \cite[Lemma 5.2]{MonmarcheRTP}, we study a BPS with no potential, i.e. with $U=0$, and we show that  we may couple them so that, with some probability $\alpha>0$, they merge in a given time $t_0$, without leaving a given compact set. Then we add independent bounces, and say that the coupling is still a success if no bounce happens before time $t_0$, which gives the desired dependency with respect to $U$.
\end{proof}

%%% Local Variables:
%%% mode: latex
%%% TeX-master: "main"
%%% End:

 \subsection{Proof of \Cref{theo:geo_ergo_gene}}
\label{sec:proof_main_theo}

%The existence of a Lyapunov function gives us fast return to the compact $\mathcal K_1 = \{ V \leq 2C\}$, where $C$ is given in Lemma \ref{LemLyapunov}, and, starting from two different points of $\mathcal K_1$, we are able to merge two BPS in a given time with a probability independent from the starting points. It is well-known that in this situation, the convergence to equilibrium is geometric. In fact, the proofs of this fact usually start from a different assumption which is, alongside the existence of a Lyapunov function, a Doeblin condition, rather than an estimate like \eqref{EqCouplageCompact}. For this reason, and in order to keep track of the constants which are involved, we give here a short proof for geometric convergence given  \eqref{EqEspV} and \eqref{EqCouplageCompact}. It is an adaptation of a clever idea of Hairer and Mattingly (\cite{HairerMattingly2011}), which is to work with a well-chosen metric for which the semi-group is a contraction.
 
The proof follows from  \Cref{lem:lyapunov} and \Cref{lem:condition-doeblin}, and an application of \cite[Theorem 6.1]{meyn:tweedie:1993:III}. However, \cite[Theorem 6.1]{meyn:tweedie:1993:III} is non quantitative and for the proofs of \Cref{sec:metast-regime-anne} need  explicit bounds for the convergence of $(P_t)_{t \geq 0}$ to $\pi$. To this end, we give a quantitative version of \Cref{theo:geo_ergo_gene} in  \Cref{app:quantitative_bounds} based on \cite[Theorem 1.2]{HairerMattingly2011}.

\subsection{Proofs of \Cref{theo:V_geo_ergo_loi_borne}}
\label{sec:autre-demo_1}

In each case, we apply  \Cref{theo:geo_ergo_gene}. Set $H(t)=t^2$ for $t\in\rset$.
Consider $r>0$ such that $\delta = \mathbb P(|Y_1|>r) >0$ where $Y=(Y_1,\dots,Y_d)\in\msy$ is distributed according to $\loiy$. Note that \Cref{ass:geo_erg_general}-\ref{ass:geo_erg_general_eq_4} is automatically satisfied in all the cases. 

Under \Cref{ass:geo_ergo_1bis}, set $\bU(x) = U(x)$ and $\ell(x) = 1$ for all $x\in\rset^d$. All the conditions of \Cref{ass:geo_erg_general} are sastisfied and so is \eqref{eq:condition_geo_4}  by \Cref{rem:additio:ass} since $\lim_{\norm{x} \to \plusinfty} \norm{\nabla U(x) } = \plusinfty$. 

Under \Cref{ass:geo_ergo_2}, set $\bU(x) = U^\varsigma(x)$ and $\ell(x) = 1$ for any $x \in \rset^d$. Then \Cref{ass:geo_erg_general} is satisfied. In addition, \eqref{eq:condition_geo_4} holds by \Cref{rem:additio:ass} since  under  \Cref{ass:geo_ergo_2}
\begin{equation*}
  \lim_{\norm{x} \to \plusinfty} \{\ell(x) \norm{\nabla U(x)} / \norm{\nabla \bU(x) }\}= \plusinfty \eqsp.
  \end{equation*}

  Under \Cref{ass:geo_ergo_2_b}, set $\bU(x) = U^\varsigma(x)$ and $\ell(x) = 1/(1+\norm{\na \bU(x)})$ for all $x\in\rset^d$. All the conditions of \Cref{ass:geo_erg_general} are satisfied and  \eqref{eq:condition_geo_4} holds by \Cref{rem:additio:ass} since $ \lim_{\norm{x} \to \plusinfty} \ell(x) = 0$.

\subsection{Proof of \Cref{theo:V_geo_ergo_lambda_0}}
\label{sec:proof_lambda_0}
We apply  \Cref{theo:geo_ergo_gene} again.
 Set $H(t)=t^2$ for $t\in\rset$.
Consider $r>0$ such that $\delta = \mathbb P(|Y_1|>r) >0$ where $Y=(Y_1,\dots,Y_d)\in\msy$ is distributed according to $\loiy$. Note that \Cref{ass:geo_erg_general}-\ref{ass:geo_erg_general_eq_4} is automatically satisfied. 
Set $\bU(x) = U(x)$ and $\ell(x) = 1$ for any $x \in \rset^d$. Then, the conditions of \Cref{ass:geo_erg_general} hold with  $c_4$ arbitrarily small. Therefore,  \eqref{eq:condition_geo_4} is satisfied  if $\rate$ is small enough.

\subsection{Proof of \Cref{theo:V_geo_ergo_loi_non_borne}}
\label{sec:proof-crefth}

We apply  \Cref{theo:geo_ergo_gene}. Set $H(t) = \eta t^2$ for $\eta$ small enough such that \Cref{ass:geo_erg_general}-\ref{ass:geo_erg_general_eq_4} is satisfied.  Set  $\bU(x) = U^\varsigma(x)$ for any $x\in\rset^d$. Note that 
\begin{eqnarray*}
 \{\underset{y\in \msa_x}\sup\norm{y}^2\} \norm{\na^2 \bU(x)}    \leqslant  3 \eta^{-1} \bU(x) \norm{\na^2 \bU(x)} \\
 \leqslant  C U^\varsigma(x) \po  \norm{\na^2 U(x)} U^{\varsigma-1}(x) + \norm[2]{\na U(x)} U^{\varsigma-2}(x)\pf 
\end{eqnarray*}
for some $C>0$, hence is bounded. Then, the proof follows the same lines as the proof of \Cref{theo:V_geo_ergo_loi_borne} under \Cref{ass:geo_ergo_2}, and is omitted.

%%% Local Variables:
%%% mode: latex
%%% TeX-master: "main"
%%% End:

\section{Miscellaneous}
\label{sec:miscellaneous}
\subsection{A specific and explicit bound for a toy model}
\label{sec:prec-expl-bound}

Following carefully the proofs of  \Cref{theo:geo_ergo_gene}, it is possible to get explicit bounds on the values of $C,\rho>0$ such that \eqref{eq:def_unif_V_geo_ergo} holds. Nevertheless, the  obtained bounds are not sharp. In particular, in \Cref{sec:couplage}, when we try to couple two processes, we do not make any use of the potential $U$. In fact, at this step, $U$ only plays the role of an hindrance in the minorization condition given by \Cref{lem:condition-doeblin} based on \Cref{LemCouplageCompact}-\Cref{lem:condition-doeblin-precis}. We try to couple the processes using only the refreshment jumps, and hope that, during this attempt, no bounce occurs. We now illustrate  on a toy model how an analysis which is model specific can circumvent this flaw. It shows that the explicit bounds we obtain in \Cref{LemCouplageCompact} may be far from optimality for some problems.
%This is precisely at this stage that the bounds we obtain get bad for multi-scale problems, and we will now illustrate this on a toy model where a more careful analysis circumvent this flaw, thus proving that the explicit bounds we obtain in Lemma \ref{EqCouplageCompact} may be far from optimality in the multi-scale case.

Consider the smooth manifold $\msd = (\rset/ \zset ) \times (\rset/ \eta \zset )^{d-1}$ for $d \geq 2$ and $\eta >0$, and let $\projd:\rset^d \to \msd$ be the corresponding projection (also referred to as quotient map). We set in this section $\pi$ to be the uniform distribution on $\msd$, $\msy = \rset^d$  and $\loiy $ to be the zero-mean $d$-dimensional Gaussian distribution on $\rset^d$ with covariance matrix $\sigma^2 \Id$, $\sigma^2 >0$. In this setting, $U$ is simply the function which is identically equal to $0$ on $\msd$. A BPS sampler $(X_t,Y_t)_{t \geq 0}$ is defined as in \Cref{sec:presentation-bps} to target $\pi \otimes \loiy$. The construction is in all the respect the same, just by replacing the  state space $\rset^d \times \msy$ by $\msd \times \msy$ and setting $X_t = \projd(X_{S_n}+tY_{S_n})$ for $t \in \coint{S_n,S_{n+1}}$ in place of $X_t = X_{S_n}+tY_{S_n}$. To show the convergence of the corresponding semi-group $(P_t^{\msd})_{t\geq 0}$, we show a uniform Doeblin condition \cite[Chapter 16]{MeynTweedie} holds using a direct coupling argument. 

Note that $\msd$ has no boundary and therefore no reflexion has to be take care of but it is worthwhile to mention that by a deterministic transformation of this process from $\msd$ to  $\ccint{0,1}\times\ccint{0,\eta}^{d-1}$, we  end up with  the reflected PDMP process targeting the uniform distribution on $\ccint{0,1}\times\ccint{0,\eta}^{d-1}$ described in \cite{BBCDDFGV2017}.

The process that we consider in this section  can be seen as a toy model for  convex potentials.  If $\eta$ is small, which is the analogous of multi-scales problems, then the proof of \Cref{theo:geo_ergo_gene} would yield a mixing time of order $\eta^d$. Indeed, in \Cref{sec:couplage}, the coupling is considered a failure as soon as one of the processes bounce (or, here, is reflected at the boundary). Hence, a successful coupling would need that, at the first refreshment time, the new Gaussian velocity is directed mainly according to the first dimension, which is unlikely. As we will see, this is a too pessimistic bound.

\begin{prop}\label{PropTore}
For all $x,\tx\in \msd$, $y,\ty \in \rset^d$ and $t >0$,
\begin{multline*}
  \tvnorm{\updelta_{(x,y)} P_t^{\msd}- \updelta_{(\tx,\ty)} P_t^{\msd}}  \\
  \leq  2 \parentheseDeux{\proba{N_t\leq 1} + \expe{\1_{\ccint{2,\plusinfty}}(N_t) \defEns{1-2\Phibf\parenthese{\frac{(1+\eta^2(d-1))^{1/2}}{2(S_{N_t}-S_1)}}}}} \eqsp. 
  \end{multline*}
  where $\Phibf$ is the cumulative distribution function of the standard Gaussian distribution on $\rset$, $(N_t)_{t\geq 0}$ is a Poisson process with rate $\rate$ and jump times
  $(S_i)_{i\in\mathbb N}$, with $S_0=0$.  
\end{prop}
\begin{proof}
  Let $(N_t)_{t\geq 0}$ be a Poisson process with rate $\rate$ and jump times
  $(S_i)_{i\in\mathbb N}$, with $S_0=0$.  
  Set first for $t \in \coint{0,S_1}$, $X_t = \projd(x + t y)$,   $Y_t = y$, $X_{S_1} = \projd(x+S_1y)$, $\tX_t = \projd(\tx + t \ty)$,   $\tY_t = \ty$, $\tX_{S_1} = \projd(\tx+S_1\ty)$.
  By \cite[Section
  2]{lindvall:rogers:1986}, given $(S_i)_{i\in\mathbb N}$, there exist
  two Brownian motions $(W_t)_{t \geq 0}$ and $(\tW_t)_{t \geq 0}$ on $\msd$ such that for any $t >0$,
  \begin{multline}
    \label{eq:couplage_toy_model}
    \probaCond{X_{S_1} + W_t = \tX_{S_1} + \tW_t }{ (S_k)_{k \geq 0}} =     \probaCond{T_{\rmcc} \leq t }{(S_k)_{k \geq 0}} \\= 1-2 \Phibf\parenthese{-\left. \norm{X_{S_1} -\tX_{S_1}}\middle/ (2t^{1/2}) \right.} \eqsp, 
  \end{multline}
 and
  \begin{equation}
    \label{eq:def_coupling_time_toy}
    T_{\rmcc} = \inf\{s \geq 0 \, : \, X_{S_1} + W_s = \tX_{S_1} + \tW_s\} \eqsp.
  \end{equation}
  We can define then for any $i \in \nset^*$,
  \begin{equation}
    \label{eq:def_g_i_toy}
    \begin{aligned}
    G_i = (W_{(S_{i+1}-S_1)^2} - W_{(S_{i}-S_1)^2})/(S_{i+1}-S_i) \eqsp, \\ \tilde{G}_i = (\tW_{(S_{i+1}-S_1)^2} - \tW_{(S_{i}-S_1)^2})/(S_{i+1}-S_i)      \eqsp.
    \end{aligned}
  \end{equation}
  Note that by the Markov property of $(W_t)_{t \geq 0}$ and  $(\tW_t)_{t \geq 0}$, $(G_i)_{ i\in\nset^*}$ and $(\tG_i)_{ i\in\nset^*}$ are sequences of \iid~$d$-dimensional standard Gaussian random variables.

Define $Y_{S_i} = G_1$, $\tY_{S_1} = \tG_1$ and now  assume that $(X_t,Y_t)$, $(\tX,\tY_t)$ are defined for $t \in \ccint{0,S_k}$, $k \geq 1$.  
Set for $t \in \ccint{S_k,S_{k+1}}$, $X_t = \projd(X_{S_k} + (t-S_k) G_{k+1})$, $\tX_t =\projd( \tX_{S_k} + (t-S_k) \tG_{k+1})$, for $t \in \coint{S_k,S_{k+1}}$, $Y_{t} = Y_{S_{k}}$, $\tY_t = \tY_{S_k}$ and $Y_{S_{k+1}} = G_{k+1}$, $\tY_{S_{k+1}} = \tG_{k+1}$.
It follows then by construction that for any $t \geq 0$, $(X_t,Y_t)_{t \geq 0}$ is distributed according to $P_t^{\msd}((x,y),\cdot)$ and $(\tX_t,\tY_t)_{t \geq 0}$ is distributed according to $P_t^{\msd}((\tx,\ty),\cdot)$.
% $(X_t,Y_t)_{t \geq 0}$ and $(\tX_t,\tY_t)_{t \geq 0}$ are two BPS
% samplers with initial conditions $(x,y)$ and $(\tx,\ty)$ respectively.
Then it remains to bound $    \proba{(X_t,Y_t) = (\tX_t,\tY_t)}$ by definition of the total variation norm. 

Note that if $(S_{i+1}-S_1)^2 \geq (t-S_1)^2 \geq (S_i - S_1)^2 \geq T_{\rmcc} \geq (S_{i-1}-S_1)^2$, $i \geq 2$, we have by  \eqref{eq:def_coupling_time_toy}-\eqref{eq:def_g_i_toy} and construction $(X_t,Y_t) = (\tX_t,\tY_t)$.
  Therefore, we get $\{(S_{N_t}-S_1)^2 \geq T_{\rmcc} \} \cap \{N_t >1\} \subset \{(X_t,Y_t) = (\tX_t,\tY_t)\}$  and we obtain
  \begin{multline*}
    \proba{(X_t,Y_t) = (\tX_t,\tY_t)} \leq \proba{\{ S_{N_t} \leq S_1+T_{\rmcc} \} \cap \{N_t \leq 1\}} \\
    \leq \proba{N_t\leq 1} + \proba{\{N_t \geq 2\} \cap\{(S_{N_t}-S_1)^2 \geq T_{\rmcc}\} } \eqsp.
  \end{multline*}
The proof is then concluded by conditioning with respect to $(S_k)_{k \in \nset}$ using \eqref{eq:couplage_toy_model} and for any $x \in \msd$, $\norm{x} \leq (1+\eta^2(d-1))^{1/2}$.

\end{proof}

\begin{corollary}
  \label{coro:toy_model}
There exist $C \geq 0$ and $\varepsilon \in \ocint{0,1}$ independent of $d$ such that setting $t_{\rmcc} = Cd^{1/2}$, for all $x,\tx\in \msd$ and $y,\ty \in \rset^d$,
\begin{equation*}
  \tvnorm{\updelta_{(x,y)} P_{t_{\rmcc}}^{\msd}- \updelta_{(\tx,\ty)} P_{t_{\rmcc}}^{\msd}}  \leq (1-\varepsilon)  \eqsp. 
  \end{equation*}  
\end{corollary}
\begin{proof}
  By \Cref{PropTore} and using the same notations, for all $x,\tx\in \msd$, $y,\ty \in \rset^d$ and $t >0$, we have since for any $s\geq 0$, $1/2-\Phibf(-s) \leq 1\wedge\{s/(2\uppi)^{1/2}\}$,
  \begin{align*}
& 2^{-1}  \tvnorm{\updelta_{(x,y)} P_t^{\msd}- \updelta_{(\tx,\ty)} P_t^{\msd}}  \\
&\leq  \proba{S_2 \geq t/4} +\proba{S_2 \leq t/4, S_{N_t}-S_2 \leq t/2} \\
    &  \qquad + \expe{\1_{\ccint{0,t/4}}(S_2) \1_{\coint{t/2,\plusinfty}}( S_{N_t}-S_2) \defEns{1-2\Phibf\parenthese{\frac{(1+\eta^2(d-1))^{1/2}}{2(S_{N_t}-S_1)}}}}\\
& \leq \proba{S_2 \geq t/4} +\proba{ S_{N_t} \leq 3t/4}  + \frac{\{2(1+\eta^2(d-1))\}^{1/2}}{t\uppi^{1/2}}  \eqsp. 
  \end{align*}
  Since $\{S_{N_t} \leq 3t/4\} \subset \{ N_t - N_{3t/4} = 0 \}$, and $N_t - N_{3t/4}$ follows a Poisson distribution with parameter $t\lambda/4$, we get
for all $x,\tx\in \msd$, $y,\ty \in \rset^d$ and $t >0$
  \begin{equation*}
    2^{-1}  \tvnorm{\updelta_{(x,y)} P_t^{\msd}- \updelta_{(\tx,\ty)} P_t^{\msd}}  \leq \proba{S_2 \geq t/4} + \rme^{-\lambda t/4}  + \frac{\{2(1+\eta^2(d-1))\}^{1/2}}{t\uppi^{1/2}} \eqsp. 
  \end{equation*}
  The proof then follows from a straightforward computation. 
\end{proof}

A direct consequence of \Cref{coro:toy_model} is that, with the same notations, for all $\nu \in \mcp(\msd\times\rset^d)$  and $t \geq 0$,
\begin{equation*}
  \tvnorm{\nu P_{t}^{\msd}- \pi \otimes \loiy }  \leq (1-\varepsilon)^{\floor{t/t_{\rmcc}}}  \eqsp. 
  \end{equation*} 
As a conclusion, for the considered toy model, we get that the rate of convergence scales only as $d^{1/2}$. Note that this result is optimal since the process has unit constant speed and the diameter of $\msd$ is $d^{1/2}$.

\subsection{The metastable regime and annealing}
\label{sec:metast-regime-anne}
The simulated annealing me\-thodology (see \cite{haario:saksman:1991}
and references therein) aims at finding a global minimum of a function
$U$ and not sampling to the target distribution $\pi$ given by
\eqref{eq:density_pi}. However, roughly, these methods need to
approximately sample from the family of distributions
$\{\pi_{\upbeta} \, : \, \upbeta >0\}$, where $\pi_{\upbeta}$ is the
distribution on $\rset^d$ associated with the potential
$x \mapsto \upbeta U(x)$, for $\upbeta >0$. To do so, we will study in this
section a simulated annealing algorithm based on the BPS, extending
the results of \cite[Theorem 1.5]{MonmarcheRTP} on the torus $(\rset/\zset)^d$.  For the sake of simplicity, the study is restricted to
the following case:
\begin{assumption}
  \label{ass:recuit}
  \begin{enumerate}[label=(\roman*)]
  \item The potential $U \in \mrc^2(\rset^d)$ satisfies
    \begin{align*}
      \int_{\rset^d} \exp(- U(x)/2) \rmd x  &< \infty \eqsp, \qquad &\lim_{\norm{x} \to \plusinfty} U(x) &= \plusinfty, \\ \liminf_{\norm{x}\rightarrow\infty} \norm{\nabla U(x)} &>  0 \eqsp, \qquad  &  \sup_{x \in \rset^d} \norm{\nabla^2 U(x)} &< \infty \eqsp.
    \end{align*}
Moreover, without loss of generality, $U(0)=\min_{\rset^d} U = 0$.
\item $\msy = \ball{0}{M}$ for $M >0 $ and  the distribution $\loiy$ on $\msy$ is rotation invariant.
  \end{enumerate}
\end{assumption}

In the rest of this section, \Cref{ass:recuit} is enforced . However, note the arguments also work under \Cref{ass:geo_erg_general} (in particular when $Y=\rset^d$, $\loiy$ has a Gaussian moment and $U$ is a perturbation of an $\chi$-homogeneous potential with $\chi>1$, as in \Cref{propo:alpha_homogeneous}), which is not implied by \Cref{ass:recuit}. 

For a measurable function $\beta : \rset_+ \to \rset_+$, referred to in the following as the cooling schedule, we consider
in this section the simulated annealing BPS process
$(X_t^{(\beta)},Y_t^{(\beta)})$ defined as follows.  Consider some initial
point $(x,y) \in \rset^{d} \times \msy$, and the family of \iid~random
variables $(E_i, F_i, G_i)_{i \in \nset^*}$ introduced in
\Cref{sec:presentation-bps}. Let $\rate >0$, $(X^{(\beta)}_0, Y^{(\beta)}_0) = (x,y)$ and $S^{(\beta)}_0=0$. We define by
recursion the jump times of the process and the process itself. For
all $n \geq 0$, consider
\begin{align*}
T_{n+1}^{(1,\beta)} &=   E_{n+1}/ \rate \notag\\
%\label{eq:def-temps-bounce}
T_{n+1}^{(2,\beta)} &= \inf\ensemble{t \geq 0}{\int_0^t\left\{\beta(s) \ps{Y^{(\beta)}_{\S^{(\beta)}_n}}{\nabla U (X^{(\beta)}_{\S^{(\beta)}_n}+s Y^{(\beta)}_{\S^{(\beta)}_n})}_+ \right\} \rmd s \geq E^{2}_{n+1}} \\
T^{(\beta)}_{n+1} &= T_{n+1}^{(1,\beta)} \wedge T_{n+1}^{(2,\beta)} .\notag
\end{align*}
Set $S^{(\beta)}_{n+1} = S^{(\beta)}_n + T^{(\beta)}_{n+1}$, $(X^{(\beta)}_t,Y^{(\beta)}_t) = ( X^{(\beta)}_{\S^{(\beta)}_n} + t Y^{(\beta)}_{\S^{(\beta)}_n},Y^{(\beta)}_{\S^{(\beta)}_n})$,  for all $t \in \cointLigne{S^{(\beta)}_n,S^{(\beta)}_{n+1}}$,
$X^{(\beta)}_{S^{(\beta)}_{n+1}} = X^{(\beta)}_{\S^{(\beta)}_n} + T^{(\beta)}_{n+1} Y^{(\beta)}_{\S^{(\beta)}_n} $ and 
\begin{equation*}
 Y^{(\beta)}_{\S^{(\beta)}_{n+1}} = 
  \begin{cases}
     G_{n+1} &  \text{ if $T^{(\beta)}_{n+1} = T_{n+1}^{(1,\beta)}$} \\
 \Refl(X^{(\beta)}_{\S^{(\beta)}_{n+1}} , Y^{(\beta)}_{\S^{(\beta)}_n}) & \text{ otherwise} \eqsp,
  \end{cases}
\end{equation*}
where $\Refl$ is defined by 
\eqref{eq:definition_Refl}.
Note that under \Cref{ass:recuit}, $\msy$ is bounded and therefore by \cite[\Cref{prop:non-explosion_BPS}]{DurmusGuillinMonmarche:toolbox}, $ \sup_{n \in \nset} \S^{(\beta)}_n = \plusinfty$.

%\begin{lemma}[\Cref{prop:non-explosion_BPS}]
%  \label{lem:definition_true}
%  Almost surely, $\sup_{n \in \nset} S^{(\beta)}_n =
%  \plusinfty$. Therefore almost surely $(X^{(\beta)}_t,Y^{(\beta)}_t)_{t \geq 0}$ is a $(\rset^{d} \times \msy)$-valued \cadlag~process.
%\end{lemma}
%\begin{proof}
%The proof is postponed to \Cref{sec:proof-crefl}.
%\end{proof}
 Therefore almost surely $(X^{(\beta)}_t,Y^{(\beta)}_t)_{t \geq 0}$ is a $(\rset^{d} \times \msy)$-valued \cadlag~process. By \cite[Theorem
25.5]{davis:1993}, the BPS process $(X^{(\beta)}_t,Y^{(\beta)}_t)_{t \geq 0}$ defines
a non-homogeneous strong Markov semi-group $(P_t)_{t\geqslant 0}$ given for all
$s,t \in \rset_+$, $(x,y)\in \rset^d \times \msy$ and $\msa \in \mcb{\rset^d \times \msy}$ by
\begin{equation*}
%  \label{EqDefP_t}
P_{t,t+s}^{(\beta)} ((x,y),\msa)  =  \proba{(X^{(\beta)}_{s},Y^{(\beta)}_{s}) \in \msa} \eqsp, 
\end{equation*}
where $\sequenceD{X^{(\beta)}_u,Y^{(\beta)}_u}[u][ \rset_+]$ is the annealed BPS process started
from $(x,y)$ and cooling schedule $s \mapsto \beta(t+s)$. 
As it is usual in simulated annealing if $t \mapsto \beta(t)$ goes to infinity sufficiently slowly for the process $(X^{(\beta)}_t,Y^{(\beta)}_t)$ to approach its instantaneous equilibrium $\exp(-\beta(t) U) \otimes \loiy$, then $X^{(\beta)}_t$ should be close to a global minimum of $U$ with high probability. %The
%following result establishes a sufficient condition for this to hold:
\begin{assumption}
   \label{ass:recuit-beta}
  The function $t \mapsto \beta(t)$ is increasing,  satisfies $\lim_{t \to \plusinfty} \beta(t) = \plusinfty$, $\beta(0) \geq 1$ and  there exist $s_0,D_1,D_2>0$ with $ D_1\geqslant D_2$ such that for all $t$ large enough, $\beta(t) \geqslant D_2 \ln t$ and $\beta(t+s_0) - \beta(t)  \leq D_1/t$.
\end{assumption}
% \begin{assumption}
%  \label{ass:recuit-beta}
%  The cooling schedule  $t\rightarrow \beta_t \geqslant 1$ is increasing, goes to infinity at infinity, and there exist $s_0>0$ and $H_1\geqslant H_2 > 0$ such that for all $t$ large enough, $\beta_t \geqslant H_2 \ln t$ and
% \[  \beta_{t+s_0} - \beta_{t} \ \leqslant \ \frac{H_1}{t}\eqsp.\]
% \end{assumption}

We can then adapt well-known techniques from the simulated annealing literature to extend the result of \cite{MonmarcheRTP} which restricts its study to the torus $(\rset/\zset)^d$. A crucial step is to show that for fixed $s,t \geq 0$, $s \leq t$, the Markov kernel $P_{s,t}^{\beta}$ is a contraction in an appropriate metric with constants which have to be explicit in $s,t$ and the cooling schedule $\beta$. However, using our approach for the proof of the geometric ergodicity of BPS, we were able to complete such a task. 
\begin{theorem}\label{thm:ThmRecuit}
Assume  \Cref{ass:recuit}. There exists $\theta >0$ such that if \Cref{ass:recuit-beta} holds with  $D_1 \leq \theta^{-1}$,
then
 for any $(x,y) \in \rset^d \times \msy$  and  any levels $\eta>\eta'>0$, there exists  $A>0$ such that, for all $t>0$, 
\begin{equation*}
\proba{ U(X^{(\beta)}_t) > \eta + \min_{\rset^d} U }  \leq A\exp(U(x)/2)/t^p \eqsp,
\end{equation*}
where $p = ( 1 - \theta D_1 ) \wedge( D_2\eta')>0$ and $(X^{(\beta)}_t,Y_t^{(\beta)})$ is the annealed BPS process starting from $(x,y)$.
\end{theorem}

  The proof is postponed to  Appendix A.

%%% Local Variables:
%%% mode: latex
%%% TeX-master: "main"
%%% End:

 \section*{Acknowledgements}
Alain Durmus acknowledges support from Chaire BayeScale ``P. Laffitte". Pierre Monmarch\'e acknowledges support from the French
ANR project ANR-12-JS01-0006 - PIECE and the European Research Council under the European Union's Seventh Framework Programme (FP/2007-2013) / ERC Grant Agreement number 614492. Arnaud Guillin and Pierre Monmarch\'e acknowledge support from the French
ANR-17-CE40-0030 - EFI - Entropy, flows, inequalities.
\bibliographystyle{plain}
\bibliography{../Bibliography/biblioBouncy}

  \appendix

%\section{Postponed proofs}

\section*{Appendix A. Postponed proofs}

%\subsection{Proof of \Cref{propo:alpha_homogeneous}} 

\subsection*{Proof of \Cref{propo:alpha_homogeneous}}%\label{sec:proof-alpha_homogeneous}

Note that since for all $x \in \rset^d$, $\norm{x} \geq 1$,
  \begin{equation}
  \label{propo:alpha_homogeneous_eq_0}    
    U_1(x) = \norm[\alpha]{x} U_1(x/\norm{x}) \eqsp,    
  \end{equation}
 that it is sufficient to show that there exists $C_1,C_2 >0$ such that for all $x \in \rset^d$, $\norm{x} \geq 1$ such that
  \begin{align}
  \label{propo:alpha_homogeneous_eq_1}    
    C_1 \norm[\alpha-1]{x} \leq \norm{\nabla U_1(x)} &\leq C_2 \norm[\alpha-1]{x} \\
      \label{propo:alpha_homogeneous_eq_2}    
    \norm{\nabla^2 U_1(x)} &\leq C_2 \norm[\alpha-2]{x}  \eqsp.
  \end{align}
  \eqref{propo:alpha_homogeneous_eq_1}  is just a consequence of \cite[Lemma 4.5]{jarner:hansen:2000}. As for   \eqref{propo:alpha_homogeneous_eq_2},  we have by \eqref{propo:alpha_homogeneous_eq_0}  for all $x \in \rset^d$, $\norm{x} \geq 1$,
  \begin{align*}
    & \nabla U_1(x) = \alpha \norm[\alpha-2]{x} x U_1(x/\norm{x}) + \norm[\alpha]{x}\defEns{\Id-x x^{\transpose}/\norm[2]{x}} \nabla U_1(x/\norm{x})\\
&    \nabla U_2(x) =  \alpha \defEns{\norm[\alpha-2]{x} + (\alpha-2) \norm[\alpha-4]x x^{\transpose}} U_1(x/\norm{x}) \\
    &   \qquad \qquad + \alpha \norm[\alpha-2]{x} \Big[\defEns{\Id-x x^{\transpose}/\norm[2]{x}} \nabla U_1(x/\norm{x}) x^{\transpose} \\
    & \qquad \qquad \qquad \qquad \qquad\qquad \qquad+ x  \nabla U_1(x/\norm{x})^{\transpose} (x)\defEns{\Id-x x^{\transpose}/\norm[2]{x}} \Big]\\
                  & \qquad +  \norm[\alpha-2]{x} \defEns{\nabla U_1(x/\norm{x}) x^{\transpose} + x  \nabla U_1(x/\norm{x})^{\transpose}  +2 \nabla U_1(x/\norm{x})^{\transpose}x x x^{\transpose}} \\
    & \phantom{   \nabla U_2(x) =} \qquad + \norm[\alpha]{x}\defEns{\Id-x x^{\transpose}/\norm[2]{x}} \nabla^2 U_1(x/\norm{x})  \eqsp.
  \end{align*}
  Since the $U_1$ is assumed to be twice continuously differentiable, the proof is finished.

\subsection*{Proof of \Cref{thm:ThmRecuit}}

\label{sec:proof-simulated}
The semi-group $(P_{s,t})_{t \geq s \geq 0 }$ is associated with the family of generator $(\generator_{\beta(t)})_{t \geq 0}$ where for any $\beta >0$, $\generator_{\beta}$ is defined  for any  $f\in\mrc^1(\rset^d\times\msy)$ by
\begin{align}
\nonumber
&  \generator_{\beta} f(x,y) = \ps{y}{\nabla f(x,y)} + \beta (\ps{y}{\nabla U(x)})_+ \defEns{f(x,\Refl(x,y)) - f(x,y)} \\
    \label{eq:generator-recuit}
& \qquad   + \rate \defEns{\int_{\msy} f(x,w) \rmd \loiy(w)  - f(x,y)} \eqsp.
\end{align}
% Then, $\generator_{\beta} $ is the generator of a BPS process with refreshment law $\loiy$, refreshment rate $\rate$ and potential $U_\beta = \beta U$.  Consider $(\beta_t)_{t\geqslant0}$ a cooling schedule, which is to say a function $t\rightarrow \beta_t \geqslant 1$ which is increasing and goes to infinity at infinity. Then we can consider the inhomogeneous Markov process $(Z_t)_{t\geqslant 0}$ on $\rset^d\times\msy$ with generator $(\generator_{\beta_t})_{t\geqslant 0}$ and semi-group $( P_{s,t})_{0\leqslant s \leqslant t}$, such that for all $f\in L^\infty(\R^{2d})$ and all $0\leqslant s \leqslant t$,
% \begin{equation*}
%  P_{s,t} f(z) & =& \mathbb E( f(Z_t)\ |\ Z_s = z)
% \end{equation*}
% and for all $f\in \mathsf C^1(\rset^d\times\msy)$ and all $t\geqslant 0$,
% \[(\partial_s)_{|s=0}\co P_{t,t+s} f(z) \cf = \generator_{\beta_t}  f(z).\]%\hspace{15pt}\text{and} \hspace{15pt} \partial_s P_{s,t}^{(\beta)} f(z) = -L^{\beta_s} P_{s,t}^{(\beta)} f(z).\]
% The construction of the process is exactly the same as the homogeneous one presented in \Cref{sec:main-results}, except that \eqref{eq:def-temps-bounce} is now replaced by:
% \begin{equation*}
% T_{n+1}^{(2,\beta)} &=& \inf\ensemble{t \geq 0}{\int_0^t \beta_{\S^{(\beta)}_n+ s} \ps{Y^{(\beta)}_{\S^{(\beta)}_n}}{\nabla U (X^{(\beta)}_{\S^{(\beta)}_n}+s Y^{(\beta)}_{\S^{(\beta)}_n})}_+  \rmd s \geq E^{2}_{n+1}}.
% \end{equation*}
First, we establish a Foster-Lyapunov drift condition for $\generator_{\beta}$ uniformly on $\beta\geqslant1$.
\begin{lemma}
\label{lem:lyap-recuit}
Assume \Cref{ass:recuit}. There exist $A_1,A_2,A_3>0$, $\beta_*\geqslant 1$ and $V_1,V_2\in \mrc^1(\rset^d\times\msy)$, with $V_i \exp(-U/2)$ bounded above and below by positive constants for $i=1,2$, such that for all $\beta\geqslant \beta_*$,
\begin{equation*}
\generator_{\beta} V_1  \leqslant  A_1( A_2 - V_1) \eqsp,
\end{equation*}
and for all $\beta \geqslant 1$,
\begin{equation*}
\generator_{\beta} V_2  \leqslant  A_3  V_2  \eqsp.
\end{equation*}
\end{lemma}
\begin{proof}
  We check that \Cref{ass:geo_erg_general}  holds for $\beta$ large enough, with $\bU = U/2$ and the potential  $x \mapsto  U_{\beta}(x)$. Indeed, set  $\ell(x) = 1$ for all $x\in\rset^d$ and $H(t) = t^2$ for $t\in\rset$. Then all the conditions of \Cref{ass:geo_erg_general}  are clearly satisfied, with $c_1$, $c_2$ and $c_4$ which does not depend on $\beta$, and $c_3 = \beta$. Let $\beta_*$ be large enough so that \eqref{eq:condition_geo_4} holds for $\beta \geqslant \beta_*$ and $\kappa$ defined in \eqref{eq:def_kappa} is equal to $1$.

Let $\lyap_1$ be the function defined by \eqref{eq:def_lyap_gene}. According to \Cref     {lem:lyapunov}, there exist $A_1,A_2>0$ such that
     \begin{equation*}
\generator_{\beta_*} V_1  \leqslant  A_1( A_2 - V_1).
\end{equation*}
Now, for $\beta \geqslant \beta_*$, keeping the notations of \Cref{sec:Lyapunov},
\[( \generator_{\beta} - \generator_{\beta_*}) V_1(x,y)  \ = \ (\beta-\beta_*)e^{ U(x)/2}\ps{y}{\na U(x)}_+( \varphi(-\theta) - \varphi(\theta) ) \ \leqslant \ 0\eqsp.\]

Second, set for any $(x,y)\in \rset^d\times \msy$, $V_2(x,y) = \exp(U(x) /2)   \varphi_2( \ps{y}{ \na U(x)}  )$, where $\varphi_2\in\mrc^1(\rset)$ is an increasing function such that $\varphi(s) = 1$ for $s\leqslant -1$ and $\varphi(s)=3$ for $s\geqslant 1$. Then, for all $\beta\geqslant 1$, 
\begin{align*}
e^{-  U(x)/2}\generator_{\beta} V_2(x,y) & \leqslant    \ps{y}{ \na U(x)} \varphi_2( \ps{y}{ \na U(x)}  ) + M^2 \norm{\na^2 U}_\infty  \norm{\varphi_2'}_\infty + 2\rate \\
&  + \beta \ps{y}{\na U(x)}_+\defEns{ \varphi( - \ps{y}{ \na U(x)}) - \varphi( \ps{y}{ \na U(x)} ) } \\
& \leqslant  3 + M^2 \norm{\na^2 U}_\infty  \norm{\varphi_2'}_\infty + 2\rate\eqsp,
%\frac12 \ps{y}{ \na U(x)} + M^2 \norm{\na^2 U}_\infty + \rate + \ps{y}{\na U(x)}_+  \varphi\co - \ps{y}{ \na U(x)} \cf
\end{align*}
and we conclude by noting that $\exp(U(x)/2)\leqslant V_2(x,y)$ for any $(x,y) \in \rset^d \times \msy$.
%When $\ps{y}{ \na U(x)} \geqslant 2$, $\varphi\co - \ps{y}{ \na U(x)} \cf =  -1$, so that for all $(x,y)\in \rset^d\times \msy$,
%\[\generator_{\beta} V_2(x,y)\ \leqslant\  ( 1 + M^2 \norm{\na^2 U}_\infty + \rate ) e^{ U(x)} \ \leqslant\  ( 1 + M^2 \norm{\na^2 U}_\infty + \rate )  V_2(x,y)\eqsp.\]
\end{proof}

\begin{corollary}
\label{corollary-recuit}
Assume that the assumptions of \Cref{thm:ThmRecuit} hold. Then there exists $A_4>0$ such that for all $t,s\geqslant 0$ and $(x,y) \in \rset\times \msy$,
\begin{equation*}
P_{t,t+s} V_1(x,y)  \leqslant  A_4 e^{A_3 s}V_1(x,y)\eqsp,
\end{equation*}
and for all $t\geqslant 0$ such that $\beta(t)\geqslant \beta_*$,
\begin{equation*}
P_{t,t+s} V_1(x,y)  \leqslant  e^{-A_1 s} V_1(x,y) + ( 1 - e^{-A_1 s} ) A_2\eqsp,
\end{equation*}
where $V_1 \in \rmc^1(\rset^d \times \msy)$, $A_1,A_2,A_3$ are given by \Cref{lem:lyap-recuit}. 
\end{corollary}
\begin{proof}
The proof follows the same line as  the proof of \Cref{Corollary-lyap}, using \Cref{lem:lyap-recuit} and $V_1/V_2$ is bounded above and below by positive constants.
\end{proof}

\begin{lemma}
\label{lem:couplage-recuit}
Assume that the assumptions of \Cref{thm:ThmRecuit} hold. Then, for all compact set $\mathsf K$ of $\R^{d}\times \msy$, there exist $s_1,\chi,A_5>0$  which depend on $\mathsf K$, $\loiy$, $\rate$ and $U$ but not on $t \mapsto \beta(t)$, such that for all  $(x,y),(\tx,\ty)\in\mathsf K$, all $t\geqslant 0$ and all $s\geqslant s_1$, %denoting by $\delta_{z}$ the Dirac mass at $z$, 
\begin{equation*}%\label{EqCouplageCompact}
 \tvnorm{ P_{t,t+s}((x,y),\cdot) -  P_{t,t+s} ((\tx,\ty),\cdot)} \leq 2\parentheseDeux{1 - \chi \exp \parenthese{ - A_5 \int_{t}^{t+s} \beta(u) \rmd u }} \eqsp. 
\end{equation*}
\end{lemma}
\begin{proof}
The arguments are exactly those of the proof of \Cref{lem:condition-doeblin-precis}, hence of \cite[Theorem 5.1]{MonmarcheRTP}, so that we only give a sketch of proof. First, considering the case $\beta=0$, we have already shown in \Cref{sec:couplage} that, starting from two different points in a given compact $\mathsf K$, it is possible to merge two processes in a time $s_1>0$ while staying in a compact $\mathsf K'$, with some probability $\chi>0$. Call $\mathsf E$ this event. Then, considering the case $\beta>0$, we follow the same coupling up to the first bounce time. The processes have merged if this first bounce happens after time $s_1$, which occurs with probability 
\begin{multline*}
\probaCond{ \int_t^{t+s_1} \beta(u) \ps{Y^{(\beta)}_{u}}{\nabla U (X^{(\beta)}_{u})}_+  \rmd u \geqslant E^{2}_{1} } {\mse}  \\ \geq  \exp \parenthese{ - M\norm{\nabla U}_{\infty, K'}\int_{t}^{t+s_1} \beta(u) \rmd u }\eqsp,
\end{multline*}
where $M = \sup_{(w,z) \in \msk'} \norm{z}$.
\end{proof}

Let $V_1$ and $A_i$, $i\in \{1,\ldots,4\}$, be given by Corollary \ref{corollary-recuit}. Then, let $s_1,\chi,A_5>0$ be given by \Cref{lem:couplage-recuit}, with $\mathsf K = \{(x,y)\in\rset^d\times\msy,\ V_1(x) \leqslant 2 A_2\}$. Let
\begin{equation}
  \label{eq:def_t_0}
  t_0 = \inf\{t\geqslant 0,\ \beta(t) \geqslant \beta_* \} \eqsp,
\end{equation}
and for $t\geqslant t_0$, define
\begin{equation}
  \label{eq:def_bfn}
 \bfn(t)= \lfloor (t-t_0)/s_1\rfloor\eqsp.
\end{equation}
Consider the following decomposition,
\begin{equation*}
P_{0,t}   = P_{0,t-\bfn(t) s_1}Q_0Q_1   \cdots  Q_{\bfn(t)-1} Q_{\bfn(t)} \eqsp, 
\end{equation*}
where $Q_0$ is the identity kernel and for $k\in \{ 1 ,\ldots, n(t)\}$, we set
\begin{equation}
  \label{eq:def_q_k}
  Q_k = P_{t- (\bfn(t)-k+1)s_1,t-(\bfn(t)-k)s_1} \eqsp.
\end{equation}
For any measurable function $\varphi : \rset^d \times \msy \to \rset$ and $\zeta \geqslant 0$, we set
\begin{equation*}
\|\varphi\|_{\zeta,V_1}  =\underset{(x,y)\in \rset^d \times \msy} \sup  \defEns  { \frac{|\varphi(x,y)|}{1+\zeta V_1(x)}} \eqsp,
\end{equation*}
and consider the weighted $V_1$-norm on $\mathcal P_{V_1}(\rset^d \times \msy)=\{\mu\in\mathcal P(\rset^d \times \msy) \, :\,  \mu(V_1) < \infty\}$, defined for $\mu_1, \mu_2 \in \mcp_{V_1}(\rset^d \times \msy)$ by
\begin{equation}
  \label{eq:def_rho_zeta}
\rho_\zeta(\mu_1,\mu_2)  =  \sup\left\{ \mu_1( \varphi) - \mu_2( \varphi) \, : \,  \|\varphi\|_{\zeta,V_1} \leqslant 1\right\}.
\end{equation}
Note that $\rho_\zeta(\mu_1,\mu_2) $ increases with $\zeta$ and that $\rho_0 = \tvnorm{\cdot}$. In addition, for any $\mu_1,\mu_2 \in  \mcp_{V_1}(\rset^d \times \msy)$,
\begin{equation*}
  \rho_\zeta(\mu_1,\mu_2) \leq \Vnorm[V_1]{\mu_1-\mu_2} \leq (1+\zeta)^{-1} \rho_\zeta(\mu_1,\mu_2) \eqsp.
\end{equation*}
\begin{lemma}
  \label{lem:convergence}
  Assume that the conditions of \Cref{thm:ThmRecuit} hold. Then for all $\nu_1,\nu_2\in\mcp_{V_1}(\R^{d} \times \msy)$, $t \geq t_0$ and  all $k\in \{1,\ldots,\bfn(t)\} $,
\begin{equation}\label{eq:recuit-contraction}
\rho_{\epsilon_k}( \nu_1 Q_k,\nu_2 Q_k )  \leqslant  \kappa_k  \rho_{\epsilon_k}( \nu_1 ,\nu_2 ) \eqsp,
\end{equation}
where 
\begin{align*}
   \epsilon_k  &=  \frac{  \chi  }{(1-\gamma)A_2}   \exp \parenthese{ - A_5 \int_{t-(n-k+1)s_1}^{t-(n-k)s_1} \beta_u \rmd u } \eqsp,\\
\kappa_k & =  1 - \parenthese{ \frac{\chi}2 \wedge \frac{1-\gamma}{4}} \exp \parenthese{ - A_5 \int_{t-(n-k+1)s_1}^{t-(n-k)s_1} \beta_u \rmd u }\eqsp,  \gamma &= \exp(-s_1 A_1) \eqsp.
\end{align*}
\end{lemma}
\begin{proof}
  It is a direct application  to $Q_k$ for all $k$ of   \Cref{thm:ThmCVexpo} based on \Cref{lem:couplage-recuit} and  \Cref{corollary-recuit}.
\end{proof}
 
For a fixed $\beta\geqslant 0$,  let $(P_t^{(\beta)})_{t\geq 0}$ be the semi-group of the BPS sampler associated with  the potential $x \mapsto \beta U(x)$ and, for $t\geqslant t_0$ and  $k\in\{0,\ldots,\bfn(t)\}$, let
\begin{equation}
  \label{eq:def_Q_'}
  Q'_k = P_{s_1}^{(\beta_{k})} \eqsp,
\end{equation}
where for ease of notation simplicity we denote
\begin{equation}
  \label{eq:def_beta_k}
  \beta_k = \beta_{t-(\bfn(t)-k)s_1} \eqsp.
\end{equation}
 In other words, $Q'_k$ is similar to $Q_k$ except that the inverse temperature is frozen. Let $\tpi_k$ be the invariant measure of $Q'_k$, namely 
\[\tpi_k  =   \pi_k\otimes \loiy\eqsp,\]
where $\pi_k$ admits a density with respect to the Lebesgue measure given for any $x \in \rset^d$ by
\[\pi_k( x) =  \rmZ_k^{-1} \exp (-\beta_{k} U(x)) \rmd x  \eqsp, \qquad \rmZ_k = \int_{\rset^d}\exp (-\beta_{k} U(\tx)) \rmd \tx\eqsp.\]

We know that the mass of $\pi_k$ concentrates, as $k\rightarrow \infty$, around the vicinity of the global minima of $U$. To get the same with $ P_{0,t}((x,y),\cdot)$, we need to show that $\tvnorm{\tpi_{\bfn(t) }- P_{0,t}((x,y),\cdot)}$ vanishes as $t\rightarrow\infty$. Denoting, for $t\geqslant t_0$ and $k\in\{0,\ldots,\bfn(t)\}$, $\nu_{k} = \updelta_{(x,y)} P_{0,t-\bfn(t)s_1}Q_0Q_1   \cdots  Q_{k-1} Q_{k}$, where $Q_k$ is defined in \eqref{eq:def_q_k}, it is then natural to study
\begin{equation}
  \label{eq:def_an}
  u_k = \rho_{\epsilon_k}( \nu_k,\tpi_k)\eqsp.
\end{equation}
From \eqref{eq:recuit-contraction}, for any $t \geq t_0$, $k \in \{1,\ldots,\bfn(t)\}$
\begin{eqnarray}\label{eq:recuit-an-en}
u_k  \leqslant  \kappa_k  \rho_{\epsilon_k}( \nu_{k-1},\tpi_{k-1} ) + \rho_{\epsilon_k}( \tpi_{k-1} Q_k,\tpi_{k} ) \ \leqslant \ \kappa_k u_{k-1} + e_k 
\end{eqnarray}
where 
\begin{equation*}
e_k = \rho_{\epsilon_k}( \tpi_{k-1} Q_k,\tpi_{k} )
\end{equation*}
and used that $\rho_{\epsilon_k}( \nu_{k-1},\tpi_{k-1} )  \leqslant \rho_{\epsilon_{k-1}}( \nu_{k-1},\tpi_{k-1} ) $ since  $(\epsilon_k)_{k \geq 0}$ is non-increasing.

%
%First, without assuming anything concerning the $\beta_n$'s, note that the remarks that led to Proposition \ref{PropMetastable} still applies here and, together with Theorem \ref{ThmCVexpo}, yield:
%
%\begin{lem}\label{LemRecuit1}
%Let $\gamma = \rme^{-\frac{rt_1}{2}}$, and $V,C$ be given in Lemma \ref{LemLyapunov}. Then, there exist $\bar \chi,\theta>0$ such that, denoting
%\begin{equation*}
%\epsilon_n  = \frac{\bar \chi \rme^{-\theta\beta_{n}}}{(1-\gamma)C} \wedge 1\\
%\kappa_n  =  1 - \frac{\bar \chi \rme^{-\theta\beta_{n}}}2,
%\end{equation*}
%and considering the distance $\rho_n=\rho_{\epsilon_n}$ on $\mathsf P_V(\R^{2d})$ such as defined in Section \ref{SectionDemoThmMain}, then we have, for all $\nu_1,\nu_2\in\mathsf P_V(\R^{2d})$ and for all $n\in\mathbb N$,
%\begin{equation*}
%\rho_{n}( \nu_1 Q_n,\nu_2 Q_n )  \leqslant  \kappa_n  \rho_{n}( \nu_1 Q_n,\nu_2 Q_n ).
%\end{equation*}
%\end{lem}
%
%We will also make use of the following:
\begin{lemma}\label{lem:recuit-en}
Assume that the conditions of \Cref{thm:ThmRecuit} hold. Then, there exists $A_6>0$ such that for all $t\geqslant t_0$, all $k \in \{ 1,\ldots, \bfn(t)\}$ and $l\geqslant 1$, there exists $A_l>0$ such that
\begin{equation*}
e_k  \leqslant  A_l (  \sqrt{\beta_{k} - \beta_{k-1}} + \beta_k - \beta_{k-1} )  + A_6  \rme^{-\frac12( \beta_{k-1}-1) l}\eqsp,
\end{equation*}
where $\beta_k$, $\bfn$ and $t_0$ are defined by \eqref{eq:def_beta_k}, \eqref{eq:def_bfn} and \eqref{eq:def_t_0} respectively.
\end{lemma}
\begin{proof}
Let $t \geq t_0$, $k\in \{1,\ldots,\bfn(t)\}$ and $l \geq 1$.
  In the proof, $C$ stands for a constant which may change from line to line but does not depend on $k$, $l$ and $\beta$.
We bound 
\begin{eqnarray}\label{eq:borne-en-recuit}
e_{k}  \leqslant  \rho_{\epsilon_k}( \tpi_{k-1} ,\tpi_{k} ) + \rho_{\epsilon_k}( \tpi_{k-1} Q_{k},\tpi_{k-1} )
\end{eqnarray}
and deal with each terms of the right hand side  apart.
% Writing $\rmZ_{k} = \int \rme^{-\beta_{k} U(w)}\dd w$ and $\mu_{k}$ the Lebesgue density of $\mu_{k}$, from \eqref{EqBorneV},
Indeed, for the first one, the first marginal of $\tpi_{k-1}$ and $\tpi_{k}$ having an explicit density, and their second marginal being equal, we bound
\begin{align}
  \nonumber
 \rho_{\epsilon_k}( \tpi_{k-1} ,\tpi_{k} ) & =  \int_{\rset^d\times\msy} ( 1 + \epsilon_{k} V_1(x,y)) |\pi_{k}(x) - \pi_{k-1}(x)| \rmd x \loiy(\rmd y)\\
  \nonumber
                                           & \leqslant   C\int_{\rset^d} \rme^{U(x)/2} \left|\frac{\rme^{-\beta_{k} U(x)}}{\rmZ_{k}} - \frac{\rme^{-\beta_{k-1} U(x)}}{\rmZ_{k-1}} \right|\dd x\\
    \nonumber
                                           & \leqslant  C \rme^{ l/2}\int_{\rset^d} \left|\frac{\rme^{-\beta_{k} U(x)}}{\rmZ_{k}} - \frac{\rme^{-\beta_{k-1} U(x)}}{\rmZ_{k-1}} \right|\dd x\\
  \label{eq:recuit_1}
&   + C \int_{\{U>l\}} \frac{\rme^{-(\beta_{k}-\frac12) U(x)}}{\rmZ_{k}} + \frac{\rme^{-( \beta_{k-1}-\frac12) U(x)}}{\rmZ_{k-1}} \dd x\eqsp.
\end{align}
We treat the two terms in the right-hand-side apart. 
The first term is the  total variation distance between  $\pi_{k}$ and $\pi_{k-1}$. Since $\beta_{k-1} \leq \beta_k$ since $\beta$ is non-decreasing, $\rmZ_{k-1} \geqslant \rmZ_{k}$. Using  Pinsker's inequality and this result, we get
\begin{align}
  \nonumber
  \parenthese{ \int_{\rset^d} \left|\frac{\rme^{-\beta_{k} U(x)}}{\rmZ_{k}} - \frac{\rme^{-\beta_{k-1} U(x)}}{\rmZ_{k-1}} \right|\dd x }^2 &\leqslant   2
                                                                                                                                             \int_{\rset^d} \ln\parenthese{ \frac{\rme^{-\beta_{k-1} U(x)}\rmZ_{k}}{\rme^{-\beta_{k} U(x)}\rmZ_{k-1}}  } \pi_{k-1}(x) \rmd x \\
    \nonumber
                                                                                                                                           &\leqslant  2 ( \beta_{k} - \beta_{k-1}) \int_{\rset^d} U(x) \frac{\rme^{-\beta_{k-1} U(x)}}{\rmZ_{k-1}} \dd x\\
    \nonumber
& \leqslant   2  ( \beta_{k} - \beta_{k-1}) ( 1 + C \sqrt{\beta_{k-1}} \rme^{-\beta_{k-1}+1} ) \\
\label{eq:recuit_masse_0}
                                                                                                                                         & \leqslant   C  ( \beta_{k} - \beta_{k-1})\eqsp,
\end{align}
where we used for the two last inequalities that 
\[\int_{\{U>1\}} U(x) \rme^{-U(x)}  \dd x \ \leqslant \ 2 \int_{\rset^d}   \rme^{-U(x)/2}  \dd x\ < \ \infty\]
and since $U(0) = 0$, $U(x) \leq \norm{\nabla^2 U}_{\infty} \norm{x}^2$, for any $x \in \rset^d$ by \Cref{ass:recuit}, 
\[\rmZ_{k-1} \ \geqslant \ \int_{\rset^d} \rme^{- \beta_{k-1} \|\na^2 U\|_\infty \norm[2]{x}} \rmd x \ \geqslant\ C \beta_{k-1}^{-d/2} >0 \eqsp.\]
%where we used that $\min U = 0$ and Laplace's method (possibly at a higher level, if all the global minima of $U$ have degenerated Hessian). % (possibly bounding $U(x)\leqslant U(x) + m (x-x_0)^2=W(x)$ with $U(x_0)=0$, so that $W$ admits a unique global minimum).

Similarly, for the second term of \eqref{eq:recuit_1} we obtain %denoting by $vol$ the Lebesgue volume of a set,
\begin{equation*}
 \int_{\{U>l\}} \frac{\rme^{-(\beta_{k}-\frac12) U(x)}}{\rmZ_{k}}  \dd x \leqslant   C \beta_k^{d/2}\rme^{-(\beta_{k}-1)l} \int_{\rset^d} \rme^{- U(x)/2} \rmd x\eqsp.
 %\frac{\int_{U>l}\rme^{- ( \frac{\beta_{k}}{2}-\frac14) U(x)}\dd x }{vol(U<l/2)}\\
 %  \leqslant  \frac{\rme^{-   \frac{\beta_{k} l}{8} }\int \rme^{-  \frac{\rho}{16}|x|^2}\dd x }{vol(U<l_0/2)}.
\end{equation*}
Using that for any $t \geq 1$, $t^{d/2}\exp(-l(t-1)/2) \leq  (d/l)^{d/2}\exp(-(d-l)/2)$ if $d \geq \ell$ and $t^{d/2}\exp(-l(t-1)/2) \leq 1$ otherwise, there exists $A_{6,1}$ which does not depend on $l$ such that 
\begin{equation}
  \label{eq:recuit-masse}
 \int_{\{U>l\}} \frac{\rme^{-(\beta_{k}-\frac12) U(x)}}{\rmZ_{k}}  \dd x \leq A_{6,1} \rme^{-(\beta_{k}-1)l/2} \eqsp.
\end{equation}
Combining this bound and~\eqref{eq:recuit_masse_0} in \eqref{eq:recuit_1}, we get that there exists $A_{l,1} \geq 0$ such that  
\begin{equation}
  \label{eq:recuit_first_bound_main_lemma}
   \rho_{\epsilon_k}( \tpi_{k-1} ,\tpi_{k} ) \leq A_{l,1}\sqrt{\beta_k-\beta_{k-1}} + A_{6,1} \rme^{-(\beta_k-1)l/2} \eqsp.
\end{equation}

The second term of \eqref{eq:borne-en-recuit} is treated through a synchronous coupling similar to \cite[Proposition 11]{DurmusGuillinMonmarche:toolbox}. Indeed, $\tpi_{k-1}$ being invariant for $Q'_{k-1}$ defined in \eqref{eq:def_Q_'} and by \eqref{eq:def_rho_zeta},
\begin{align}
\nonumber
  \rho_{\epsilon_{k}}( \tpi_{k-1} Q_{k},\tpi_{k-1} )& = \rho_{\epsilon_{k}}( \tpi_{k-1} Q_{k},\tpi_{k-1}Q_{k-1}' )\\
  \label{eq:bound_rho_eps_k}
 & =  \underset{\norm{\varphi}_{\epsilon_{k},V_1}\leqslant 1}\sup\left\{ \mathbb E[ \varphi( X_{s_1},Y_{s_1}) - \varphi( X_{s_1}',Y_{s_1}')]\right\}\eqsp,
\end{align}
where $(X_t,Y_t)_{t\geqslant0}$ (resp. $(X_t',Y_t')_{t\geqslant0}$) is a BPS process with a fixed temperature $\beta_{k-1}$ (resp. a annealed BPS process with  cooling schedule $s \mapsto \beta(t-(\bfn(t)-k+1)s_1 +s)$) and $(X_0,Y_0) = (X_0',Y_0')$ is distributed according to  $\tpi_{k-1}$. Following \cite[Proposition 11]{DurmusGuillinMonmarche:toolbox}, we construct such processes in such a way $(X_t,Y_t) = (X_t',Y_t')$ up to time $T^{'}_{\rmb}$, where $T^{'}_{\rmb}$ is the first time $(X_t',Y_t')_{t \geq 0}$ bounces while $(X_t,Y_t)_{t \geq 0}$ does not, defined by
\begin{equation*}
T^{'}_{\rmb} =  \inf\left\{\tau \geqslant 0,\ E < \int_0^\tau ( \beta_{k} - \beta_{t-(\bfn(t)-n)s_1 + s}) \ps{Y^{(\beta)}_s}{\nabla U(X^{(\beta)}_s)}_+ \rmd s\right\}\eqsp,
\end{equation*}
where $E$ is a standard exponential random variable independent of $Z$.

Consider the compact sets $\mathsf K = \{(x,y)\in\rset^d\times\msy \, : \, U(x)<l\}$ and $\tmsk = \{(x,y) \in\rset^d\times\msy \, : \, \dist((x,y),\msk) \leq M s_1\}$, where $\dist(\cdot,\msk)$ is the distance from $\msk$ and $M = \sup_{z \in \msy} \norm{z}$. That way, if a BPS with refreshment law $\loiy$ over $\msy$ have an initial condition in $\mathsf K$, then on the time interval $[0,s_1]$ it necessarily stays in $\tmsk$.
 
Consider $\varphi:\rset^d\times \msy : \rightarrow\rset$ with $\norm{\varphi}_{\epsilon_{n+1},V_1}\leqslant 1$ and the following decomposition
\begin{align}
  \nonumber
  \mathbb E[ \varphi( X_{s_1},Y_{s_1}) - \varphi( X_{s_1}',&Y_{s_1}')] =  \mathbb E[\1_{\msk}(X_0,Y_0)\{ \varphi( X_{s_1},Y_{s_1}) - \varphi( X_{s_1}',Y_{s_1}')\}] \\ & +  \mathbb E[\1_{\rset^d \times \msy \setminus \msk}(X_0,Y_0)\{ \varphi( X_{s_1},Y_{s_1}) - \varphi( X_{s_1}',Y_{s_1}')\}] \eqsp.
\label{eq:decomp_recuit}
\end{align}
We bound the two terms separately. First, using that if $(X_0,Y_0) \in \msk$, then for any $t \in \ccint{0,s_1}$, $(X_t',Y_t') \in \tmsk$, we have 
 \begin{align}
\nonumber
      \mathbb E[ \1_{\msk}(X_0,Y_0) & \{\varphi( X_{s_1},Y_{s_1}) - \varphi( X_{s_1}',Y_{s_1}')\}]\\
\nonumber
   =   2 ( 1 + \epsilon_1)& \| V_1\|_{\infty, \tmsk} \proba{(X_0,Y_0) \in \msk, (X_{s_1},Y_{s_1}) \neq (X_{s_1}',Y_{s_1}')}\\
   \nonumber
   =   2 ( 1 + \epsilon_1) &\| V_1\|_{\infty, \tmsk} \proba{(X_0,Y_0) \in \msk, T'_{\rmb} < s_1} \\
\nonumber
    \leqslant   2 ( 1 + \epsilon_1) & \| V_1\|_{\infty, \tmsk}  \proba{ E < M \norm{\na U}_{\infty,\tmsk} \int_0^{s_1} ( \beta_{k} - \beta_{t-(\bfn(t)-n)s_1 +s})  \rmd s}\\
   \label{eq:first_bound_decomp_recuit}
 & \leqslant   2 ( 1 + \epsilon_1) \| V_1\|_{\infty, \tmsk}  M \norm{\na U}_{\infty,\tmsk} s_1 ( \beta_{k} - \beta_{k-1})   \eqsp,
 \end{align}
 where $M = \sup_{z \in \msy}\norm{z}$.
 % and
 % \begin{align*}
 % \mathbb P ( T^{'}_{\rmb} < s_1) & 
 % \end{align*}
Note that $\norm{\na U}_{\infty,\tmsk} $ depends on $\tmsk$, hence on $l$.
 
 Next using \Cref{lem:lyap-recuit} and the Markov property, we get
 \begin{align*}
  \mathbb E[\1_{\rset^d \times \msy \setminus \msk}(X_0,Y_0) & \{ \varphi( X_{s_1},Y_{s_1}) - \varphi( X_{s_1}',Y_{s_1}')\}] \\
 & \leqslant   ( 1 + \epsilon_1) \mathbb E[\1_{\rset^d \times \msy \setminus \msk}(X_0,Y_0)\{ V_1( X_{s_1}) + V_1( X_{s_1}')\}]\\
                                                             &  \leqslant  2 ( 1 + \epsilon_1) \parenthese{  \expe{\1_{\rset^d \times \msy \setminus \msk}(X_0,Y_0) V_1(X_0)} + A_2}\\
                                                             & \leq 2 ( 1 + \epsilon_1) \parenthese{ C \int_{U\geqslant l} \rme^{ U(x)/2}  \tpi_{k-1}(\rmd x)  + A_2} \\
%   \label{eq:second_bound_decomp_recuit}
   & \leq 2 ( 1 + \epsilon_1) \parenthese{ C \rme^{-(\beta_{k}-1)l/2} +A_2} \eqsp.
 \end{align*}
where we used for the penultimate inequality  that $(X_0,Y_0)$ is distributed according to $\tpi_{k-1}$,
Combining this result and  \eqref{eq:first_bound_decomp_recuit}  in \eqref{eq:decomp_recuit} and \eqref{eq:bound_rho_eps_k}, we get there exist $A_{6,2} \geq 0$ independent of $l$ and $A_{l,2} \geq 0$ satisfying 
\begin{equation*}
  %\label{eq:2}
  \rho_{\epsilon_{k}}( \tpi_{k-1} Q_{k},\tpi_{k-1} ) \leq A_{l,2}(\beta_{k}-\beta_{k-1}) + A_{6,2}   \rme^{-(\beta_{k}-1)l/2} \eqsp.
\end{equation*}
The proof is concluded combining this result and \eqref{eq:recuit_first_bound_main_lemma} in \eqref{eq:borne-en-recuit}. 
\end{proof}

%We can now state the main result of this section:
\begin{lemma}
  \label{lem:recuit-an}
Assume  \Cref{ass:recuit}. There exists $\theta >0$ such that if \Cref{ass:recuit-beta} holds with  $D_1 \leq \theta^{-1}$,
then
 there exists $A_7 >0$ satisfying for all $t\geqslant t_0$,  $k\leqslant \bfn(t)$ and $(x,y) \in \rset^d\times \msy$,
$
u_{k} \leq A_7V_1(x,y) /k^{q_1}$
where $u_{k}$ is given in \eqref{eq:def_an}  and  $q_1 = (1/2) (1 - \theta D_1)$.
\end{lemma}

\begin{proof}
  Let $l \geq 1$, $t \geq t_0$ and $k \in \{1,\ldots \bfn(t)\}$.
  In the proof, $C$ stands for a constant which may change from line to line but does not depend on $k$, $l$ and $\beta$.  Denoting $d_0=0$ and $d_{k} = \kappa_{k}  d_{k-1} + e_{k} $, \eqref{eq:recuit-an-en} reads
\begin{equation*}
u_{k} - d_{k}  \leqslant  \kappa_{k} ( u_{k-1} - d_{k-1})
\end{equation*}
and yields
\begin{equation}
  \label{eq:bound_a_k}
u_{k}  \leqslant    u_0 \prod_{j=1}^k \kappa_j +   \sum_{i=1}^k \defEns{ e_i \prod_{j=i}^{k-1} \kappa_j }\eqsp,
\end{equation}
with the convention that $\prod_{j=k}^{k-1} \kappa_j = 1$. From \Cref{lem:recuit-en} applied with $l=1/D_2$, and bounding
\[\beta_{k} - \beta_{k-1} \ \leqslant \ \frac{D_1\lceil s_1/s_0\rceil }{t - (\bfn(t)-k+1)s_1}\]
for $k$ large enough, we get
\begin{equation}
  \label{eq:bound_ek}
  e_{k} \leqslant C/\sqrt{k} \eqsp
\end{equation}
Let  $\theta = 2A_5 s_1$,  so that, by definition of $\kappa_k$ given in \Cref{lem:convergence} the condition $\beta$,  and using $1-s \leq \rme^{-s}$, we have 
\[\kappa_{k}\ \leqslant \ 1 - \parenthese{ \frac{\chi}2 \wedge \frac{1-\gamma}{4}}\exp ( - \theta \beta_{k}/2 ) \ \leqslant \ \exp( - C n^{-\theta D_1/2}) \]
 Hence, for $i\in\{1,\ldots,k\}$,
 \begin{equation*}
%   \label{eq:5}
   \prod_{j=i}^k \kappa_j \ \leqslant \exp\parenthese{ -C \sum_{j=i}^n j^{-\theta D_1/2}} \leqslant \exp\parenthese{ -(C/q)\{  (n+1)^{q} - i^{q}\}} \eqsp,
 \end{equation*}
with $q= 1 - \theta D_1/2 \in\ooint{1/2,1}$ by assumption. Thus, combining this result and \eqref{eq:bound_ek} in \eqref{eq:bound_a_k}, we get
\begin{equation*}
u_{k}  \leqslant    u_0 \rme^{-(C/q) (n^{q}-1)}+  C ( k^{-1/2} + I(k))  
\end{equation*}
with
\begin{align*}
I(k) & =  \rme^{-(C/q) k^q }\int_1^{k} \frac1{\sqrt{s}} \rme^{(C/q) s^q}\rmd s \\
& = \frac1{C}\parenthese{ \frac{1}{k^{q-\frac12}} - \rme^{-(C/q) (k^q -1)} }+ \frac{\rme^{-(C/q) k^q }}{C}( q - 1/2) \int_1^k s^{q-\frac32} \rme^{(C/q) s^q}\rmd s\\
&\leqslant   \frac{1}{C k^{q-\frac12}}  + \frac{\rme^{-(C/q) k^q }}{C} \int_1^{k_0} s^{q-\frac32} \rme^{(C/q) s^q}\rmd s + \frac{I(k)}{C k_0^{1-q}}
%& =  \frac{1}{q n^q} - \frac1{q}{\rme^{-n^q}} + \rme^{-n^q} \int_1^n \frac{1}{s^{1+q}} \rme^{s^q} \dd s\\
%& \leqslant  \frac{1}{q n^q}  +  \rme^{-n^q} \int_1^2 \frac{1}{s^{1+q}} \rme^{s^q} \dd s + \frac1{2^q} I,
\end{align*}
for all $k_0\geqslant 1$. In particular, for $k_0\geqslant (2/C)^{1/(1-q)}$, this means $I(k) \leqslant Ck^{1/2-q}$.

Finally, from the first part  \Cref{corollary-recuit}, $u_0 \leqslant C  V_1(x,y)$.
\end{proof}
% so that $I\leqslant c_5 n^{-q}$ for some $c_5>0$. As a consequence, \eqref{EqRecuitDemo} implies that
% \begin{equation*}
%\|\nu P_{0,t}^{(\beta)} - \mu_{\beta_t}\|_1 & \leqslant & \frac{c_6}{t^q}
%\end{equation*}
%for some $c_6>0$.
%
%In particular, if $\widetilde X^{(\beta)}_t$ is a random variable with law $\mu_{\beta_t}$ and $\mathsf A$ is a measurable set of $\R^d$, then 
%\begin{equation*}
%\mathbb P ( X^{(\beta)}_t \in \mathsf A ) & \leqslant & \mathbb P ( \widetilde X^{(\beta)}_t \in \mathsf A )  + \frac{c_6}{2 t^q}. 
%\end{equation*}
%Let $\mathsf A_0 = \{x\in\R^d,\   U(x) < \eta-\eta'\}$ and  $\mathsf A_1 = \{x\in\R^d,\ \eta < U(x) < \eta + \frac{\rho}{4}|x|^2\}$, which are a compact sets, and $\mathsf A_2 = \{x\in\R^d,\   U(x) > \eta+ \frac{\rho}{4}|x|^2\}$. Then
%\begin{equation*}
%\mathbb P ( \widetilde X^{(\beta)}_t \in \mathsf A )  & \leqslant & \frac{\rme^{-\beta_t \eta}( vol(\mathsf A_1) + \int_{\mathsf A_2} \rme^{-\frac{\rho}{4}|x|^2 }\dd x)}{\rme^{-\beta_t( \eta'-\eta)} vol(\mathsf A_0)}\\
%& \leqslant & \frac{c_7}{t^{D_2 \eta'}}
%\end{equation*}
%for some $c_7>0$, which concludes.

\begin{proof}[Proof of  \Cref{thm:ThmRecuit}]
Let $t>t_0$, $n= \bfn(t)$, $\eta>\eta'>0$.   In the proof, $C$ stands for a constant which may change from line to line but does not depend on $n$, $\eta,\eta',t$ and $\beta$. First,
\begin{equation*}
\mathbb P ( U(X^{(\beta)}_t) > \eta + \min U )   \leqslant  \int_{\{U\geqslant \eta\}} \tpi_{k}(\rmd x,\rmd y) + (1/2) \tvnorm{ P_{0,t}((x,y),\cdot) - \tpi_{k}}\eqsp.
\end{equation*}
  Similarly to \eqref{eq:recuit-masse},
\begin{equation*}
\int_{\{U\geqslant \eta\}} \tpi_{k}(\rmd x,\rmd y)    \leqslant   C \rme^{-\beta_{k} \eta'} \ \leqslant \  C t^{-D_2\eta'}\eqsp.
\end{equation*}
We conclude, with \Cref{lem:recuit-an} and the first part of \Cref{corollary-recuit}, by
\begin{equation*}
\tvnorm{\nu P_{0,t} - \tpi_{k}}  \leqslant   u_{k} \ \leqslant C V_1(x,y) /t^q \eqsp.
\end{equation*}
\end{proof}

   \section*{Appendix B. Quantitative contraction rates for Markov chains}
\label{app:quantitative_bounds}
In this section, we give for completeness a quantitative version of  \cite[Theorem 1.2]{HairerMattingly2011} which is used in \Cref{sec:metast-regime-anne}. Let $Q$ be a Markov operator on a smooth finite dimension manifold $\msm$ (in our applications $Q=P_{t_0}$ for some $t_0>0$, with $\msm= \R^{d} \times \msy$) and $V:\msm \to \coint{1,\plusinfty}$ (which can be thought as the one given by \eqref{eq:def_lyap_gene}). 
%If  $\mu_1$ and $\mu_2$ admit a density (still denoted $\mu_1$ and $\mu_2$) with respect to a measure $\dd x$ on $\msm$, then, equivalently,
%\begin{eqnarray*}
%\rho_\zeta(\mu_1,\mu_2) & = & \int ( 1+\zeta V(x)) |\mu_1(x)-\mu_2(x)|\dd x.
%\end{eqnarray*}

For any measurable function $\varphi : \msm \to \rset$ and $\zeta \geqslant 0$, we set
\begin{equation*}
\|\varphi\|_{\zeta,V}  =\underset{x\in \msm} \sup  \defEns  { \frac{|\varphi(x)|}{1+\zeta V(x)}} \eqsp,
\end{equation*}
and consider the weighted $V$-norm on $\mathcal P_{V}(\msm)=\{\mu\in\mathcal P(\msm) \, :\,  \mu(V) < \infty\}$, defined for $\mu_1, \mu_2 \in \mcp_{V}(\msm)$ by
\begin{equation}
  \label{eq:def_rho_zeta_D}
\rho_\zeta(\mu_1,\mu_2)  =  \sup\left\{ \mu_1( \varphi) - \mu_2( \varphi) \, : \,  \|\varphi\|_{\zeta,V} \leqslant 1\right\}.
\end{equation}

\begin{theorem}\label{thm:ThmCVexpo}
Suppose that there exist $\alpha,\gamma\in(0,1)$, $C_1>0$ and $C_2 > 2C_1$ such that for all $x,y,z \in \msm$, $V(x)+ V(y) \leq C_2$, 
\begin{equation*}
 \tvnorm{ Q(x,\cdot) -  Q(y,\cdot)} \leqslant 2( 1 - \alpha) \eqsp, \qquad QV(z) \leqslant  \gamma V(z) + C_1(1-\gamma)\eqsp.
\end{equation*} 
Then there exists $\zeta>0$ and $\kappa\in \ooint{0,1}$ such that for all $\mu_1,\mu_2\in\mathcal P_V( \msm)$,
\begin{equation*}
\rho_\zeta(\mu_1 Q,\mu_2 Q) \leqslant  \kappa \rho_\zeta(\mu_1,\mu_2) \eqsp,
\end{equation*}
where $\rho_{\zeta}$ is defined by \eqref{eq:def_rho_zeta_D}.
More precisely, if $C_2 = 4C_1$, then this holds with
%Consider some $\alpha_0\in(0,\alpha)$ and set
\begin{align*}
\zeta  = \alpha((1-\gamma)C_1)^{-1} \eqsp, \qquad \qquad \kappa  =  ( 1 - \alpha/ 2) \vee ( (3+\gamma)/4) \eqsp.
\end{align*}
\end{theorem}

\begin{proof}
 \cite[Lemma 2.1]{HairerMattingly2011} shows that 
\begin{equation*}
\rho_\zeta(\mu_1,\mu_2)  =  \sup\left\{ \mu_1( \varphi) - \mu_2 (\varphi) \, :\,   |\varphi |_{\zeta} \leqslant 1\right\} \eqsp,
\end{equation*}
where $\rho_\zeta$ is defined by \eqref{eq:def_rho_zeta} and
\begin{equation*}
 |\varphi |_{\zeta}  =  \underset{x\neq y}\sup\  \defEns{\frac{|\varphi(x) - \varphi(y)|}{2+ \zeta V(x) + \zeta V(y)} } =  \underset{c\in\R }\inf \| \varphi +c\|_{\zeta,V}.
\end{equation*}
Let $\varphi$ be a measurable function  such that $|\varphi |_{\zeta} = \| \varphi  \|_{\zeta,V} = 1$. We aim to show that $|Q \varphi |_{\zeta}\leqslant \kappa$ or, in other words, that
\begin{equation*}
|Q\varphi(x) - Q\varphi(y)|  \leqslant  \kappa ( 2 + \zeta V(x) + \zeta V(y)) 
\end{equation*}
for all $x,y\in \msm$.

First, consider the case where $V(x)+V(y) \geqslant C_2$. For $\zeta>0$, set
%Consider $\alpha_0 \in (0,\alpha)$, and set $\zeta = \frac{2(\alpha-\alpha_0)}{K}$ and  
 $\kappa_1 = \gamma + (1-\gamma)\frac{1+\zeta C_1}{1+ \zeta C_2/2}$. Note that $\gamma < \kappa_1 <1$, and 
%\[ \gamma + \frac{2K}{R} \ < \ \frac{2+\zeta R( \gamma + \frac{2K}{R})}{2+\zeta R} \ = \ \kappa_1 \ = \ 1 + \zeta \frac{R(\gamma-1)+ 2 K}{2+\zeta R}\ < \ 1.\]
%and that 
\[ 2(1-\kappa_1) + (\gamma - \kappa_1)\zeta C_2 + 2\zeta C_1(1-\gamma) \leqslant 0 \eqsp.\]
Hence,
\begin{align*}
&|Q\varphi(x) - Q\varphi(y)|  \leqslant   2 + \zeta QV(x) + \zeta QV(y)  \leqslant  2 + \zeta \gamma V(x) + \zeta \gamma V(y) + 2\zeta C_1  \\
& \leqslant  \kappa_1 ( 2 + \zeta  V(x) + \zeta  V(y)) + 2(1-\kappa_1) + (\gamma - \kappa_1)\zeta ( V(x) + V(y)) + 2\zeta C_1\\
& \leqslant  \kappa_1 ( 2 + \zeta  V(x) + \zeta  V(y)) \eqsp.
\end{align*}

Second, consider the case where $V(x)+V(y) \leqslant C_2$. Let $( Z_x, Z_y)$ be an optimal coupling of $Q(x,\cdot)$ and $Q(y,\cdot)$. % , in the sense that $Q\varphi(z) = \mathbb E[ \varphi(Z_z)]$, $z= x$ or $y$, and that
% \begin{equation*}
% \mathbb P( Z_x \neq Z_y )  =  (1/2)\frac12 \tvnorm{Q(x,\cdot) - Q(y,\cdot)}.
% \end{equation*}
Then, writing $\kappa_2 = ( 1 - \alpha +  \zeta C_1(1-\gamma)/2)  \vee \gamma$ (which is smaller than 1 for $\zeta$ small enough),
\begin{align*}
  |Q\varphi(x) - Q\varphi(y)| & \leqslant    \mathbb P( Z_x \neq Z_y ) \expe{ |\varphi(Z_x) - \varphi(Z_y)| |\ Z_x \neq Z_y}\\
& \leqslant  \mathbb P( Z_x \neq Z_y ) ( 2 + \zeta \, \expe{ V(Z_x) + V(Z_y) }) \\
& \leqslant  2 (1-\alpha) + \zeta\gamma ( V(x) + V(y)) + \zeta C_1 (1-\gamma)\\
& \leqslant   \kappa_2 ( 2 + \zeta V(x) + \zeta V(y) ) \eqsp,
\end{align*}
which concludes the general proof.

For $C_2=4C_1$, we chose $\zeta = \frac{\alpha}{(1-\gamma)C_1}$, so that $\kappa_2 = 1 - \alpha/2$ and
\[\kappa_1 \ = \ \gamma + (1-\gamma)\parenthese{ 1 - \frac{\alpha}{1-\gamma + 2\alpha}} \ = \ 1 - \frac{\alpha (1-\gamma)}{1-\gamma + 2\alpha} \eqsp. \]
Using that, for $a,b>0$,
\[ \frac{ab}{a+b} \ = \ \frac{a\wedge b}{1 + \frac{a\wedge b}{a\vee b}} \ > \ \frac{a\wedge b}2 \eqsp,\]
we get
\[\kappa_1 \ \leqslant \  1 - \frac{(2\alpha)\wedge (1-\gamma)}{4} =  ( 1 - \alpha/ 2) \vee ( (3+\gamma)/4)  \eqsp . \]
\end{proof}

Remark that, under the same assumptions that \Cref{thm:ThmCVexpo} but with $\alpha=0$, the same proof yields, for all $\zeta>0$ and all $\mu_1,\mu_2\in\mathcal P_V( \msm)$,
\begin{equation}
\label{eq:rmq-rho-beta}
\rho_\zeta(\mu_1 Q,\mu_2 Q)  \leqslant  (1+\zeta C_1) \rho_\zeta(\mu_1,\mu_2) \eqsp.
\end{equation}

\end{document}